\documentclass{amsart}
\pdfoutput=1
\usepackage{amsfonts}
\usepackage{enumitem}
\usepackage{amsmath,amssymb,amsfonts,amsthm}
\usepackage{latexsym}
\usepackage{amscd}
\usepackage[all,cmtip]{xy}
\usepackage{mathrsfs}
\usepackage{braket}
\usepackage{verbatim}
\usepackage{cancel}
\allowdisplaybreaks

\numberwithin{equation}{section}
\newtheorem{proposition}{Proposition}[section]
\newtheorem{theorem}[proposition]{Theorem}
\newtheorem{lemma}[proposition]{Lemma}
\newtheorem{corollary}[proposition]{Corollary}
\newtheorem{remarks}[proposition]{Remark}
\newtheorem{claim}[proposition]{Claim}
\newtheorem{exam}[proposition]{Example}

\newtheorem{innercustomthm}{Theorem}
\newenvironment{customthm}[1]
 {\renewcommand\theinnercustomthm{#1}\innercustomthm}
 {\endinnercustomthm}

\newcommand{\bC}{\mathbf{C}}
\newcommand{\bD}{\mathbf{D}}

\newcommand{\bF}{\mathbf{F}}
\newcommand{\bH}{\mathbf{H}}
\newcommand{\bP}{\mathbf{P}}

\newcommand{\bL}{\mathbf{L}}
\newcommand{\bM}{\mathbf{M}}
\newcommand{\bQ}{\mathbf{Q}}

\newcommand{\bp}{\mathbf{p}}
\newcommand{\CC}{\mathbf{C}}

\newcommand{\NN}{\mathbb{N}}

\newcommand{\RR}{\mathbf{R}}

\newcommand{\cA}{\mathcal A}
\newcommand{\cB}{\mathcal B}
\newcommand{\cC}{\mathcal C}
\newcommand{\cD}{\mathcal D}
\newcommand{\cE}{\mathcal E}

\newcommand{\cG}{\mathcal G}
\newcommand{\cH}{\mathcal H}
\newcommand{\cI}{\mathcal I}
\newcommand{\cJ}{\mathcal J}
\newcommand{\cK}{\mathcal K}
\newcommand{\cL}{\mathcal L}
\newcommand{\cM}{\mathcal M}

\newcommand{\cO}{\mathcal O}
\newcommand{\cP}{\mathcal P}
\newcommand{\cQ}{\mathcal Q}
\newcommand{\cR}{\mathcal R}
\newcommand{\cS}{\mathcal S}
\newcommand{\cT}{\mathcal T}
\newcommand{\cU}{\mathcal U}
\newcommand{\cV}{\mathcal V}
\newcommand{\cW}{\mathcal W}

\newcommand{\cY}{\mathcal Y}
\newcommand{\cZ}{\mathcal Z}

\def\fB{\mathfrak{B}}

\newcommand{\dist}{\mathrm{dist}}
\newcommand{\beq}{\begin{equation}}
\newcommand{\eeq}{\end{equation}}
\newcommand{\beqa}{\begin{eqnarray*}}
\newcommand{\eeqa}{\end{eqnarray*}}
\newcommand{\eps}{\epsilon}

\newcommand{\pa}[2]{\frac{\partial #1}{\partial #2}}
\def\wl{\par \vspace{\baselineskip}}

\newcommand{\res}{\mathbin{\vrule height 1.6ex depth 0pt width
0.13ex\vrule height 0.13ex depth 0pt width 1.3ex}}

\DeclareMathOperator{\tr}{tr}

\DeclareMathOperator{\divr}{div}
\DeclareMathOperator{\diam}{diam}

\DeclareMathOperator{\gr}{graph}

\DeclareMathOperator{\Tan}{Tan}

\DeclareMathOperator{\spt}{spt}
\DeclareMathOperator{\reg}{reg}
\DeclareMathOperator{\sing}{sing}

\DeclareMathOperator{\Var}{Var}
\address{Massachusetts Institute of Technology\\
Department of Mathematics\\
77 Massachusetts Avenue\\
Cambridge, MA 02139-4307 \\}

\email{stbeckerkahn@mit.edu}
\author{Spencer T. Becker-Kahn}

\title[Singularities of Minimal Two-Valued Graphs]{Transverse Singularities of Minimal Two-Valued Graphs in Arbitrary Codimension}
\thanks{Partially supported by the EPSRC grant EP/H023348/1 for the University of Cambridge's Centre for Doctoral Training: the Cambridge Centre for Analysis. Partially completed while the author was a PhD student under the supervision of Neshan Wickramasekera. The author was previously known as Spencer Hughes}

\begin{document}

\begin{abstract}
We prove some epsilon regularity results for $n$-dimensional minimal two-valued Lipschitz graphs. The main theorems imply uniqueness of tangent cones and regularity of the singular set in a neighbourhood of any point at which at least one tangent cone is equal to a pair of transversely intersecting multiplicity one $n$-dimensional planes, and in a neighbourhood of any point at which at which at least one tangent cone is equal to a union of four distinct multiplicity one $n$-dimensional half-planes that meet along an $(n-1)$-dimensional axis. The key ingredient is a new Excess Improvement Lemma obtained via a blow-up method (inspired by the work of L. Simon on the singularities of `multiplicity one' classes of minimal submanifolds) and which can be iterated unconditionally. We also show that any tangent cone to an $n$-dimensional minimal two-valued Lipschitz graph that is translation invariant along an $(n-1)$ or $(n-2)$-dimensional subspace is indeed a cone of one of the two aforementioned forms, which yields a global decomposition result for the singular set. 
\end{abstract}

\maketitle
\setcounter{tocdepth}{1}
\tableofcontents
\wl

There are very few results about the nature of the singular set of a minimal submanifold in arbitrary codimension. Allard's seminal work in this area (\cite{allard}) shows only that the singular set is closed and nowhere dense. In light of the simple example of a transverse union of hyperplanes, the optimal dimension estimate for the singular set (of an $n$-dimensional stationary integral varifold) would be $\dim_{\cH}(\sing V) \leq n-1$, but the possibility of a singular set with positive $\cH^n$ measure, \emph{i.e.} a `fat Cantor set'-like singular set, has not been ruled out. Despite this, sharp dimension estimates for the singular set have been obtained in various special cases, most notably for area-minimizing surfaces (\emph{i.e.} integral currents) in the celebrated work of Almgren (\cite{almgrenbig}). The same result has recently been obtained by De Lellis and Spadaro in a series of works \cite{despadQ, despadmultiple, despadregularityI, despadregularityII, despadregularityIII}. A sharp dimension estimate for the singular set is also known for minimal Lipschitz graphs via the combined work of Allard (\cite{allard}), Allard -- Almgren (\cite{allardalmgrenstructure}), Barbosa (\cite{barbosaextrinsic}) and Lawson -- Osserman (\cite{lawsonosserman}). However, there are even fewer examples where, in arbitrary codimension, a more detailed analysis of the singular set has been possible (\emph{e.g.} gaining precise asymptotics on approach to singularities or proving uniqueness of tangent cones). Of particular note are the following cases: The $n=1$ case (a complete description of one-dimensional stationary varifolds was given by Allard and Almgren in \cite{allardalmgrenstructure}), the two-dimensional area-minimizing case (uniqueness of tangent cones is due to White: \cite{whiteuniqueness}, and complete regularity Chang: \cite{changtwodimcurr}) and the work of Simon (\cite{simoncylindrical}) on `multiplicity one classes' of minimal submanifolds, in which he introduced techniques designed to control the linearization of the minimal surface operator (the `blow-up') at certain singular minimal cones. Since Simon's work, the blow-up method has been adapted by Wickramasekera to some codimension 1 settings in which the multiplicity one hypothesis does not hold (in \cite{wickrigidity}, \cite{wick08} and \cite{wickgeneral}, Wickramasekera studies certain higher multiplicity singularities of stable hypersurfaces). 

Here we study the regularity properties of the graph of a two-valued Lipschitz function when that graph is assumed to be minimal, \emph{i.e.} assumed to be a stationary point of the $n$-dimensional area functional in $\RR^{n+k}$. We will call such an object a `minimal two-valued graph' (note that by `two-valued function' we mean a function that maps points in $\RR^n$ to unordered pairs of points in $\RR^k$). This context, the context in which we work, is in arbitrary codimension and in the presence of higher multiplicity singularities (which we will explain shortly). Also, we do not assume that our objects are area-minimizing or stable. Our most restrictive assumption, and one that we rely on, is that of being a two-valued graph. We are interested in the local structure of a minimal two-valued graph close to certain density two singular points. More specifically, we describe the structure of an $n$-dimensional minimal two-valued graph and its singular set close to points at which at least one tangent cone is equal to a transversely intersecting pair of $n$-dimensional subspaces, and close to points at which at least one tangent cone is equal to a union of four $n$-dimensional half-spaces meeting only along an $(n-1)$-dimensional axis. The main results will be stated in full detail shortly, but roughly speaking can be summarized in the following three statements. 

\vspace{0.5cm}

\noindent \textbf{Theorem 1.} \emph{If an $n$-dimensional minimal two-valued graph lies sufficiently close to a pair of planes meeting along an axis of dimension at most $(n-2)$, then it must be equal to the union of two smooth minimal submanifolds, each of which lies close to one of the two planes and which intersect only along a subset of an $(n-2)$-dimensional smooth submanifold that is graphical over the axis of the pair of planes}.

\vspace{0.5cm}

\noindent \textbf{Theorem 2.} \emph{If an $n$-dimensional minimal two-valued graph lies sufficiently close to a pair of planes meeting along an $(n-1)$-dimensional axis, then its singular set is contained in an $(n-1)$-dimensional $C^{1,\alpha}$ submanifold and at each singular point there is a unique tangent cone equal to either a transversely intersecting pair of planes or a union of four half-planes meeting only along an $(n-1)$-dimensional axis}.

\vspace{0.5cm}

\noindent \textbf{Theorem 3.} \emph{If an $n$-dimensional minimal two-valued graph lies sufficiently close to a union of four $n$-dimensional half-planes that meet only along an $(n-1)$-dimensional axis and that are not equal to a pair of planes, then it must be equal to the union of four smooth, minimal submanifolds with boundary meeting only along an $(n-1)$-dimensional $C^{1,\alpha}$ submanifold, their common boundary}.

\vspace{0.5cm}

\noindent The key ingredient in the proof of Theorems 1-3 is a so-called ``Excess Improvement Lemma'' (Lemma 6.2). It says that if a minimal two-valued graph is sufficiently close in $L^2$ distance at scale 1 to a cylindrical cone $\bC$ of the appropriate form, then there exists another cone $\bC^{\prime}$, relative to which the $L^2$ distance at a smaller scale $\theta$ has decayed by a factor that is better than that which is expected from just scale invariance. The basic structure of the proof of this Lemma is very similar to that of Lemma 1 of \cite{simoncylindrical}, from which the main results are then achieved by careful iteration of this Lemma. At a technical level, our difficulties are compounded by the fact that there can be significant gaps in the part of the singular set consisting of points $X$ at which the density $\Theta_V(X)$ is at least 2, \emph{i.e.} gaps in the set $\{X : \Theta_V(X) \geq \Theta_{\bC}(0) = 2 \}$. In the aforementioned manifestations of the method, strongest results were achieved when it was either assumed (Remark 1.14 of \cite{simoncylindrical}) or could be checked (\emph{e.g.} using the stability inequality as in \cite{wickrigidity} and \cite{wickgeneral}) that there were `lots' of good density points in the sense that a $\delta$-neighbourhood of the set $\{X : \Theta_V(X) \geq \Theta_{\bC}(0) \}$ contained the axis of the cone. In our setting, this does not hold and extra effort must be expended to prove an excess decay Lemma that can (at multiplicity two points) be iterated indefinitely and thus still yield strong results.

We also show the following result, which classifies certain tangent cones and in light of well-known stratification results about the singular set, leads to global information.

\vspace{0.5cm}

\noindent \textbf{Theorem 4} \emph{Any tangent cone to an $n$-dimensional minimal two-valued graph that is invariant under translations along an $(n-2)$-dimensional subspace must be equal to either a union of two distinct, multiplicity one $n$-dimensional planes intersecting along an $(n-2)$-dimensional subspace or equal to a union of four $n$-dimensional half-planes meeting only along an $(n-1)$-dimensional axis}.

\vspace{0.5cm}

\noindent \textbf{Acknowledgements.} I would like to thank my advisor Neshan Wickramasekera for his guidance and encouragement and for the many hours that he spent teaching me. I would also like to thank the referee for making many helpful remarks on the first version of the article and drawing my attention to places where the clarity needed to be improved.

\section{Notation and Main Theorems} 
\subsection{Basic Notation}

We start by setting out basic notation and terminology that is common to all sections.

\noindent We will use upper case letters such as $X$ to denote points in $\RR^{n+k}$. \smallskip

\noindent For $X \in \RR^{n+k}$, we will write $R = R(X) = |X|$. \smallskip

\noindent For $X_0 \in \RR^{n+k}$ and $\rho > 0$, $B_{\rho}(X_0) = \{X \in \RR^{n+k} : |X_0-X| < \rho\}$.\smallskip

\noindent For $X_0 \in \RR^n\times\{0\}^k$ and $\rho > 0$, $B^n_{\rho}(X_0) = \{X \in \RR^n \times \{0\}^k : |X_0-X| < \rho\}$. \smallskip

\noindent For $X_0 \in \RR^{n+k}$ and $\rho > 0$, we define the transformations $\eta_{X_0,\rho}$, $T_{X_0}$ and $\tau_{X_0} : \RR^{n+k} \to  \RR^{n+k}$ by $\eta_{X_0,\rho}(X) = \rho^{-1}(X-X_0)$, $T_{X_0}(X) = X + X_0$ and $\tau_{X_0}(X) = X - X_0$. \smallskip

\noindent For $s \geq 0$, $\cH^s$ denotes the $s$-dimensional Hausdorff measure on $\RR^{n+k}$ and $\omega_n = \cH^n(B^n_1(0))$. \smallskip

\noindent For $A, B \subset \RR^{n+k}$, $d_{\cH}(A,B)$ denotes the Hausdorff distance between $A$ and $B$.\smallskip

\noindent For $X \in \RR^{n+k}$ and $A \subset \RR^{n+k}$, $\dist(X,A) = \inf_{Y \in A}|X-Y|$. \smallskip

\noindent For $A \subset \RR^{n+k}$ and $\rho > 0$, we write $(A)_{\rho} = \{X \in \RR^{n+k} : \dist(X,A) < \rho \}$. \smallskip

\noindent By a \emph{plane} we mean any affine $n$-dimensional subspace of $\RR^{n+k}$ and for any plane $T$, we use $\bp_T$ to denote the orthogonal projection onto $T$. More commonly, we will use the shorthand $Y^{\top_T} = \bp_T Y$ and $Y^{\perp_T} = \bp_{T^{\perp}}Y$. \smallskip

\noindent By a \emph{half-plane}, we mean a closed half-plane: Any set which is the closure of one of the connected components of $T\setminus L$, where $T$ is any plane and $L$ is any $(n-1)$-dimensional subspace of $T$. For any half-plane $H$, we write $\bp_{H}$ for the orthogonal projection onto the unique plane containing $H$. \smallskip

\noindent $G_n$ denotes the space of $n$-dimensional subspaces of $\RR^{n+k}$. \smallskip

For an integral $n$-varifold $V$ (see \cite{allard} or \cite[Chapter 4, 8]{simongmt}) in the open set $U$, we use the following notation:\smallskip

\noindent The \emph{weight measure} $\|V\|$ of $V$ is the Radon measure on $U$ given by $\|V\|(A) = V(\{(x,S) \in U \times G_n : x \in A\})$ and $\spt\|V\|$ is called the \emph{support} of the varifold $V$. \smallskip

\noindent Given an $n$-rectifiable set $M$, $|M|$ denotes the multiplicity one varifold associated with $M$. \smallskip

\noindent For $Z \in \spt \|V\|$, $\Var\Tan(V,Z)$ denotes the set of all varifold tangent cones to $V$ at $Z$, \emph{i.e.} each $W \in \Var\Tan(V,Z)$ arises as $W = \lim_{j \to \infty}(\eta_{Z,\rho_j})_*V$ for some sequence of positive numbers $\rho_j \to 0$, where, for any proper, injective, Lipschitz map $f$, $f_*V$ is the \emph{pushforward} of $V$ by $f$. \smallskip

\noindent For $\cH^n$-a.e. $Z \in \spt \|V\|$, we write $T_ZV$ for the approximate tangent plane (see \cite[Chapter 3]{simongmt}) to $\spt \|V\|$ at $Z$. \smallskip

\noindent $\reg V$ denotes the regular part of $V$, by which we mean that $X \in \reg V$ if and only if $X \in \spt \|V\|$ and there exists $\rho > 0$ such that $B_{\rho}(X) \cap \spt \|V\|$ is a smooth, $n$-dimensional embedded submanifold of $B_{\rho}(X)$. \smallskip

\noindent $\sing V$ denotes the (interior) singular part of $V$, \emph{i.e.} $\sing V = (\spt \|V\| \setminus \reg V) \cap U$. \smallskip

\subsection{The Minimal Surface System} For the single-valued function $f : B_2^n(0) \to \RR^k$, the area formula tells us that
\beq
\cH^n(\gr f) = \int_{B_2^n(0)} \det(\delta_{ij} + \Sigma_{\kappa=1}^k D_if^{\kappa}D_jf^{\kappa})^{1/2} \, d\cH^n.
\eeq
If $V_f := |\gr f|$ is stationary as a rectifiable $n$-varifold in $B_2^n(0)\times \RR^k$, then it is in particular stationary with respect to deformations only in the vertical directions and therefore
\beq
\left.\frac{d}{dt}\right|_{t=0} \cH^n(\gr (f + t\varphi)) = 0
\eeq
for any $\varphi \in C^{\infty}_c(B_2^n(0);\RR^k)$. This implies (as can be seen by direct computation using the fact that for a square matrix $A(t)$ that is a function of the scalar parameter $t$, one has $(d/dt)(\det A(t)) = \tr(\mathrm{adj}A(t)A'(t))$ ) that $f$ is a Lipschitz weak solution to the Minimal Surface System, \emph{i.e.} it satisfies
\beq \label{weakmss}
\int_{B_2^n(0)} \sqrt{g(f)}g^{ij}(f)\sum_{\kappa=1}^kD_if^{\kappa}D_j\varphi^{\kappa} d\cH^n = 0
\eeq
 for all $\varphi \in C^{\infty}_c(B_2^n(0);\RR^k)$, where $g_{ij}(f) = \delta_{ij} + \Sigma_{\kappa=1}^k D_if^{\kappa}D_jf^{\kappa}$, $g^{ij}(f) = (g_{ij}(f))^{-1}$ and $g(f) = \det g_{ij}(f)$. A homogeneous degree one Lipschitz weak solution $g : \RR^d \to \RR^k$ to the Minimal Surface System is necessarily linear if $d \in \{1,2,3\}$ (this follows from the main theorem of \cite{barbosaextrinsic}). Using this in conjunction with Allard's Regularity Theorem and the stratification of the singular set (see \eqref{singstra}), we see that a Lipschitz weak solution to the Minimal Surface System is $C^{1,\alpha}$ away from codimension four set. Hence, by standard regularity theory for elliptic systems (see \cite{morreymultiple}), such a function is analytic away from a codimension four set. In particular, $\sing V_f \leq n-4$. The following example due to Lawson and Osserman (\cite{lawsonosserman}) shows that such a function can indeed have singularities on a codimension four set:

\begin{exam} [\cite{lawsonosserman}] \label{lawsossecone} Consider $S^3$ to be the unit sphere in $\CC^2 \cong \RR^4$ and consider $S^2$ to be the unit sphere in $\RR\times\CC \cong \RR^3$. Define 
$\eta : S^3 \to S^2$ by 
\[
\eta(z_1,z_2) = (|z_1|^2 - |z_2|^2,2z_1\bar{z_2})
\]
(this is the \emph{Hopf map}). The homogeneous degree one function $f : \RR^4 \to \RR^3$ given by 
\[
f(x) = \frac{\sqrt{5}}{2}|x|\eta\Bigl(\frac{x}{|x|}\Bigr)\ \ \text{for}\ x \neq 0
\]
is a Lipschitz weak solution to the minimal surface system on $\RR^4$ with an isolated singularity at the origin. \end{exam}

\subsection{Two-Valued Functions}
We write $\cA_2(\RR^k)$ for the set of all unordered pairs of points in $\RR^k$. A \emph{two-valued function}  (or more generally a \emph{two-$\RR^k$-valued function}) on an open set $\Omega \subset \RR^n$ is a map $f : \Omega \to \cA_2(\RR^k)$. We equip $\cA_2(\RR^k)$ with the metric
\[
\cG(a,b) := \min \{ |a_1 - b_1| + |a_2 - b_2|, |a_1 - b_2| + |a_2 - b_1|\},
\]
where $a = \{a_1,a_2\} \in \cA_2(\RR^k)$ and $b = \{b_1,b_2\} \in \cA_2(\RR^k)$. Thus a two-valued function $f$ on $\Omega$ is \emph{Lipschitz} if there exists some constant $L$ such that 
\[
\cG(f(x),f(y)) \leq L|x-y|
\]
for all $x,y, \in \Omega$. We say that $f$ is \emph{differentiable} at $x \in \Omega$ if there exists a two-$\RR^k$-valued affine function $A_x$ on $\RR^n$ of the form
\[
A_x(h) = \{f_1(x) + A_1(x)h,f_2(x) + A_2(x)h\}
\]
for $k\times n$ matrices $A_1(x)$ and $A_2(x)$ such that
\[
\lim_{h \to 0}|h|^{-1}\cG(f(x),A_x(h)) = 0.
\]
In this case, we write $Df(x) = \{Df_1(x),Df_2(x)\}$ instead of $\{A_1(x),A_2(x)\}$. A two-valued function $f$ on $\Omega$ is \emph{continuously differentiable} on $\Omega$ and we write $f \in C^1(\Omega;\cA_2(\RR^k))$ if it is both differentiable at every point of $\Omega$ and the two-valued function $Df$ is continuous on $\Omega$. We say that $f \in C^{1,\mu}(\Omega;\cA_2(\RR^k))$ for $\mu \in (0,1]$ if $f$ is $C^1$ and also
\[
|f|_{1,\alpha;\Omega} < \infty
\]
where
\[
|f|_{1,\alpha;\Omega} = \sup_{\Omega}|f| + \sup_{\Omega}|Df| + [Df]_{\alpha,\Omega}.
\]
Here, the H\"{o}lder coefficient is interpreted in the obvious way, \emph{i.e.}
\[
[Df]_{\alpha;\Omega} = \sup_{\substack{x,y, \in \Omega\\ x\neq y}} |x-y|^{-\alpha}\cG(Df(x), Df(y)).
\]
Note that $C^1(\Omega;\cA_2(\RR^k))$ and $C^{1,\alpha}(\Omega;\cA_2(\RR^k))$ are not linear spaces as there is in general no well-defined pointwise addition on two-valued functions. We define the graph of a two-valued function $f$ by 
\[
\gr f := \{ (x,y) \in \Omega\times\RR^k : y \in \{f_1(x),f_2(x)\}\}.
\]

\subsection{Minimal Two-Valued Graphs}

Write $\pi$ for the orthogonal projection of $\RR^{n+k}$ onto $\RR^n\times\{0\}^k$. Now, it is not difficult to see that if $f : B_2^n(0) \to \cA_2(\RR^k)$ is Lipschitz, then $\gr f$ is $n$-rectifiable. So, by taking $\gr f$ together with the multiplicity function defined on it which is equal to two at points $Y \in \gr f$ for which $f_1(\pi Y) = f_2(\pi Y)$ and equal to 1 otherwise, we can consider $\gr f$ to be an integral $n$-varifold $V_f$ in $B_2^n(0)\times \RR^k$. We will say that $V = V_f$ is the varifold \emph{associated to} $\gr f$ or to $f$. When $V$ is stationary in $B_2^n(0) \times \RR^k$, \emph{i.e.} when
\beq \label{stat}
\int_{(B_2^n(0)\times \RR^k) \times G_n}\divr_{S}\Phi(x)\ dV(x,S) = 0
\eeq
for all $\Phi \in C^1_c(U;\RR^{n+k})$, we say that $V$ is a \emph{minimal two-valued graph}. Write $\cV$ for the set of all minimal two-valued graphs in $B_2^n(0)\times\RR^k$ that are associated to some Lipschitz function $f : B_2(0) \to \cA_2(\RR^k)$.

Let $V = V_f \in \cV$. For $X \in \spt \|V\|$, the assignment of single-valued Lipschitz functions $f_i : B_{\delta}^n(\pi X)\times\{0\}^k \to \RR^k$ for $i=1,2$ and some $\delta > 0$ such that $V \res (B^n_{\delta}(\pi X)\times\RR^k) = |\gr f_1| + |\gr f_2|$ is called a \emph{labelling} of $f$ in $B^n_{\delta}(\pi X)$. If $U \subset B_2^n(0) \times \RR^k$ is such that $V \res U = V_1 + V_2$, where for $i=1,2$, $V_i$ is a (possibly empty) stationary, Lipschitz single-valued graph, then we say that $V$ \emph{decomposes} in $U$. From the definition of stationarity and the fact that $f$ is continuous, it is easy to see that $V$ decomposes in any cylindrical region $\Omega \times \RR^k$ which is free of multiplicity two points and in a neighbourhood of any multiplicity one point. The \emph{branch set} $\cB_f$ is the complement in $\gr f$ of the set $\{X \in \gr f : V\ \text{decomposes in a neighbourhood of}\ X\}$. Any $X \in \cB_f$ is called a \emph{branch point} of $V$. Let us remark here that in general $V \in \cV$ does not globally decompose: Large classes of $C^{1,\alpha}$ branched minimal two-valued graphs were constructed directly in \cite{simonwickpresboun}, \cite{rosales2valumse} and \cite{krumexis}. 


\subsection{Stratification of The Singular Set} 

It is well-known that the singular set of a stationary integral varifold can be `stratified' in the following way: For any stationary cone $\bC$ (where, for our purposes, `cone' will mean an integral varifold whose support is a union of rays emanating from the origin), we write $S(\bC) : = \{ Z \in \RR^{n+k} : \Theta_{\bC}(Z) = \Theta_{\bC}(0)\}$. We call this set the \emph{spine} of $\bC$ and it is not difficult to show that it is a linear subspace of $\RR^{n+k}$. Given a stationary varifold $V$ we write
\beq 
\cS_j = \{ X \in \sing V : \dim S(\bC) \leq j\ \ \forall\ \bC \in \Var\Tan(V,X)\}
\eeq
Then we have that
\beq \label{singstra}
\dim_{\cH}\cS_j \leq j.
\eeq
 This was first shown for stationary integral varifolds by F. Almgren (\cite{almgrenbig}), but is true in analogous forms in other settings in the study of solutions to geometric variational problems ( \emph{e.g.} energy-minimizing maps \cite{simontheorems} and mean curvature flow \cite{whitestratification}. Or see \cite{simongmt} for a general abstract version).

\subsection{Relevant Classes of Varifolds}

\noindent Write $\cP$ for the set of all integral $n$-varifolds in $\RR^{n+k}$ which are of the form $\bC = |\bP_1| + |\bP_2|$, where $\bP_1$, $\bP_2$ are \emph{distinct} planes meeting only along an affine subspace $A(\bC) := \bP_1 \cap \bP_2 \neq \emptyset$, which we call the \emph{axis} of $\bC$. \smallskip

\noindent We write $\cP_{\emptyset}$ for the set of all integral $n$-varifolds in $\RR^{n+k}$ which are of the form $\bC = |\bP_1| + |\bP_2|$, where $\bP_1$, $\bP_2$ are \emph{disjoint} planes. \smallskip

\noindent We write $\cP_{\leq n-2}$ for the set of all $\bC \in \cP$ with $\dim A(\bC) \leq n-2$ and $\cP_{n-1}$ for the set of all $\bC \in \cP$ with $\dim A(\bC) = n-1$. \smallskip

\noindent Write $\cC_{n-1}$ for the set of all integral $n$-varifolds in $\RR^{n+k}$ which are of the form $\bC = \sum_{i=1}^4 |\bH_i|$, where for $i=1,...,4$, the $\bH_i$ are distinct half-planes meeting only along their common boundary $A(\bC) = \cap_{i=1}^4 \bH_i$, the axis of $\bC$, which is an affine $(n-1)$-dimensional subspace. Note that $\cP_{n-1} \subset \cC_{n-1}$. \smallskip

\noindent Write $\cC := \cC_{n-1} \cup \cP_{\leq n-2}$. \smallskip

\noindent For $\bC \in \cC$, when the coordinates of $\RR^{n+k}$ are labelled in such a way that for $m:= \dim A(\bC)$ and $l:=n-m$ we have $A(\bC) = \{0\}^{l+k}\times\RR^m \subset \RR^{l+k}\times\RR^m$ (so that $X = (x,y) \in A(\bC)^{\perp}\times A(\bC) = \RR^{l+k}\times \RR^m = \RR^{n+k}$), we will say that $\bC$ is \emph{properly aligned}. In this case, $\bC = \bC_0\times\RR^m$, where $\sing \bC_0 = \{0\}$ and $\bC_0$, the \emph{cross-section} of $\bC$, is either the sum of two distinct $l$-dimensional subspaces of $\RR^{l+k}$ meeting only at the origin or the sum of four distinct rays in $\RR^{1+k}$ meeting only at the origin, depending on whether $\bC \in \cP_{\leq n-2}$ or $\bC \in \cC_{n-1}$, respectively. When $\bC^{(0)} \in \cC_{n-1}$ is properly aligned, we write $\{\omega_1,...,\omega_4\} = \{r_{\bC^{(0)}} = 1\} \cap A(\bC^{(0)})^{\perp} \cap \spt \|\bC^{(0)}\|$, where, for any cone $\bC \in \cC$ we define $r_{\bC} = r_{\bC}(X) := \dist(X,A(\bC))$. \smallskip

\noindent For $V \in \cV$ and $\bC$, $\bC^{(0)} \in \cC$, we define
 \begin{align*}
\cQ_V&(\bC) := \left(\int_{B_2^n(0)\times\RR^k}\dist^2(X,\spt \|\bC\|) \, d\|V\|(X)\right. \\
&+ \left.\int_{(B_2^n(0)\times\RR^k) \setminus \{r_{\bC^{(0)}} < 1/8\}}\dist^2(X,\spt \|V\|) \, d\|\bC\|(X)\right)^{1/2}.
\end{align*}  

\noindent Finally we define $\cV_L$ to be the set of all $V = V_f \in \cV$ for which the Lipschitz constant of $f$ is at most $L$.

\subsection{Main Results}

Suppose throughout these statements that we have fixed $L > 0$.

\begin{customthm}{1} \label{one} Let $\bC^{(0)} \in \cP_{\leq n-2}$. There exists $\eps = \eps(n,k,\bC^{(0)},L)$ $> 0$ such that the following is true. If $V \in \cV_L$ is such that $0 \in \spt \|V\|$ and $\cQ_V(\bC^{(0)}) < \eps$, then we have the following conclusions:
\begin{enumerate}
\item $V\res B_{1/2}(0) = |M_1| + |M_2|$, where, for $i=1,2$, $M_i$ is a smooth, embedded $n$-dimensional minimal submanifold of $B_{1/2}(0)$.

\item $\sing V \cap B_{1/2}(0) = M_1 \cap M_2 \subset \gr \varphi$, where, for some \\ $\alpha = \alpha(n,k,\bC^{(0)},L) \in (0,1)$, $\varphi : A(\bC^{(0)}) \cap B_{1/2}(0) \to A(\bC^{(0)})^{\perp}$ is a $C^{1,\alpha}$ function satisfying $\|\varphi\|_{C^{1,\alpha}(A(\bC^{(0)}) \cap B_{1/2}(0))} \leq c\cQ_V(\bC^{(0)})$ for some $c$ $=$ $c(n,k,$ $\bC^{(0)},L)$.

\item At each $Z \in \sing V \cap B_{1/2}(0)$, we have that $\Var\Tan(V,Z) = \{\bC_Z\}$ for some $\bC_Z \in \cP_{\leq n-2}$ and we have the decay estimate
\beq \label{onea}
\rho^{-n-2}\int_{B^n_{\rho}(\pi Z)\times\RR^k} \dist^2(X,\spt \|\bC_Z\|)d\|V\|(X) \leq c\rho^{2\alpha} \cQ^2_V(\bC^{(0)}),
\eeq
which holds for all $\rho \in (0,1/8)$ and for some $c = c(n,k,\bC^{(0)},L) > 0$.
\end{enumerate}
\end{customthm}

\begin{customthm}{2} \label{two} Let $\bC^{(0)} \in \cP_{n-1}$. There exists $\eps = \eps(n,k,\bC^{(0)},L) > 0$ such that the following is true.  If $V \in \cV_L$ is such that $0 \in \spt \|V\|$ and $\cQ_V(\bC^{(0)}) < \eps$, then we have the following conclusions:
\begin{enumerate}
\item $\sing V \cap B_{1/2}(0) \subset \gr \varphi$, where, for some $\alpha = \alpha(n,k,\bC^{(0)},L) \in (0,1)$, $\varphi : A(\bC^{(0)}) \cap B_{1/2}(0) \to A(\bC^{(0)})^{\perp}$ is a $C^{1,\alpha}$ function satisfying \\$\|\varphi\|_{C^{1,\alpha}(A(\bC^{(0)}) \cap B_{1/2}(0))}$ $\leq$  $c\cQ_V(\bC^{(0)})$ for some $c$ $=$ $c(n,k,$ $\bC^{(0)},L)$.

\item At each $Z \in \sing V \cap B_{1/2}(0)$, we have that $\Var\Tan(V,Z) = \{\bC_Z\}$ for some $\bC_Z \in \cC$ and we have the decay estimate
\beq \label{twoa}
\rho^{-n-2}\int_{B^n_{\rho}(\pi Z)\times\RR^k} \dist^2(X,\spt \|\bC_Z\|)d\|V\|(X) \leq c\rho^{2\alpha} \cQ^2_V(\bC^{(0)}),
\eeq
which holds for all $\rho \in (0,1/8)$ and for some $c = c(n,k,\bC^{(0)},L) > 0$.
\end{enumerate}
\end{customthm}

\begin{customthm}{3} \label{three} Let $\bC^{(0)} \in \cC_{n-1}\setminus \cP_{n-1}$. There exists $\eps = \eps(n,k,\bC^{(0)},L) > 0$ such that the following is true.  If $V \in \cV_L$ is such that $0 \in \spt \|V\|$ and $\cQ_V(\bC^{(0)}) < \eps$, then we have the following conclusions:
\begin{enumerate}
\item $V \res B_{1/2}(0) = \sum_{j=1}^4 |M_j|$, where for $j=1,2,3,4$, $M_j$ is a smooth, embedded $n$-dimensional minimal submanifold in $B_{1/2}(0)$.

\item $\sing V \cap B_{1/2}(0) = \gr \varphi \cap B_{1/2}(0) = \cap_{j=1}^4\overline{M_j}$, where, for some $\alpha = \alpha(n,k,\bC^{(0)},L) \in (0,1)$, $\varphi : A(\bC^{(0)}) \cap B_{1/2}(0) \to A(\bC^{(0)})^{\perp}$ is a $C^{1,\alpha}$ function satisfying $\|\varphi\|_{C^{1,\alpha}(A(\bC^{(0)}) \cap B_{1/2}(0))}$ $\leq$ $c\cQ_V(\bC^{(0)})$ for some $c$ $=$ $c(n,k,$ $\bC^{(0)},L)$. Moreover,  $\gr \varphi \cap B_{1/2}(0) = \partial M_i$ occurs in the sense of $C^{1,\alpha}$ manifolds-with-boundary for $i = 1,2,3,4$.

\item At each $Z \in \sing V \cap B_{1/2}(0)$, we have that $\Var\Tan(V,Z) = \{\bC_Z\}$ for some $\bC_Z \in \cC_{n-1}\setminus \cP_{n-1}$ and we have the decay estimate
\beq \label{threea}
\rho^{-n-2}\int_{B^n_{\rho}(\pi Z)\times\RR^k} \dist^2(X,\spt \|\bC_Z\|)d\|V\|(X) \leq c\rho^{2\alpha} \cQ^2_V(\bC^{(0)}),
\eeq
which holds for all $\rho \in (0,1/8)$ and for some $c = c(n,k,\bC^{(0)},L) > 0$.
\end{enumerate}
\end{customthm}

Note first that our assumptions do not immediately ensure multiplicity one convergence of $V$ to $\bC^{(0)}$ away from the axis of $\bC^{(0)}$ (this was an assumption in \cite{simoncylindrical}). To expand on this point a little, in a region very close to the axis of $\bC^{(0)}$, the smallness of $\cQ_V(\bC^{(0)})$ only amounts to the $L^2$ smallness of $\dist^2(X,\spt\|\bC^{(0)}\|)$ over the support of the varifold and a priori, smallness of this latter quantity allows for both `sheets' of the two-valued graph to be very close to the same plane of $\bC^{(0)}$. This is what we mean by having to deal with higher multiplicity singularities.

Note that the conclusions of Theorem \ref{one} imply that the the varifold is smooth \emph{as a two-valued graph} in a neighbourhood of a singular point at which at least one tangent cone belongs to $\cP_{\leq n-2}$. The conclusions of the other theorems however do not imply that $V$ is $C^1$ as a two-valued graph. The following example makes this explicit:

\begin{exam} Let $f$ denote the two-$\RR^2$-valued function on $\RR$ given by $f(t) = \{(t,0), (-t,0)\}$ for $t \leq 0$ an $f(t) = \{(0, t), (0,-t)\}$ for $t > 0$. This is Lipschitz as a two-valued function and its graph is minimal in $\RR^3$ (and indeed equal to a union of four smooth `sheets'). However, it is clearly not $C^1$ as a two-valued function at the origin. The codimension of this example is irrelevant and so one can produce such examples of any dimension and codimension by `crossing' this example with Euclidean space. \end{exam}

The next example shows that more exotic singular minimal two-valued graphical cones exist:

\begin{exam} If $f$ is as in Example \ref{lawsossecone}, then the two valued function $x \mapsto \{f(x),-f(x)\}$ is an example of a minimal two-valued Lipschitz graph which is a cone and which is not equal to a union of planes or half-planes. \end{exam}

The conclusions of Theorem \ref{one} to \ref{three} imply that sufficiently close to a singular point at which at least one tangent cone belongs to $\cC$, every tangent cone is unique and also belongs to $\cC$. In particular, a point with a tangent cone in $\cC$ cannot be the limit point of points at which there are `exotic tangent cones', such as that of the above example.

If one is able to achieve $C^{1,\alpha}$ regularity for a minimal two-valued graph, then one may apply the results of \cite{simonwickfrequency} to deduce that the two-valued function in question is $C^{1,1/2}$. Assuming only $C^{1,\alpha}$ regularity to begin with, this is the best possible general result for the regularity of such objects, as the following example shows:

\begin{exam} Consider the irreducible holomorphic variety $I := \{(z,w) \in \CC\times\CC : z^2 = w^3\} \subset \RR^4$. It is well known that such a variety is area-minimizing (because it is `calibrated') and therefore minimal and yet it is easy to see that if viewed as the two-valued graph of $w \mapsto w^{3/2}$, then the regularity at the origin is no better than $C^{1,1/2}$. \end{exam}

We also establish the following theorem:

\begin{customthm}{4} \label{four} Let $V \in \cV$ and $X \in \sing V$. If $\bC \in \Var\Tan(V,X)$ is such that $\dim S(\bC) = n-2$, then $\bC \in \cP$ and $A(\bC) = S(\bC)$. \end{customthm}

When combined with our two main $\eps$-regularity theorems, this result implies that we have a complete description of a minimal two-valued Lipschitz graph near points in $\cS_{n-1} \setminus \cS_{n-2}$ and points in $\cS_{n-2} \setminus \cS_{n-3}$. Let $\fB$ denote the set of points $X \in \sing V$ for which there exists $\bC \in \Var\Tan(V,X)$ equal to a multiplicity two hyperplane. Write $\tilde{\cS}_{n-1}$ for the set of points $X \in \sing V$ for which there exits $\bC \in \Var\Tan(V,X) \cap \cC_{n-1}$ and similarly $\tilde{\cS}_{n-2}$ for the set of points $X \in \sing V$ for which there exits $\bC \in \Var\Tan(V,X) \cap \cP_{ \leq n-2}$. Finally define $\tilde{\cS}_{n-3} := \sing V \setminus (\fB \cup \tilde{S}_{n-1} \cup \tilde{\cS}_{n-2})$.

\begin{corollary} For $V \in \cV$, the singular set $\sing V$ is the disjoint union $\fB \cup \tilde{\cS}_{n-1} \cup \tilde{\cS}_{n-2} \cup \tilde{\cS}_{n-3}$, where
\begin{enumerate}[nolistsep]
\item By definition, for every $X \in \fB$, there is $\bC \in \Var\Tan(V,X)$ equal to a multiplicity two $n$-dimensional plane. 

\item $\dim_{\cH} \tilde{\cS}_{n-1} \leq n-1$ , $\tilde{\cS}_{n-1} \cup \tilde{\cS}_{n-2}$ is relatively open in $\sing V$ and for every $X \in \tilde{\cS}_{n-1}$, we have that the conclusions of Theorem \ref{two} or \ref{three} hold in a neighbourhood of $X$.

\item $\dim_{\cH} \tilde{\cS}_{n-2} \leq n-2$, $\tilde{\cS}_{n-2}$ is relatively open in $\sing V$ and for every $X \in \tilde{\cS}_{n-2}$, we have that the conclusions of Theorem \ref{one} hold in a neighbourhood of $X$.

\item $\dim_{\cH} \tilde{\cS}_{n-3} \leq n-3$ and the closure of $\tilde{\cS}_{n-3}$ does not intersect $\tilde{\cS}_{n-1} \cup \tilde{\cS}_{n-2}$.
\end{enumerate} 
\end{corollary}

One application of our main results is to two-valued graphs that are locally area minimizing: It is a standard fact (and not difficult to see by a comparison argument) that an $n$-dimensional area-minimizing current without boundary cannot have a tangent cone with spine dimension $(n-1)$. The work of Almgren (\cite{almgrenbig}) implies that for such a current, the set of points where there is a multiplicity 2 tangent plane has Hausdorff dimension at most $(n-2)$.  Thus we get the following corollary of our main theorems:

\begin{corollary} If $V \in \cV$ is a locally area-minimizing current with $\partial \, V  = 0$ in $B_2^n(0)\times \RR^k$, then $V$ is smoothly immersed away from a closed set $S$ with $\dim_{\cH}(S) \leq n-2$. Moreover, $S$ is the disjoint union $S_1 \cup S_2$, where $S_1$ is the set of points at which there exists at least one tangent cone equal to a multiplicity two hyperplane and $\dim_{\cH}(S_2) \leq n-3$.
\end{corollary}

\section{Gaps In The Top Density Part}

In this section we analyse the structure of a minimal two-valued graph in regions in which there are no points of density greater than or equal to 2.

\subsection{Decomposition Into Single-Valued Graphs}
The result of this subsection is the following:

\begin{theorem} \label{thm:grapdeco}
Let $V \in \cV$ and let $U \subset B_2^n(0)\times \RR^k$ be a simply connected open set. If $\{Z \in U\ : \Theta_V(Z) \geq 2\} = \emptyset$, then $V$ decomposes in $U$.
\end{theorem}

We prove a separate lemma first.

\begin{lemma} \label{dimredu}
Suppose that $V \in \cV$ is such that $\{Z \in U\ :\ \Theta_V(Z) \geq 2\} = \emptyset$ for some set $U \subset B_2^n(0)\times \RR^k$. Then $\dim_{\cH} (\sing V \cap U) \leq n-3$.
\end{lemma}

\begin{proof} Pick $X \in \sing V$ and consider $\bC \in \Var\Tan(V,X)$. Bearing in mind the stratification of the singular set (\eqref{singstra}), the proof will be complete once we show that $\dim S(\bC) \leq n-3$. 

Suppose first that $\spt\|\bC\|$ is a single-valued graph (\emph{i.e.} suppose that $X$ is a multiplicity one point) and assume for the sake of contradiction that $\dim S(\bC) \in \{n,n-1,n-2\}$. If $\dim S(\bC) = n$, then $\bC$ would be a multiplicity one plane. By Allard's Regularity Theorem, this would mean that $X \in \reg V$, which is a contradiction. If $\dim S(\bC) = \{n-1, n-2\}$, then we can write $\bC = \bC_0 \times \RR^d$ where $d \in \{n-1,n-2\}$ and $\bC_0$ is the graph of a single-valued, Lipschitz, homogeneous degree one weak solution to the Minimal Surface System $g : \RR^{d'} \to \RR^k$, where $d' \in \{1,2\}$. In both of these cases, $g$ must be linear, which proves that $\bC$ is a multiplicity one plane and thus we derive a contradiction as before.

Suppose now that $\spt\|\bC\|$ is a two-valued graph. If $\dim S(\bC) = n$, then $\bC$ must be a multiplicity two plane which implies that $\Theta_V(X) = 2$, but this is false by hypothesis. If $\dim S(\bC) = n-1$ and we again write $\bC = \bC_0\times\RR^{n-1}$, then $\spt\|\bC_0\|$ is the union of four rays meeting at a point. This means that $\bC_0$ and hence $\bC$ has density equal to two at the origin and hence $\Theta_V(X) = 2$, which is again a contradiction. Finally, suppose that $\dim S(\bC) = n-2$ and write $\bC = \bC_0\times\RR^{n-2}$. Consider the \emph{link} $\bM := |\spt\|\bC_0\| \cap S^{2+k-1}|$, which is a 1-dimensional stationary integral varifold in the sphere. Suppose that $\sing \bM \neq \emptyset$ and pick $Y \in \sing \bM$. The Allard-Almgren classification of stationary 1-varifolds (\cite{allardalmgrenstructure}) together with the fact that $\spt\|\bC_0\|$ is a two-valued graph implies that a tangent cone $\bD \in \Var\Tan(\bM,Y)$ is a union of $4$ rays meeting at a point. This means that
\beq \label{dimredu1}
2 =  \Theta^1_{\bD}(0) = \Theta^1_{\bM}(Y) = \Theta^2_{\bC_0}(Y) \leq \Theta^2_{\bC_0}(0) = \Theta^n_{\bC}(0) = \Theta^n_V(X),
\eeq
which is a contradiction. Therefore $\bM$ is free of singular points and so must consist of a union of two disjoint great circles. We then deduce that $\Theta^2_{\bC_0}(0) = 2$ and therefore that $\Theta^n_V(X) = 2$. This is again a contradiction and we therefore have that $\dim S(\bC) \leq n-3$, as required. \end{proof}

\noindent \emph{Proof of Theorem \ref{thm:grapdeco}}  Write $\cS := \sing V \cap U$. Crucially, since $\dim_{\cH}\pi\cS \leq n-3$ (by the above lemma), we have that $\pi U \setminus \pi\cS$ is simply connected (see \emph{e.g.} the appendix to \cite{simonwickfrequency} for a proof that the complement in $\RR^n$ of a set of $\cH^{n-2}$ measure zero is simply connected). Write $\Omega := \pi U \setminus \pi\cS$. We construct two smooth functions $f_1, f_2 : \Omega \to \RR^k$ which are solutions to the Minimal Surface System on $\Omega$ and such that $\gr f \cap (\Omega\times\RR^k)$ is the disjoint union $\gr f_1 \cup \gr f_2$. 

First note that for any $z \in \Omega$, there exists $\eta_z > 0$ such that $\gr f \cap (B_{\eta_z}^n(z) \times\RR^k)$ is the disjoint union of two smooth graphs $G^z_a$ and $G^z_b$, say. Fix a point $x \in \Omega$ and define $f_1$ and $f_2$ by writing $G^x_a = \gr f_1$ and $G^x_b = \gr f_2$. For any other point $y \in \Omega$, since $\Omega$ is path-connected, we can find a simple continuous path $\gamma_y : [0,1] \to \Omega$ with $\gamma(0) = x$ and $\gamma(1) = y$. 

By the compactness of $\gamma := \gamma([0,1])$, we have that 
\beq
\gamma \subset B_{\eta_x}^n(x) \cup \bigcup_{i=1}^N B_{\eta_{z_i}}^n(z_i) \cup B_{\eta_y}^n(y)
\eeq
for some $z_i \in \gamma$ for $i=1,....N$. Assume that $\gamma(t_i) = z_i$ where $t_1 < ... < t_N$. Since $\gr f \cap ((B_{\eta_x}^n(x) \cup B^n_{\eta_{z_1}}(z_1)) \times \RR^k)$ is embedded and consists of two connected components, there is a bijection $\Phi : \{a,b\} \to \{a,b\}$ so that $G_{\Phi(a)}^{z_1} \cup \gr f_1$ and $G_{\Phi(b)}^{z_1} \cup \gr f_2$ are two disjoint, embedded, smooth submanifolds. Thus we can extend $f_1$ and $f_2$ to the domain $B_{\eta_x}^n(x) \cup B_{\eta_{z_1}}^n(z_1)$ in such a way that $f_1$ and $f_2$ are both still smooth solutions to the Minimal Surface System. 

We continue this process: Given smooth solutions to the minimal surface system $f_i : B_{\eta_x}^n(x) \cup (\bigcup_{i=1}^{K-1} B_{\eta_{z_i}}^n(z_i)) \to \RR^k$, for which 
\[
\gr f \cap \left( \bigl[ B_{\eta_x}^n(x) \cup \bigcup_{i=1}^{K-1} B_{\eta_{z_i}}^n(z_i) \bigr] \times\RR^k \right)
\]
is the disjoint union $\gr f_1 \cup \gr f_2$, the above procedure gives a labelling in $B_{\eta_{z_K}}^n(z_K)$ which suitably extends the domains of $f_i$ for $i=1,2$. When this process terminates, we have defined a labelling at $y$. Without loss of generality, we will assume that
\beq \label{grapdeco1}
\gr f_1 \cap (B_{\eta_y}^n(y) \times\RR^k) = G^y_a\ \text{and}\ \gr f_2 \cap (B_{\eta_y}^n(y) \times\RR^k) = G^y_b.
\eeq
Now, let $F$ denote the two-$\RR^{n+k}$-valued function $F(x) = (x,f(x))$ and notice that $F(\gamma)$ is embedded and is the disjoint union of two paths $\omega_1$ and $\omega_2$ in $\gr f \subset \RR^{n+k}$ such that $\omega_i(0) = (x,f_i(x))$ for $i=1,2$, say, and
\beq
\omega_1(1) \in G^y_a\ \text{and}\ \omega_2(1) \in G^y_b.
\eeq
We now see that the labelling produced in \eqref{grapdeco1} is well-defined: Take another path $\gamma'$ connecting $x$ to $y$ along which the same process has been performed but assume for the sake of contradiction that by labelling along a finite sequence of balls covering $\gamma'$ in the manner described above, we obtain a different - \emph{i.e} the opposite, as there are only two - labelling of $f$ in $B_{\eta_y}^n(y)$. Note again that since $\gamma ' \in \Omega$, the image $F(\gamma')$ is the disjoint union of two paths $\omega_1'$ and $\omega_2'$ in $\gr f$. This time we have $\omega_i'(0) = (x,f_i(x))$ as before, but 
\beq
\omega_1'(1) = \omega_2(1)\ \text{and}\ \omega_2' (1) = \omega_1(1).
\eeq
Consider now the loop $\Gamma$ formed by joining $\gamma$ and $\gamma'$, \emph{i.e.}
\[
\Gamma(t) := \begin{cases} \hfill \gamma'(2t) \hfill & \text{for}\ 0 \leq t \leq \tfrac{1}{2} \\
\hfill \gamma(2-2t) \hfill & \text{for}\ \tfrac{1}{2} \leq t \leq 1 \end{cases}
\]
By construction, we have that $F(\Gamma)$ is a loop in $\gr f$. Since $\Omega$ is simply connected we can continuously contract $\Gamma$ while staying in $\Omega$, \emph{i.e.} we have a continuous family $\{\Gamma_t\}_{0 \leq t \leq 1}$ of loops, all of which lie in $\Omega$ and such that $\Gamma_0 = \Gamma$ and $\Gamma_1$ is a single point $\{x_0\} \subset \Omega$. Now, using the Lipschitz continuity of $f$, the continuity of $\Gamma_t(s)$ in both variables, and the fact that a (two-valued) graph over a simply connected domain is simply connected, we get that $F(\Gamma_t)$ must also contract \emph{to a single point}. This means that $(x_0,f(x_0))$ is a single multiplicity two point, which means that the graph is not embedded at $(x_0,f(x_0))$. Since $x_0 \in \Omega$, this is a contradiction. Therefore the labelling we described must in fact be well-defined. We can therefore define two functions $f_1$ and $f_2$ on the whole of $\Omega$ as claimed. 

We extend $f_1$ and $f_2$ to the whole of $\pi U$, by using the facts that $\Omega$ is dense in $\pi U$ and $f$ is continuous. We claim that the graphs of $f_1$ and $f_2$ are both stationary in $U$. We have that $V_{f_1}$ - the varifold associated to $\gr f_1$ - is an integral $n$-varifold which is stationary away from $\cS$. However, since $\cH^{n-1}(\cS) = 0$ and we have the volume growth bound
\beq
\|V_{f_1}\|(B_{\rho}(X)) \leq c\rho^n\ \forall\ X \in \cS
\eeq
(which follows from the fact that $V_{f_1}$ is a Lipschitz graph), a standard cut-off argument implies that $\gr f_1$ is stationary. The same holds for the varifold associated to $\gr f_2$ and this completes the proof. $\hfill$ q.e.d.

\subsection{Excess Relative to a Single Plane}

What follows is a very important Lemma that bounds the excess relative to one plane by the excess relative to the pair of planes. It is elementary in the sense that it does not use stationarity; it is just a geometric fact about Lipschitz graphs.

\begin{lemma} \label{lemm:planlemm} Fix $\bC^{(0)} = |\bP_1^{(0)}| + |\bP_2^{(0)}| \in \cP$ and $L > 0$. There exists a number $\eps_0 = \eps_0(n,k,\bC^{(0)},L) \in (0,1)$ such that the following is true. If, for some $\eps < \eps_0$, we have that $\bC = |\bP_1| + |\bP_2| \in \cP$ and $V = V_f$ for $f \in C^{0,1}(B_2^n(0),\cA_2(\RR^k))$ with $\mathrm{Lip}\; f < L$ satisfy the following hypotheses:
\begin{enumerate}[nolistsep]
\item $\|V\|(B_2^n(0)\times \RR^k) \leq \|\bC^{(0)}\|(B_2^n(0)\times \RR^k) + 1/2$.
\item $0 \in A(\bC) \subset A(\bC^{(0)})$ and $d_{\cH}(\spt\|\bC\| \cap B_2(0),\spt\|\bC^{(0)}\| \cap B_2(0)) < \eps$.
\item $\int_{B_2(0)}\dist^2(X,\bP_1^{(0)})d\|V\|(X) < \eps$.
\end{enumerate}
Then,
\beq \label{planlemm1}
\int_{B_{1/2}(0)} \dist^2(X,\bP_1) d\|V\|(X) \leq c\int_{B_1(0)} \dist^2(X,\spt\|\bC\|) d\|V\|(X).
\eeq
for some $c = c(n,k,\bC^{(0)},L) > 0$.
\end{lemma}

\begin{proof} By tilting $\RR^n$ by an arbitrarily small amount (if necessary), we can assume that $\bP_1 = \gr p_1$ for a linear function $p_1 : \RR^n \to \RR^k$ the domain of which contains the domain as $f$. The idea of the proof is to slice $V$ into one-dimensional varifolds for which we can prove the result directly and then to use the coarea formula to reconstruct the full result.

Suppose that the hypotheses of the lemma are satisfied and let $\bQ$ be a hyperplane in $\RR^{n+k}$ containing $\bP_2$ and such that $\dim (\bP_1 \cap \bQ) = n-1$. Let $\cK \subset \bP_1$ be an $n$-dimensional cube in $\bP_1$ of edge length $1/(4\sqrt{n})$ centred at the origin and with sides parallel and perpendicular to $\bP_1 \cap \bQ$. Let $K := \pi \cK$ and notice that $K \Subset B_1^n(0)$. Define $M := \pi (\bP_1 \cap \bQ \cap \cK)$. Now for each $Y \in M$, let $\cL_Y'$ be the unique line that lies in $\bP_1$ with $\cL_Y' \perp (\bP_1 \cap \bQ)$ and $\cL_Y' \cap \bP_1 \cap \bQ = \{(Y,p_1(Y))\}$. Then define $\cL_Y := \cL_Y' \cap \cK$ and let $L_Y := \pi \cL_Y$. Finally, define $V_Y := V_{f\vert_{L_Y}}$.

\bigskip

\noindent \textbf{Step 1.} Define
\[
\delta = \delta(Y) := \sup_{X \in \spt\|V_Y\| \cap W_{1/2}} |X|,
\]
where $W_{1/2} :=  \{X \in \RR^{n+k} : 2\dist(X,\bQ) \leq \dist(X,\bP_1)\}$. We first claim that if $\eps$ is sufficiently small, then there exists $\gamma$ $=$ $\gamma(n,k,\bC^{(0)},L)$ $>$ $0$ such that for each $Y \in M$, either
\[
\spt \|V_Y\| \cap W_{1/2} = \emptyset
\]
or
 \beq \label{planlemm2}
 \|V_Y\|\bigl(\{X \in W : \dist(X,\bQ) \geq \gamma\delta \}\bigr) \geq \gamma\delta,
 \eeq
 where $W :=  \{X \in \RR^{n+k} : \dist(X,\bQ) \leq \dist(X,\bP_1)\}$. To prove this claim, take sequences $\{f^j\}_{j=1}^{\infty} \in C^{0,1}(B_2^n(0),\cA_2(\RR^k))$, $\{\bC^j\}_{j=1}^{\infty} \in \cP$, and $\{\eps_j\}_{j=1}^{\infty} \downarrow 0^+$ satisfying the hypotheses of the lemma with $f^j$, $\bC^j$ and $\eps_j$ in place of $f$, $\bC$ and $\eps$, respectively. Write $\bC^j = |\bP_1^j| + |\bP_2^j|$ in such a way that $\bP^j_i \to \bP_i^{(0)}$ as $j \to \infty$ for $i=1,2$. Define $M^j$,$K^j$, $L_Y^j$, $Q^j$ and $\delta_j$ analogously to their definitions above, but of course with $\bP_1$, $\bP_2$ and $V$ replaced by $\bP_1^j$, $\bP_2^j$ and $V^j := V_{f^j}$ respectively. Let $p_1^{(0)}$ and $p_1^j$ be the linear functions the graphs of which are equal to $\bP_1^{(0)}$ and $\bP_1^j$ for $j=1,2,...$ respectively. Now, assume for the sake of contradiction, that there exist $\{Y_j\}_{j=1}^{\infty} \in M_j$ and $\gamma_j \downarrow 0^+$ such that for sufficiently large $j$ we have $\spt \|V^j_{Y_j}\| \cap W^j_{1/2} \neq \emptyset$ and
 \beq \label{planlemm3}
\|V^j_{Y_j}\|(\{X \in W^j : \dist(X,\bQ^j) \geq \gamma_j\delta_j(Y_j) \}) < \gamma_j\delta_j(Y_j),
\eeq
where $W^j$ and $W^j_{1/2}$ are defined analogously to $W$ and $W_{1/2}$ but with $\bQ^j$ and $\bP^j$ in place of $\bQ$ and $\bP$. Now, hypotheses 2) and 3) imply that $f^j \to p_1^{(0)}$ pointwise as $j \to \infty$. Using the fact that $\{f^j\}_{j=1}^{\infty}$ is a sequence of Lipschitz functions with fixed Lipschitz constant, this implies (after passing to a subsequence) that $f^j \to p_1^{(0)}$ uniformly on compact subsets of $B_1^n(0)$. Of course hypothesis 2) also implies that $p_1^j \to p_1^{(0)}$ uniformly. So we have that $\sup_{U}|f^j - p_1^j| \to 0$ as $j \to \infty$ for all $U \Subset B_1^n(0)$. We may also assume that along this subsequence:
\begin{itemize}
\item $d_{\cH}(\bP_1^j \cap \bQ^j \cap B_1(0), A'\cap B_1(0)) \to 0$, where $A'$ is a fixed $(n-1)$-dimensional subspace of $\RR^{n+k}$.
\item $d_{\cH}(\bQ^j \cap B_2(0), \bQ^{(0)} \cap B_2(0)) \to 0$, where $\bQ^{(0)}$ is some $(n+k-1)$-dimensional subspace containing $\bP_2^{(0)}$,
\item $Y_j \to Y_0$ for some $Y_0 \in \pi(\bP_1^{(0)} \cap \bQ^{(0)} \cap \cK)$, so that $d_{\cH}(L_{Y_j},L_{Y_0}) \to 0$. Here $L_{Y_0}$ is a interval that contains $Y_0$ and lies in $\bP_1^{(0)}$.
\end{itemize} 
Let $\lambda \in (0,1)$ be such that for sufficiently large $j$, we have $L_{Y_j} \setminus \pi(B_{\lambda}(0)) \neq \emptyset$. Then set 
\[
\tilde{V}^j := T_{Y_j\; *}\eta_{Y_j,\delta_j(Y_j)/\lambda\; *} V^j_{Y_j}.
\]
Notice that $\tilde{V}^j = V_{\tilde{f}^j}$ for some $\tilde{f}^j \in C^{0,1}(L_{Y_j}; \cA_2(\RR^k))$ with $\mathrm{Lip}\; \tilde{f}^j \leq L$. Also, we observe that
 \beq \label{planlemm4}
\|\tilde{V}^j\|\bigl(\{X \in W^j : \dist(X, \bQ^j) \geq \gamma_j\lambda \}\bigr) < \gamma_j\lambda
\eeq
and
\beq \label{planlemm5}
\sup_{X \in \spt\|\tilde{V}^j\| \cap W^j_{1/2}} |X| = \lambda.
\eeq
Then, let $l_j : L_{Y_0} \to L_{Y_0}^{\perp}$ be the function that represents $L_{Y_j}$ as a graph over $L_{Y_0}$ and define $g^j(X) = \tilde{f}^j(X + l^j(X))$ for $X \in L_{Y_0}$. After passing to another subsequence, $g^j$ converges uniformly on compact subsets to some Lipschitz two-valued function $g$. From \eqref{planlemm4} and \eqref{planlemm5} we deduce that 
\begin{align}
\gr g \cap W^{(0)} \subset \bQ^{(0)} \cap \overline{B_{\lambda}(0)},
\end{align}
where $W^{(0)} := \{X \in \RR^{n+k} : \dist(X,\bQ^{(0)}) \leq \dist(X,\bP_1^{(0)})\}$. Also, \eqref{planlemm5} shows that there exists $X \in \gr g \cap W^{(0)}$ with $|X| = \lambda$, but since $L_{Y_0} \setminus \pi(B_{\lambda}(0)) \neq \emptyset$, this now means that 
\[
d_{\cH}\bigl(\gr g\vert_{\pi(\gr g \cap W^{(0)})} ,\ \gr g\vert_{L_{Y_0}\setminus \pi(\gr g \cap W^{(0)})}\bigr)  > 0
\]
which contradicts the fact that $g$ is a Lipschitz graph over $L_{Y_0}$. This proves the claim.

\medskip

\noindent \textbf{Step 2.} We can now prove the `$1$-dimensional' version of the lemma: Namely, if the hypotheses are satisfied for sufficiently small $\eps > 0$, then
\beq \label{planlemm6}
\int_{B_1(0)} \dist^2(X,\bP_1)d\|V_Y\|(X) \leq c\int_{B_1(0)}\dist^2(X,\bP_1 \cup \bQ)d\|V_Y\|(X),
\eeq
for every $Y \in M$, where $c = c(n,k,\bC^{(0)},L) > 0$. To see this we observe that for every $Y \in M$,
\beq \label{planlemm7}
\int_{W_{1/2} \cap B_1(0)} \dist^2(X,\bP_1) d\|V_Y\|(X) \leq c\delta^3,
\eeq
for some constant $c = c(n,k,\bC^{(0)},L) > 0$. Then notice that \eqref{planlemm2} of \textbf{Step 1.} implies that
\beq \label{planlemm8}
\int_{W_{1/2} \cap B_1(0)}\dist^2(X,\bP_1 \cup \bQ)d\|V_Y\|(X) \geq \gamma^3\delta^3.
\eeq
Putting these together proves \eqref{planlemm6} with a constant $c$ that does not depend on $Y$.

\medskip

\noindent \textbf{Step 3.} Now we can prove the full lemma. When the hypotheses are satisfied for appropriate $\eps > 0$, the previous two steps show that
\beq \label{planlemm9}
\int_{B_1(0)} \dist^2(X,\bP_1) d\|V_Y\|(X) \leq c\int_{B_1(0)} \dist^2(X,\bP_1 \cup \bQ) d\|V_Y\|(X)
\eeq
for every $Y \in M$. We also note that when $\eps$ is small enough we have that
\beq \label{planlemm10}
\spt\|V\| \cap B_{1/8n}(0) \subset \bigcup_{Y \in M} \spt\|V_Y\|.
\eeq
So, using a change of variables, the coarea formula, the uniform Lipschitz constant bound and another change of variables, we have that
\begin{align*}
 \int_{B_{1/8n}(0)}&\dist^2 (X,\bP_1)d\|V\|(X) \nonumber \\
& \leq  \sum_{\alpha = 1,2} \int_{\bigcup_{Y \in M}L_Y}\dist^2 ((X,f^{\alpha}(X)),\bP_1) \times \\
& \hspace{2cm} \det(\delta_{\alpha \beta} + \Sigma_{\kappa} D_{\alpha}f^{\alpha,\kappa}\cdot D_{\beta}f^{\alpha,\kappa})^{1/2}d\cH^n(X) \\
& \leq c\sum_{\alpha =1,2}\int_{M} \int_{L_Y}\dist^2((x,Y,f^{\alpha}(x,Y)),\bP_1) \times \nonumber \\
&  \hspace{2cm} \det(\delta_{\alpha \beta} + \Sigma_{\kappa} D_{\alpha}f^{\alpha,\kappa} \cdot D_{\beta}f^{\alpha,\kappa})^{1/2}d\cH^1(x)d\cH^{n-1}(Y) \\
&  \leq  c \int_{M} \int_{B_1(0)}\dist^2(X,\bP_1)d\|V_Y\|(X) d\cH^{n-1}(Y),
\end{align*}
where $c = c(n,k,\bC^{(0)},L) > 0$ and where $f^j(X) = \{f^{1}(X),f^{2}(X)\}$ and $f^{\alpha}(X) = (f^{\alpha,1}(X), \dots, f^{\alpha,k}(X))$. From here we can apply \eqref{planlemm9} to get that this is at most
\beq
c \int_{M} \int_{B_1(0)}\dist^2(X,\bP_1 \cup \bQ)d\|V_Y\|(X) d\cH^{n-1}(Y).
\eeq
Then we can again write the integral as an integral over the domain of $f$, apply the coarea formula and re-write as an integral with respect to $\|V\|$ in order to see that this is indeed at most
\beq
c\int_{B_1(0)} \dist^2(X,\bP_1 \cup \bQ) d\|V\|(X),
\eeq
which is in turn bounded above by the right-hand side of \eqref{planlemm1}. This proves \eqref{planlemm1} with an integral over $B_{1/8n}(0)$ on the left-hand side, but this can now be leveraged to prove the inequality as claimed. 
\end{proof}

\section{$L^2$ Estimates}

In this section we discuss in detail the main $L^2$ estimates. The proofs are technically involved and are deferred to Section 6. Suppose we have fixed $\bC^{(0)} \in \cC$.

\subsection{Specific Notation}
Recall that for $\bC \in \cC$ and $V \in \cV$, we define
 \begin{align*}
\cQ_V&(\bC) := \left(\int_{B_2^n(0)\times\RR^k}\dist^2(X,\spt \|\bC\|) \, d\|V\|(X)\right. \\
&+ \left.\int_{(B_2^n(0)\times\RR^k) \setminus \{r_{\bC^{(0)}} < 1/8\}}\dist^2(X,\spt \|V\|) \, d\|\bC\|(X)\right)^{1/2}.
\end{align*} 
We also define:
\begin{align*} 
E^2_V(\bC) & := \int_{B_1(0)}\dist^2(X,\spt\|\bC\|)d\|V\|(X) \\
\text{and}\quad a^{\bC}_V(X)& := \left(\sum_{j=1}^m|e_{l+k+j}^{\perp_{T_XV}}|^2\right)^{1/2},
\end{align*}
where $e_1,...,e_{n+k}$ is an orthonormal basis of $\RR^{n+k}$ in which $\bC$ is properly aligned (so that $e_{l+k+1},...,e_{n+k}$ is an orthonormal basis of $A(\bC)$).

\subsubsection{Comparing Nearby Cones}
For $\bC \in \cC$, the integer quantity $q_{\bC} := \dim A(\bC^{(0)}) - \dim A(\bC)$, which throughout will be non-negative, will play an important role: It will be the parameter with which we perform an induction argument in Section 6 in order to prove the main $L^2$ estimates. Given another $\bD \in \cC$, we define
\beq
\nu_{\bC,\bD} := d_{\cH}(\spt\|\bC\|\cap B_2(0),\spt\|\bD\| \cap B_2(0)).
\eeq
This will play a similar role to the $\nu$ defined on \cite[p. 908]{wickgeneral}.
\begin{remarks} \label{R:C0inP} \mbox{}
\begin{enumerate} [ref=\arabic*)]
\item The reason $q_{\bC}$ will always be non-negative is that if $\eps > 0$ is sufficiently small depending on $\bC^{(0)}$ then $\nu_{\bC,\bC^{(0)}} < \eps$ implies that $q_{\bC} \geq 0$.  
\item\label{en:C0inP2} If $\bC^{(0)} \notin \cP$, then there is some constant $c = c(n,k,\bC^{(0)}) > 0$ such that $\inf_{\bC^{\prime} \in \cP}\nu_{\bC^{(0)},\bC^{\prime}} \geq c > 0$. 
\item If $q_{\bC} > 0$, then $\bC \in \cP_{\leq n-2}$. Thus, if $\eps > 0$ is sufficiently small depending on $\bC^{(0)}$, then $\nu_{\bC,\bC^{(0)}} < \eps$ with $q_{\bC} > 0$ implies that $\bC^{(0)} \in \cP$.
\end{enumerate}
\end{remarks}
Fix $Z \in B_1(0)$ and write $\xi = Z^{\perp_{A(\bC)}}$. Using the invariance of $\bC$ under translations in directions along its axis, we have that for any $X \in B_1(0)$:
\begin{align} 
|\dist(X,\spt\|\bC\|) - \dist(X,\spt\|T_{Z *}\bC\|)|\  &\leq\  \nu_{\bC,T_{Z *}\bC} \label{E:CtoC_Z1}\\
& \leq |\xi |. \label{E:CtoC_Z2}
\end{align}

Also, if $0 \in A(\bC) \subsetneq A(\bC^{(0)})$ with $\bC^{(0)} = |\bP_1^{(0)}| + |\bP_2^{(0)}| \in \cP$, then we have that:
\begin{align} 
\nu_{\bC,T_{Z *}\bC}  &\leq   2\max_{i=1,2} d_{\cH}(\bP_i\cap B_2(0), T_{Z *}\bP_i \cap B_2(0)) \label{E:CtoC_Z3}\\
&\leq c\max_{i=1,2}\left[|\xi^{\perp_{\bP_i^{(0)}}}| + \nu_{\bC,\bC^{(0)}}|\xi^{\top_{\bP_i^{(0)}}}|\right].  \label{E:CtoC_Z4}
\end{align}
To get to \eqref{E:CtoC_Z4}, we have used the triangle inequality and the fact that $d_{\cH}(\bP_i\cap B_2(0), T_{Y *}\bP_i \cap B_2(0)) \leq c\nu_{\bC,\bC^{(0)}}|Y|$ for any $Y \in \bP_i^{(0)}$. This last inequality holds because if one were to write $\bP_i$ as a graph over $\bP_i^{(0)}$, the Lipschitz constant would be controlled by $\nu_{\bC,\bC^{(0)}}$. 

\subsubsection*{Hypotheses A}
We will often have the following hypotheses in place for some appropriate $\eps \in (0,1)$: 
\begin{enumerate}[nolistsep]
\item $\bC^{(0)} \in \cC$ and $V \in \cV$; $\|V\|(B_2^n(0)\times\RR^k) \leq \|\bC^{(0)}\|(B_2^n(0)\times\RR^k) + 1/2$.
\item $\bC \in \cC$ and $0 \in A(\bC) \subset A(\bC^{(0)})$.
\item $\nu_{\bC, \bC^{(0)}} < \eps$.
\item $\cQ_V(\bC^{(0)}) < \eps$.
\end{enumerate}
 
\subsection{Main Theorems}
We now state the main $L^2$ estimates. These theorems are analogous to Theorem 3.1 of \cite{simoncylindrical} and Theorems 10.1 and 16.2 of \cite{wickgeneral}. We assume throughout that we have fixed $\bC^{(0)} \in \cC$ and $L > 0$.

\begin{theorem} \label{L2} Fix $\tau \in (0,1)$. There exists $\eps_0 = \eps_0(n,k,\bC^{(0)},\tau,L) > 0$ such that the following is true. Suppose that for some $\eps < \eps_0$, we have that $V \in \cV_L$ and $\bC, \bC^{(0)} \in \cC$ satisfy Hypotheses A and that $\Theta_V(0) \geq 2$. Then we have the following conclusions:
\smallskip
\begin{enumerate}[nolistsep, label = \roman*),ref= \roman*)]
\item \label{en:L21} $\{X \in B_2(0) : \Theta_V(X) \geq 2\} \subset \{r_{\bC^{(0)}} < \tau\}$ and $V \res (B_{15/8}(0) \cap \{r_{\bC^{(0)}} > \tau \}) = |\gr (u + c)| \res  (B_{15/8}(0) \cap \{r_{\bC^{(0)}} > \tau \})$,
where $u \in C^{\infty}(\spt\|\bC^{(0)}\| \cap \{r_{\bC^{(0)}} > \tau/2\};\bC^{(0)\perp})$ and $c \in C^{\infty}(\spt\|\bC^{(0)}\| \cap \{r_{\bC^{(0)}} > 0\};\bC^{(0)\perp})$ is such that $\bC \res \{r_{\bC^{(0)}} > 0\} = |\gr c|$ and for $X \in \spt\|\bC^{(0)}\| \cap \{r_{\bC^{(0)}} > \tau/2\}$, $|u(X)| \leq (1 + c\eps)\dist(X + u(X) + c(X),\spt\|\bC\|)$ for some constant $c = c(n,k) > 0$.
\item \label{en:L22}
$\int_{B_{5/8}(0)} |a^{\bC}_V(X)|^2d\|V\|(X)  \leq  c\int_{B_1(0)} \dist^2(X,\spt \|\bC\|)d\|V\|(X)$.
\item \label{en:L23} $\int_{B_{5/8}(0)}\frac{\dist^2(X,\spt \|\bC\|)}{|X|^{n+7/4}} d\|V\|(X) \leq c\int_{B_1(0)} \dist^2(X,\spt \|\bC\|)d\|V\|(X)$.
\item \label{en:L24} $\int_{B_{5/8}(0)}\frac{|X^{\perp_{T_XV}}|^2}{|X|^{n+2}} d\|V\|(X) \leq c\int_{B_1(0)} \dist^2(X,\spt \|\bC\|)d\|V\|(X)$.
\item \label{en:L25} \begin{align*}
\int_{\Omega_{\tau/2} \cap B_{5/8}(0)} & R^{2-n}\left|\pa{((u(X) + c(X))/R)}{R}\right|^2d\cH^n(X) \nonumber \\ 
&\leq  c\int_{B_1(0)} \dist^2(X,\spt \|\bC\|)d\|V\|(X),
\end{align*}
where $\Omega_{\tau/2} := \spt\|\bC^{(0)}\| \cap \{r_{\bC^{(0)}} > \tau/2\}$
\end{enumerate}
\smallskip
In ii) to v), $c = c(n,k,\bC^{(0)},L) > 0$.
\end{theorem}

\begin{corollary} \label{cor:L2Z} For $\sigma \in (0,1)$ and each $\rho \in (0,1-\sigma]$ and, there exists $\eps_1 = \eps_1(n,k,\bC^{(0)},L,\rho,\sigma) > 0$ such that the following is true. Suppose that for some $\eps < \eps_1$, we have that $V \in \cV$ and $\bC, \bC^{(0)} \in \cC$ satisfy Hypotheses A and that $Z \in \spt\|V\| \cap B_{\sigma}(0)$ has $\Theta_V(Z) \geq 2$. Then we have the following conclusions: 
\smallskip
\begin{enumerate}[nolistsep, label = \roman*),ref= \roman*)]
\item \label{en:L2Z1}
$\dist^2(Z,A(\bC^{(0)})) \leq c\int_{B_1(0)}\dist^2(X,\spt\|\bC\|)d\|V\|(X)$.
\item \label{en:L2Z2} If $q_{\bC} = 0$, then $|\xi|^2 \leq c\int_{B_1(0)}\dist^2(X,\spt\|\bC\|)d\|V\|(X)$. If instead $q_{\bC} > 0$, then $\bC^{(0)} = |\bP_1^{(0)}| + |\bP_2^{(0)}| \in \cP$ and $|\xi^{\perp_{\bP_i^{(0)}}}|^2 + \nu_{\bC,\bC^{(0)}}^2|\xi^{\top_{\bP_i^{(0)}}}|^2 \leq c\int_{B_1(0)}\dist^2(X,\spt\|\bC\|)d\|V\|(X)$ for $i=1,2$, where in both cases $\xi = Z^{\perp_{A(\bC)}}$.
\item \label{en:L2Z3} $\int_{B_{5\rho/8}(Z)}\frac{\dist^2(X,\spt \|T_{Z*}\bC\|)}{|X - Z|^{n+7/4}} d\|V\|(X) $\\
$\leq c\rho^{-n-7/4}\int_{B_{\rho}(Z)} \dist^2(X,\spt \|T_{Z*}\bC\|)d\|V\|(X)$.
\item \label{en:L2Z4} $\int_{B_{5\rho/8}(Z)}\frac{\dist^2(X,\spt \|\bC\|)}{|X - Z|^{n-1/4}} d\|V\|(X) $\\
$\leq c\int_{B_1(0)} \dist^2(X,\spt \|\bC\|)d\|V\|(X)$.
\item \label{en:L2Z5} \begin{align*}
\int_{\Omega_{\tau/2} \cap B_{5\rho/8}(Z)}& \bar{R}_Z^{2-n}\left|\pa{((u(X) + c(X) - Z^{\perp_{T_X\bC^{(0)}}})/\bar{R}_Z)}{\bar{R}_Z}\right|^2d\cH^n(X) \\ 
& \leq  c\rho^{-n-2}\int_{B_{\rho}(Z)}\dist^2(X,\spt \|\bC\|)d\|V\|(X),
\end{align*}
for some constant $c = c(n,k,\bC^{(0)},L,\sigma) > 0$, where $\bar{R}_Z = \bar{R}_Z(X) = |X-Z^{\top_{T_X\bC^{(0)}}}|$, where $u$ and $c$ are functions as in \ref{en:L21} of Theorem \ref{L2} and where $\Omega_{\tau/2}$ is as in \ref{en:L25} of Theorem \ref{L2}.
\end{enumerate}
\smallskip
In i) to iv), $c = c(n,k,\bC^{(0)},L) > 0$. 
\end{corollary}

\begin{remarks} \label{R:CtoC_Z} When the hypotheses of Corollary \ref{cor:L2Z} are satisfied, we can combine \eqref{E:CtoC_Z1} to \eqref{E:CtoC_Z4} with \ref{en:L2Z2} of Corollary \ref{cor:L2Z} to get that
\begin{align} \label{E:CtoC_Z5}
\nu^2_{\bC,T_{Z *}\bC} \leq c\int_{B_1(0)}\dist^2(X,\spt\|\bC\|)d\|V\|(X)
\end{align}
for some $c = c(n,k,L,\bC^{(0)}) > 0$. 
\end{remarks}

\section{The Blow-Up Class}

In this section, we will assume Theorem \ref{L2} and Corollary \ref{cor:L2Z} and use them to construct the blow-up class, which is a class of functions defined on $\bC^{(0)}$ that represents the linearised problem. We will say that a properly aligned cone $\bC^{(0)} \in \cC$, sequences $\{\bC^j\}_{j=1}^{\infty} \in \cC$ (with $q_{\bC^j}$ constant, equal to $q$, say) and $\{V^j\}_{j=1}^{\infty} \in \cV_L$, a sequence of real numbers $\{\eps_j\}_{j=1}^{\infty}$ with $\eps_j \downarrow 0^+$ and a relatively closed set $\cD \subset A(\bC^{(0)}) \cap B_2(0)$ satisfy Hypotheses $\dagger$ if the following hold.

\bigskip

\noindent \textbf{Hypotheses} $\dagger$

\begin{enumerate}[label = (\arabic*$\dagger$), ref=(\arabic*$\dagger$)]
\item \label{en:dagger1} For each $j$, we have that $V^j$, $\bC^j$ and $\bC^{(0)}$ satisfy Hypotheses A with $\eps_j$ in place of $\eps$.
\item \label{en:dagger2} $d_{\cH}(\cD_j \cap \overline{B_2(0)}, \cD \cap \overline{B_2(0)}) \to 0$ and $\cD_j \cap B_{1/16}(0) \neq \emptyset$ for every $j$, where $\cD_j := \{Z \in B^n_2(0)\times \RR^k : \Theta_{V^j}(Z) \geq 2 \}$.
\item \label{en:dagger3} If $\cD \cap B_1(0) \neq A(\bC^{(0)}) \cap B_1(0)$, then $\bC^j \in \cP$ for all $j$.
\end{enumerate}

When we have these hypotheses in place, we write $E_j := E_{V^j}(\bC^j)$, which we call the \emph{excess} of $V^j$ relative to $\bC^j$.

\begin{remarks} \label{rema:dagger} Suppose $\bC^{(0)}$, $\{\bC^j\}_{j=1}^{\infty}$, $\{V^j\}_{j=1}^{\infty}$,  $\{\eps_j\}_{j=1}^{\infty}$, and $\cD$ satisfy Hypotheses $\dagger$. We make the following observations.
\begin{enumerate}
\item If $\cD \cap \overline{B_1(0)} \neq A(\bC^{(0)}) \cap \overline{B_1(0)}$, then $\bC^{(0)} \in \cP$. To see this, consider $Z \in (A(\bC^{(0)}) \setminus \cD) \cap \overline{B_1(0)}$. By \ref{en:dagger2}, we know that there exists some $\delta = \delta(Z) > 0$ such that for sufficiently large $j$ we have 
\beq \label{persgap1}
B_{\delta}(Z) \cap \cD_j = \emptyset.
\eeq
Thus we have a decomposition $V^j \res B_{\delta}(Z) = V^j_1 + V^j_2$ of $V^j$ in $B_{\delta}(Z)$ as per Theorem \ref{thm:grapdeco}. Using the mass bound (1) of Hypotheses A) and the compactness theorem for stationary integral varifolds (\cite{allard}), we may pass to a subsequence (depending on $Z$) along which $V^j_i$ converges in the sense of varifolds to a stationary integral $n$-varifold $W_i$ in $B_{\delta}(Z)$ for $i=1,2$.  Since $V^j \res B_{\delta}(Z) \to \bC^{(0)} \res B_{\delta}(Z)$ (from 4) of Hypotheses A), we know that $W_1 + W_2 = \bC^{(0)} \res B_{\delta}(Z)$. By applying the Constancy Theorem (\cite[$\S$ 41]{simongmt}) on each of the connected components of $\spt\|\bC^{(0)}\| \cap \{r_{\bC^{(0)}} > \eta\}$ for arbitrary $\eta > 0$, we deduce that $W_i$ has constant multiplicity along each of the half-planes that constitute $\bC^{(0)}$. Since $W_i$ is itself stationary in $B_{\delta}(Z)$ this implies that there must be a plane $\bP^{(0)}_i$ for which $W_i = \bP^{(0)}_i \res B_{\delta}(Z)$ whence $\bC^{(0)} \res B_{\delta}(Z) = (|\bP^{(0)}_1| + |\bP^{(0)}_2|) \res B_{\delta}(Z)$ whence $\bC^{(0)} = |\bP^{(0)}_1| + |\bP^{(0)}_2|$, because $\bC^{(0)}$ is a cylindrical cone. 

\item For $\delta \in (0,1/8)$ and $\sigma \in (0,1)$, there exists $J = J(\delta,\sigma) \in \NN$ such that for all $j \geq J$, we have
\beq \label{nonconcnearD1}
\int_{(\cD)_{\delta} \cap B_{\sigma}(0)} \dist^2(X,\spt\|\bC^j\|)d\|V^j\|(X) \leq c\delta^{3/4} E_j^2,
\eeq
for some $c = c(n,k,\bC^{(0)},L,\sigma) > 0$. The proof of this claim follows the argument of \cite[Corollary 3.2(ii)]{simoncylindrical}, except that we use \ref{en:L2Z4} of Corollary \ref{cor:L2Z} instead of (i) of Theorem 3.1 therein. The minor difference is that instead of being able to cover the whole axis by balls, here we can only cover the set $\cD$ and the correspondingly weaker conclusion therefore follows naturally. 

\item The argument of 1) above shows generally that singularities concentrate near good density points in the following sense: Given $\eta > 0$ there is $\eps = \eps(n,k,\bC^{(0)},\eta) > 0$ such that if $V \in \cV$ satisfies $\cQ_V(\bC^{(0)}) < \eps$, then $\sing V \cap B_{1/2}(0) \subset (\{Z \in \spt\|V\| \cap B_{1/2}(0) : \Theta_V(Z) \geq 2\})_{\eta}$.
\end{enumerate}
\end{remarks} 
\wl

\subsection{Constructing Blow-Ups} Given parameters $\tau_0$, $\rho_0$ and $\sigma_0$ we define $\bar{\eps}$ as follows: Let $\bar{\eps}_i$ for $i=0,1$ be those constants the existence of which is asserted by Theorem \ref{L2} and Corollary \ref{cor:L2Z} when one takes $\tau = \tau_0$, $\rho = \rho_0$ and $\sigma = \sigma_0$ in the statements. Then set $\bar{\eps} = \min_i\bar{\eps}_i$. Now if we suppose that $\bC^{(0)}$, $\{\bC^j\}_{j=1}^{\infty}$, $\{V^j\}_{j=1}^{\infty}$, $\{\eps_j\}_{j=1}^{\infty}$ and $\cD$ satisfy Hypotheses $\dagger$, then we can pick $\tau_j, \rho_j$ and $\sigma_j \to 0$ sufficiently \emph{slowly} so as to ensure that $\eps_j < \bar{\eps}(n,k,\bC^{(0)},L,\tau_j,\rho_j,\sigma_j)$ for every $j$. Possibly after passing to a subsequence we have the following:

\begin{enumerate} 
\renewcommand{\theenumi}{\textbf{(\arabic{enumi}$_j$)}}
\renewcommand{\labelenumi}{\textbf{(\arabic{enumi}$_j$)}}

\item\label{en:1j} $\|V^j\|(B^n_2(0)\times\RR^k) \leq \|\bC^{(0)}\|(B_2^n(0)\times\RR^k) + 1/2$ for all $j$.
\item \label{en:2j} $0 \in A(\bC^j) \subset A(\bC^{(0)})$.
\item\label{en:3j} $\cQ_{V^j}(\bC^j) < \eps_j$ for all $j$.
\item\label{en:4j} $\nu_{\bC^j,\bC^{(0)}} < \eps_j$ for all $j$.

\item\label{5j} By Theorem \ref{thm:grapdeco}: For any $Y \in A(\bC^{(0)}) \cap B_{15/8}(0)$ and $\rho \in (0,1/4)$, such that $B_{\rho}(Y) \cap \cD_j = \emptyset$, we have the decomposition $V^j \res (B^n_{\rho}(Y)\times\RR^k) = V^j_1 + V^j_2$, where for $i=1,2$, $V^j_i$ is a minimal Lipschitz graph (which means that we also have $\dim_{\cH}(\sing V^j_i) \leq n-4$ for $i=1,2$).

\item\label{6j} By 1) of Remark \ref{rema:dagger}, \ref{en:dagger2} and \ref{en:dagger3}, a closed set $\cD \subset A(\bC^{(0)}) \cap B_2(0)$ for which $\cD \cap B_{1/16}(0) \neq \emptyset$, $d_{\cH}(\cD_j \cap \overline{B_2(0)}, \cD \cap \overline{B_2(0)}) \to 0$ and such that if $\cD \cap B_1(0) \subsetneq A(\bC^{(0)}) \cap B_1(0)$, then $\bC^{(0)} \in \cP$ and $\bC^j \in \cP$ for all $j$.

\item\label{7j} By Theorem \ref{L2}
\begin{align}
V^j \res (B_{15/8}(0) \cap &\{r_{\bC^{(0)}} > \tau_j \}) \nonumber \\
& = |\gr (u^j + c^j)| \res  (B_{15/8}(0) \cap \{r_{\bC^{(0)}} > \tau_j \}),  \label{7j1}
\end{align}
where $u^j \in C^{\infty}(\spt\|\bC^{(0)}\| \cap \{r_{\bC^{(0)}} > \tau_j/2\} \cap B_{31/16}(0);\bC^{(0)\perp})$, $c^j$ $:$ $\spt\|\bC^{(0)}\|\setminus A(\bC^{(0)}) \to \bC^{(0)\perp}$ is such that $\bC^j \res \{r_{\bC^{(0)}} > 0\} = |\gr c^j|$ and for $X \in \spt\|\bC^{(0)}\| \cap \{r_{\bC^{(0)}} > \tau_j/2\}$, \\ $\dist(X + u^j(X) + c^j(X),\spt\|\bC^j\|) \geq (1 + c\eps_j)|u^j(X)|$ for some constant $c = c(n,k) > 0$.

\item\label{8j} By 2) of Remark \ref{rema:dagger}, \emph{i.e.} by \eqref{nonconcnearD1}: For any $\delta \in (0,1/8)$ and $\sigma \in (0,1)$, there exists $J(\delta,\sigma) \in \NN$ such that for $j \geq J$:
\beq  \label{8j1} 
\int_{(\cD)_{\delta} \cap B_{\sigma}(0)}\dist^2(X,\spt\|\bC^j\|) d\|V^j\|(X) \leq c\delta^{3/4}E_j^2,
\eeq
for some $c = c(n,k,\bC^{(0)},L) > 0$.
\item\label{en:9j} By \ref{en:L2Z1} to \ref{en:L2Z5} of Corollary \ref{cor:L2Z}: Given $Z_j \in \cD_j \cap B_{1/4}(0)$ and $\rho \in (0,1/4]$, there exists an integer $J(\rho)$ such that for $j \geq J(\rho)$ we have 
\beq \label{9j1} 
\dist^2(Z_j,A(\bC^{(0)})) \leq cE_j^2,
\eeq
\beq \label{9j2}
\begin{aligned} 
|\xi_j |^2 \leq cE_j^2 \qquad &\text{if}\ q_{\bC^j} \equiv q = 0, \\
\max_{i=1,2}\Bigl[ |\xi_j^{\perp_{\bP_i^{(0)}}}|^2 + \nu_{\bC^j,\bC^{(0)}}^2|\xi_j^{\top_{\bP_i^{(0)}}}|^2\Bigr] \leq cE_j^2\qquad &\text{if}\ q_{\bC^j} \equiv q > 0,
\end{aligned}
\eeq
\begin{align}
&\int_{B_{5\rho/8}(Z_j)}\frac{\dist^2(X,\spt\|T_{Z_j *}\bC^j\|)}{|X-Z_j|^{n+7/4}}d\|V^j\|(X) \leq cE_j^2, \label{9j3}\\
&\int_{B_{5\rho/8}(Z_j)}\frac{\dist^2(X,\spt\|T_{Z_j *}\bC^j\|)}{|X-Z_j|^{n+7/4}}d\|V^j\|(X) \label{9j4} \\
&\hspace{2cm} \leq c\rho^{-n-7/4}\int_{B_{\rho}(Z_j)} \dist^2(X,\spt\|T_{Z_j *}\bC^j\|)d\|V^j\|(X), \nonumber\\
\text{and} \nonumber \\
& \int_{\Omega_{\tau_j} \cap B_{5\rho/8}(Z_j)} R_{Z_j}^{2-n}\left|\pa{((u^j(X) + c^j(X) - Z_j^{\perp_{T_X\bC^{(0)}}})/R_{Z_j})}{R_{Z_j}}\right|^2d\cH^n(X) \label{9j5}  \\ 
& \hspace{2cm} \leq  c\rho^{-n-2}\int_{B_{\rho}(Z_j)}\dist^2(X,\spt \|\bC^j\|)d\|V^j\|(X), \nonumber
\end{align}
for some constant $c = c(n,k,\bC^{(0)},L) > 0$, where $R_Z = R_Z(X) = |X-Z^{\top_{T_X\bC^{(0)}}}|$ and where $\xi_j := Z_j^{\perp_{A(\bC^j)}}$.
\item \label{en:10j}
From \eqref{E:CtoC_Z5} of Remark \ref{R:CtoC_Z}: Given $Z_j \in \cD_j \cap B_{1/4}(0)$ we have 
\beq \label{10j1}
\nu_{\bC^j,T_{Z_j *}\bC^j}^2 \leq cE_j^2.
\eeq
\end{enumerate}

\wl

\noindent In the rest of this section, we construct the blow-up class $\fB(\bC^{(0)})$ and prove a list of basic properties. Suppose first that $\cD \cap B_1(0) = A(\bC^{(0)}) \cap B_1(0)$. Extend each function $u^j$ to all of $\spt \|\bC^{(0)}\| \cap \{r_{\bC^{(0)}} > 0\} \cap B_{31/16}(0)$ by defining its values to be zero on $\spt\|\bC^{(0)}\| \cap \{0 < r_{\bC^{(0)}} < \tau_j\}$. By \eqref{7j1} and elliptic estimates, there exists a harmonic function $v : \spt \|\bC^{(0)}\| \cap \{r_{\bC^{(0)}} > 0\} \cap B_1(0) \to \bC^{(0)\perp}$ such that 
\beq \label{consblowup1}
E_j^{-1}u^j \to v,
\eeq
where the convergence is in $C^2(K)$ for every compact subset of the domain of $v$. Then, using \eqref{8j1} of \ref{8j}, we deduce that for $\sigma \in (0,1)$, sufficiently small $\delta$ and sufficiently large $j$ depending on $\delta$ and $\sigma$, we have
\beq \label{consblowup2}
\int_{\spt\|\bC^{(0)}\| \cap \{0 < r_{\bC^{(0)}} < \delta\} \cap B_{\sigma}(0)}|E_j^{-1}u^j|^2 d\cH^n \leq c\delta^{3/4},
\eeq
from which we deduce that the convergence in \eqref{consblowup1} is also in $L^2(\spt\|\bC^{(0)}\| \cap \{r_{\bC^{(0)}} > 0\} \cap B_{\sigma}(0); \bC^{(0)\perp})$ for every $\sigma \in (0,1)$. 

Now suppose instead that $\cD \cap B_1(0) \subsetneq A(\bC^{(0)}) \cap B_1(0)$. From \ref{6j} we have that $\bC^{(0)} = |\bP_1^{(0)}| + |\bP_2^{(0)}| \in \cP$ and $\bC^j = |\bP_1^j| + |\bP_2^j| \in \cP$, which we label so that $\bP_i^j \to \bP_i^{(0)}$ as $j \to \infty$. Define $r_{\cD} = r_{\cD}(X) := \dist(X,\cD)$. Note that we can extend the domain of definition of $c^j_i := c^j\vert_{\bP_i^{(0)}}$ over the axis $A(\bC^{(0)})$ so that $\bP_i^j = \gr c^j_i$. Fix a small number $\tau_0 \in (0,1/64)$ and using \ref{6j}, choose $j$ sufficiently large (depending on $\tau_0$) such that $d_{\cH}(\cD_j \cap \overline{B_2(0)}, \cD \cap \overline{B_2(0)}) < \tau_0/2$. Now for $Y \in A(\bC^{(0)}) \cap B_{31/32}(0) \cap \{r_{\cD} > \tau_0\}$, we have that $B_{\tau_0/2}(Y) \cap \cD_j = \emptyset$ and so (as per \ref{5j}), we have the decomposition $V^j \res B_{\tau_0/2}(Y) = V_1^j + V_2^j$ for each $j$, where $V^j_i \to \bP_i^{(0)} \res B_{\tau_0/2}(Y)$ as $j \to \infty$, for $i=1,2$. By covering $\{r_{\bC^{(0)}} < \tau_0/8\} \cap \{r_{\cD} > \tau_0\} \cap B_{31/32}(0)$ by a finite collection of balls $\{B_{\tau_0/2}(Y_p)\}_{p=1}^M$ where $M = M(n,\tau_0)$, performing this argument at each point $Y_p$ for $p=1,..,M$ and using Allard's Regularity Theorem on each of the single-valued minimal graphs obtained, we have that (for sufficiently large $j$, depending on $\tau_0$) we can extend $u^j_i := u^j\vert_{\bP_i^{(0)}}$ for $i=1,2$ to a smooth function $u^j_i \in C^{\infty}(\bP_i^{(0)} \cap \{r_{\cD} > \tau_0\} \cap B_{31/32}(0);\bP_i^{(0)\perp})$ such that
\beq
V^j\ \res\ (\{r_{\cD} > \tau_0\} \cap B_{15/8}(0))  = \sum_{i=1}^2 V^j_i
\eeq
where
\beq
V^j_i = |\gr (u^j_i + c^j_i)|\ \res\ (\{r_{\cD} > \tau_0\} \cap B_{15/8}(0)).
\eeq
Further extend each $u^j_i$ to be equal to zero everywhere inside the region $\bP_i^{(0)} \cap \{r_{\cD} \leq \tau_0\} \cap B_{15/8}(0)$. Now, using elliptic estimates and combining these with Lemma \ref{lemm:planlemm}, we have that for each $i=1,2$, there exists a harmonic function $v_i : \bP_i^{(0)} \cap \{r_{\cD} > 0\} \cap B_1(0) \to \bP_i^{(0)\perp}$ such that 
\beq \label{consblowup3}
E_j^{-1}u^j_i \to v_i,
\eeq
where the convergence is in $C^2(K)$ for every compact subset of the domain of $v_i$. We can then again use \ref{8j1} of \ref{8j} to deduce that the convergence is also in $L^2(\bP_i^{(0)} \cap \{r_{\cD} > 0\} \cap B_{\sigma}(0);\bP_i^{(0)\perp})$ for every $\sigma \in (0,1)$. We then define $v : \spt\|\bC^{(0)}\| \cap B_1(0) \to \bC^{(0)\perp}$ by $v\vert_{\bP_i^{(0)}} = v_i$.

\smallskip

\noindent Write $\Omega := \spt\|\bC^{(0)}\| \cap \{r_{\bC^{(0)}} > 0\} \cap B_1(0)$

\bigskip

\noindent \textbf{Definition.} Corresponding to $\bC^{(0)}$, $\{\bC^j\}_{j=1}^{\infty}$, $\{V^j\}_{j=1}^{\infty}$,  $\{\eps_j\}_{j=1}^{\infty}$, and $\cD$ satisfying Hypotheses $\dagger$, a function $v \in$ $L^2(\Omega; \bC^{(0)\perp})$ $\cap$ $C^{\infty}(\Omega; \bC^{(0)\perp})$ constructed in this way is called a \emph{blow-up} of the sequence $V^j$ off $\bC^{(0)}$ relative to $\bC^j$. We define $\fB(\bC^{(0)})$ to be the class of all blow-ups off $\bC^{(0)}$.


\subsection{Properties of Blow-Ups}

We now prove that the class of functions $\fB(\bC^{(0)})$ satisfies certain fundamental properties that will enable us in the next section to prove that they exhibit quantitative $C^{1,\alpha}$ regularity properties.

\wl

\noindent \textbf{Definition.} Given a properly aligned cone $\bC^{(0)} \in \cC$, the class of functions $\cH(\bC^{(0)})$ is defined as follows. If $\bC^{(0)} \in \cC_{n-1}$, then it consists of functions $\psi : \spt\|\bC^{(0)}\| \cap \{r_{\bC^{(0)}} > 0 \} \to \bC^{(0)\perp}$ of the following form: For some collection of vectors $c_1,...,c_{n-1} \in A(\bC^{(0)})^{\perp}$ and a function $\varphi : \{\omega_1,...,\omega_4\} \to \bC^{(0)\perp}$ with $\varphi(\omega_j) \in T_{(\omega_j,0)}^{\perp}\bC^{(0)}$ for $j=1,...,4$ we have
\beq \label{Hdefn}
\psi(X) = \psi(x,y) = \sum_{p=1}^{n-1} y^pc_p^{\perp_{T_X\bC^{(0)}}} + |x|\varphi(x/|x|).
\eeq
If $\bC^{(0)} \in \cP_{\leq n-2}$ then it consists of functions $\psi : \spt\|\bC^{(0)}\| \cap \{r_{\bC^{(0)}} > 0 \} \to \bC^{(0)\perp}$ for which $\psi\vert_{\bP_i^{(0)}}$ is linear for $i=1,2$.

\begin{remarks} \label{rema:Hasblowup} When $\bC^{(0)} \in \cC_{n-1}$, the class $\cH(\bC^{(0)})$ accounts for all blow-ups of sequences of cones $\{\bD^j\}_{j=1}^{\infty} \in \cC_{n-1}$ of the form $\bD^j = R^j_*\tilde{\bD}^j$ where $R^j$ are rotations with $R^j \to \mathrm{id}_{\RR^{n+k}}$ and $\tilde{\bD}^j$ has $A(\tilde{\bD}^j) = A(\bC^{(0)})$ for all $j$. See section 2 of \cite{simoncylindrical}. \end{remarks}

The rest of this section is devoted to the proof of the following theorem.
 
\begin{theorem} \label{thm:propblowup} For a properly aligned cone $\bC^{(0)} \in \cC$, the class $\fB = \fB(\bC^{(0)})$ satisfies the following properties. 

\begin{enumerate} 
[nolistsep, ref = ($\fB$\arabic*), label= ($\fB$\arabic*)]
\item\label{B1} $v \in L^2(\Omega; \bC^{(0)\perp}) \cap C^{\infty}(\Omega; \bC^{(0)\perp})$

\item\label{B2} $\Delta v = 0$ on $\Omega$.

\item\label{B3} For each $v \in \fB$ there is a distinguished closed set $\cD_v \subset A(\bC^{(0)}) \cap B_1(0)$ with $\cD_v \cap B_{1/16}(0) \neq \emptyset$ such that if $\cD_v \neq A(\bC^{(0)}) \cap B_1(0)$, then $\bC^{(0)} \in \cP$ and $v_i := v \vert_{\Omega \cap \bP_i^{(0)}}$ extends to a smooth, $L^2$ harmonic function on $(\bP_i^{(0)} \setminus \cD_v) \cap B_1(0)$ for $i=1,2$.

\item\label{B4} When $\bC^{(0)} \in \cC_{n-1}$, we have that 
\[
\sup_{|y| \leq 3/8}\left|\frac{\partial^2}{\partial r_{\bC^{(0)}} \partial y^p}\sum_{j=1}^4 v(r_{\bC^{(0)}}\omega_j,y)\right| \to 0
\]
as $r_{\bC^{(0)}} \downarrow 0^+$, for $p=1,...,n-1$.

\item\label{B5} For any $v \in \fB$, we have the following closure and compactness properties:
\begin{enumerate}
[nolistsep, ref = ($\fB$5\Roman*), label= ($\fB$5\Roman*)]
\item \label{B5I} $\tilde{v}_{Y,\rho}(X) := \|v(Y + \rho(\cdot))\|_{L^2(\Omega)}^{-1}v(Y + \rho X) \in \fB$ for any $Y \in A(\bC^{(0)}) \cap B_{1/2}(0)$ and $\rho \in (0,1/4(1/2-|Y|)]$ for which $v(Y + \rho(\cdot)) \not\equiv 0$.
\item  \label{B5II}$\|v  - \kappa^{\perp_{T_{(\cdot)}\bC^{(0)}}} - \psi\|_{L^2(\Omega)}^{-1}(v - \kappa^{\perp_{T_{(\cdot)}\bC^{(0)}}} - \psi)  \in \fB$ for any $\kappa \in A(\bC^{(0)})^{\perp}\times\{0\}^m$ and any $\psi \in \cH(\bC^{(0)})$ such that $v  - \kappa^{\perp_{T_{(\cdot)}\bC^{(0)}}} - \psi \not\equiv 0$
\item  \label{B5III} For any sequence $\{v^j\}_{j=1}^{\infty} \in \fB$, there exists a subsequence $\{j'\}$ of $\{j\}$ and some $v \in \fB$ such that $v^{j'} \to v$ in $C^2_{loc}(\spt\|\bC^{(0)}\| \cap \{\dist(\cdot,\cD_v) > 0 \} \cap B_1(0);\bC^{(0)\perp})$.
\end{enumerate} 

\item\label{B6} For every $Z \in \cD_v \cap B_{1/4}(0)$, there exists $\kappa_v(Z) \in A(\bC^{(0)})^{\perp}$ satisfying $|\kappa_v(Z)|^2 \leq c\int_{\Omega} |v|^2$ for some $c = c(n,k,\bC^{(0)},L) > 0$ and such that for all $\rho \in (0,1/4]$ we have the estimates

\begin{align}
& \int_{\Omega \cap B_{\rho/2}(Z)} \frac{|v(X) - \kappa_v(Z)^{\perp_{T_X\bC^{(0)}}}|^2}{|X-Y|^{n+7/4}}d\cH^n(X) \nonumber \\
& \hspace{2cm} \leq c\int_{\Omega} |v(X)|^2 d\cH^n(X), \label{B6I}  \\
& \int_{\Omega \cap B_{\rho/2}(Z)}  \frac{|v(X) - \kappa_v(Z)^{\perp_{T_X\bC^{(0)}}}|^2}{|X-Y|^{n+7/4}}d\cH^n(X) \nonumber \\
& \hspace{2cm} \leq c\rho^{-n-7/4}\int_{\Omega \cap B_{\rho}(Z)} |v(X) - \kappa_v(Z)^{\perp_{T_X\bC^{(0)}}}|^2 d\cH^n(X), \label{B6II} \\
\text{and} \nonumber \\
& \int_{\Omega \cap B_{\rho/2}(Z)} R_Z^{2-n}  \left|\pa{((v(X) - \kappa_v(Z)^{\perp_{T_X\bC^{(0)}}})/R_Z)}{R_Z}\right|^2d\cH^n(X) \nonumber \\
& \hspace{2cm} \leq c\rho^{-n-2}\int_{\Omega \cap B_{\rho}(Z)} |v(X)|^2 d\cH^n(X),  \label{B6III}
\end{align}
where here $R_Z = R_Z(X) = |X-Z|$.
\end{enumerate}
\end{theorem}

\begin{proof} It is clear from the construction in Section 4.1 that \ref{B1} and \ref{B2} hold. And \ref{B3} follows from 1) of Remark \ref{rema:dagger} once we set $\cD_v := \cD \cap B_1(0)$, where $\cD$ is as in \ref{en:dagger2}.

Proof of \ref{B4}: This is proved using the argument of the proof of Lemma 1 of \cite{simoncylindrical} from equation (16) therein until the end. The only significant difference is that in \cite{simoncylindrical} at line (20), Theorem 3.1 was used, whereas here we must use estimate \ref{en:L22} of Theorem \ref{L2} applied to $\eta_{Z_j,9/10\ *}V^j$ for some $Z_j \in \cD_j \cap B_{1/16}(0)$ (the existence of which is guaranteed by \ref{B3}).

Proof of \ref{B5}: Firstly, if $v \in \fB$ is not identically zero, then for any $Y \in A(\bC^{(0)}) \cap B_{1/2}(0)$ and $\rho \in (0,1/4(1/2-|Y|)]$, we have that $\tilde{v}_{Y,\rho}$ is a blow-up of $\{(\eta_{Y,\rho})_*V^j\}_{j=1}^{\infty}$ off $\bC^{(0)}$ relative to $\{\bC^j\}_{j=1}^{\infty}$. This establishes \ref{B5I}.

Now, if $\bC^{(0)} \in \cP_{\leq n-2}$ and we are given $v \in \fB$ and $\psi \in \cH(\bC^{(0)})$, then firstly let $\hat{\bC}^j$ be the unique element of $\cP$ which contains the graph of $c^j + E_j\psi$ (where $c^j$ is the function that graphically represents $\bC^j$ over $\bC^{(0)}$). Secondly, for $\kappa \in A(\bC^{(0)})^{\perp}$, we replace the sequence $\{V^j\}_{j=1}^{\infty}$ with $\{\tau_{E_j\kappa *}V^j\}_{j=1}^{\infty}$. One can then check that Hypotheses $\dagger$ are still satisfied and that $\|v - \kappa^{\perp_{T_X\bC^{(0)}}} - \psi\|_{L^2(\Omega)}^{-1}(v - \kappa^{\perp_{T_X\bC^{(0)}}} - \psi)$ is a blow-up of $\tau_{E_j\kappa *}V^j$ off $\bC^{(0)}$ relative to $\hat{\bC^j}$.

Now suppose that $\bC^{(0)} \in \cC_{n-1}$ and that we are given $\psi \in \cH(\bC^{(0)})$. Let $\bD^j = R^j_*\tilde{\bD}^j$ be as in Remark \ref{rema:Hasblowup}. Let $d_j$ be the function that represents $\tilde{\bD}^j$ as a graph over $\bC^{(0)}$ and then let $\hat{\bC}^j$ be the unique element of $\cC_{n-1}$ that contains the graph of $c^j + E_jd_j$ (where $c^j$ is the function that graphically represents $\bC^j$ over $\bC^{(0)}$). If we are also given $\kappa \in A(\bC^{(0)})^{\perp}$, we replace the sequence $\{V^j\}_{j=1}^{\infty}$ by $\tilde{V}^j := \tau_{E_j\kappa *}(R^j)^{-1}_*V^j$ and again the result is that $\|v - \kappa^{\perp_{T_X\bC^{(0)}}} - \psi\|_{L^2(\Omega)}^{-1}(v - \kappa^{\perp_{T_X\bC^{(0)}}} - \psi)$ is a blow-up of $\{\tilde{V}^j\}_{j=1}^{\infty}$ off $\bC^{(0)}$ relative to $\{\hat{\bC}^j\}_{j=1}^{\infty}$.
  
To see \ref{B5III}, suppose that for each $j$, we have that $v^j$ is the blow-up of $\{V^p_j\}_{p=1}^{\infty}$ relative to $\{\bC^p_j\}_{p=1}^{\infty}$. For each $j$, notice that we can choose $p_j$ such that $\{p_j\}_{j=1}^{\infty}$ is strictly increasing and such that 
\beq \label{propblowup1}
\|(E_{V^{p_j}_j}(\bC^p_j)^{-1}u^{p_j}_j - v_j\|_{L^2(\Omega)} < j^{-1},
\eeq
where $u^{p}_j$ is the function that represents $V^p_j$ as a graph over as per \ref{7j}. That this is possible is clear from the construction of the blow-up. We then select a further subsequence of the $\{V^{p_j}_j\}_{j=1}^{\infty}$ to ensure that $E_{V^{p_j}_j}(\bC^p_j) \to 0$ as $j \to \infty$. Now, with $\cD'_j = \{X \in B_2(0) : \Theta_{V^{p_j}_j}(X) \geq 2\}$, we construct $\cD$ by using the sequential compactness of the space of closed sets with the Hausdorff metric: This means we can pass to a subsequence for which $\cD'_j \cap \overline{B_2(0)}$ converges in the Hausdorff metric to $\cD \subset \overline{B_2(0)}$. Then we choose $\eps_j$ such that $\bC^{(0)}$, $\{\eps_j\}_{j=1}^{\infty}$, $\{V^{p_j}_j\}_{j=1}^{\infty}$, $\cD$ and $\{\bC^{p_j}_j\}_{p=1}^{\infty}$ satisfy Hypotheses $\dagger$ and therefore we can define $v$ to be a blow-up of $V^{p_j}_j$ relative to $\bC^{p_j}_j$. Then using \eqref{propblowup1}, elliptic estimates, the Arz\'{e}la-Ascol\'{i} theorem, a compact exhaustion and a diagonalisation, we deduce that along a further subsequence, $v_{j'} \to v$ locally in $C^2$ as required.

Proof of \ref{B6}: Let $Z \in \cD_v \cap B_{1/4}(0)$ and $\rho \in (0,1/2]$. Suppose that $Z_j \in \cD_j$ is such that $Z_j \to Z$ as $j \to \infty$ and pick $\rho' \in (\rho/2,\rho)$ so that for sufficiently large $j$ we have $B_{\rho/2}(Z) \subset B_{5\rho'/8}(Z_j) \subset B_{\rho'}(Z_j) \subset B_{\rho}(Z)$. Write $\xi_j := Z_j^{\perp_{A(\bC^j)}}$. Fix $\tau > 0$. Having applied \eqref{9j4} of \ref{en:9j} at the point $Z_j$ and at scale $\rho'$, we then use the area formula to write the graphical part of the integral on the left-hand side of \eqref{9j4} over the domain in $\bC^{(0)}$. The result is that for sufficiently large $j$ we have:
\begin{align}  
\int_{\Omega \cap \{r_{\cD_v} > \tau\} \cap B_{\rho/2}(Z)}&\frac{|u^j(X) - \xi_j^{\perp_{T_X\bC^{(0)}}}|^2}{|X + u^j(X) + c^j(X) - Z_j|^{n+7/4}}d\cH^n(X)  \nonumber \\
& \leq c\rho^{-n-7/4}\int_{B_{\rho}(Z)} \dist^2(X,\spt\|T_{Z_j *}\bC^j\|)d\|V^j\|(X) \label{propblowup2}
\end{align}
where $c = c(n,k,\bC^{(0)},L) > 0$. By splitting the domain of integration of the integral on the right-hand side of the above line into the complementary regions $B_{\rho}(Z) \cap \{r_{\cD_v} < \tau\}$ and $B_{\rho}(Z) \cap \{r_{\cD_v} > \tau\}$ and using \eqref{E:CtoC_Z1} on the second integral, we see that the right-hand side is at most
\begin{align}
c\rho^{-n-7/4}\int_{B_{\rho}(Z) \cap \{r_{\cD_v} > \tau\}}& \dist^2(X,\spt\|T_{Z_j *}\bC^j\|)d\|V^j\|(X)\nonumber \\
+ &c\rho^{-n-7/4}\int_{B_{\rho}(Z) \cap \{r_{\cD_v} < \tau\}} \dist^2(X,\spt\|\bC^j\|)d\|V^j\|(X)\nonumber \\
+ &c\rho^{-n-7/4}\nu^2_{\bC^j,T_{Z^j *}\bC^j}\; \times\; \|V^j\|(B_{\rho}(Z) \cap \{r_{\cD_v} < \tau\}). \label{propblowup3}
\end{align}
From \eqref{8j1} of \ref{8j} we have
\beq \label{propblowup4}
\rho^{-n-7/4}\int_{B_{\rho}(Z) \cap \{r_{\cD_v} < \tau\}} \dist^2(X,\spt\|\bC^j\|)d\|V^j\|(X) \leq \rho^{-n-7/4}\tau^{3/4}E_j^2,
\eeq
and using \eqref{10j1} of \ref{en:10j} we have
\beq \label{propblowup5}
\rho^{-n-7/4}\nu^2_{\bC^j,T_{Z^j *}\bC^j}\|V^j\|(B_{\rho}(Z) \cap \{r_{\cD_v} < \tau\}) \leq \rho^{-n-7/4}E_j^2 \tau .
\eeq
Also, using \eqref{9j2} of \ref{en:9j}, we have that if $\bC^{(0)} \notin \cP$, then $E_j^{-1}|\xi_j^{\perp_{\bH_i^{(0)}}}|$ is bounded for $i=1,\dots,4$ and if $\bC^{(0)} \in \cP$ then $E_j^{-1}|\xi_j^{\perp_{\bP_i^{(0)}}}|$ is bounded for $i=1,2$. Therefore, by taking an appropriate subsequential limit and using \eqref{9j1} of \ref{en:9j}, we can conclude that there exists $\kappa_v(Z) \in A(\bC^{(0)})^{\perp}$ with
\beq \label{propblowup6}
|\kappa_v(Z)| \leq c = c(n,k,\bC^{(0)},L)
\eeq
and such that $E_j^{-1}\xi_j^{\perp_{T_X\bC^{(0)}}} \to \kappa_v(Z)^{\perp_{T_X\bC^{(0)}}}$ as $j \to \infty$, for all $X \in \spt\|\bC^{(0)}\|\setminus A(\bC^{(0)})$. Now we divide \eqref{propblowup2} by $E_j^2$. And, using \eqref{propblowup3}, \eqref{propblowup4}, \eqref{propblowup5}, the strong $L^2$ convergence of $E_j^{-1}u^j$ to $v$, the $C^2_{loc}$ convergence of $c^j$ and $u^j$ to $0$, and the dominated convergence theorem, we can justify taking limits first along a subsequence as $j \to \infty$ and then as $\tau \downarrow 0^+$.  The result is 
\begin{align}
 \int_{\Omega \cap B_{\rho/2}(Z)} & \frac{|v(X) - \kappa_v(Z)^{\perp_{T_X\bC^{(0)}}}|^2}{|X-Z|^{n+7/4}}d\cH^n(X) \nonumber \\
& \leq c\rho^{-n-7/4}\int_{\Omega \cap B_{\rho}(Z)} |v(X) - \kappa_v(Z)^{\perp_{T_X\bC^{(0)}}}|^2 d\cH^n(X). \label{propblowup7}
\end{align}

By starting with \eqref{9j3} and performing a similar procedure we also have that
\begin{align}
 \int_{\Omega \cap B_{\rho/2}(Z)} \frac{|v(X) - \kappa_v(Z)^{\perp_{T_X\bC^{(0)}}}|^2}{|X-Z|^{n+7/4}}d\cH^n(X) \leq c\int_{\Omega} |v(X)|^2 d\cH^n(X). \label{propblowup8}
\end{align}
This latter estimate implies that $\kappa_v(Z)$ indeed depends only on $Z$ and the particular blow-up $v$. Finally, the estimate \eqref{B6III} is obtained similarly, by dividing \eqref{9j5} by $E_j^2$ and carefully justifying the taking of the limit as $j \to \infty$.
\end{proof}

\begin{remarks} \label{rema:psiasblowup} Suppose $v \in \fB(\bC^{(0)})$ arises as the blow up of $\{V^j\}_{j=1}^{\infty}$. The proof of \ref{B5II} implies the following:  Given $\psi \in \cH(\bC^{(0)})$ with $ \sup_{\Omega}|\psi|^2 \leq c\int_{\Omega}|v|^2$, there exist rotations $R^j$ with $|R^j - \mathrm{id}_{\RR^{n+k}}| \leq cE_j$ and a sequence $\hat{\bC}^j \in \cC$ such that $\|v - \psi\|_{L^2(\Omega)}^{-1}(v - \psi)$ arises as a blow-up of $\{R^j_*V^j\}_{j=1}^{\infty}$ off $\bC^{(0)}$ relative to $\{\hat{\bC}^j\}_{j=1}^{\infty}$.
\end{remarks}

\begin{remarks}  \label{rema:kappa=0} The way that the function $\kappa_v$ is defined in the proof of \ref{B6} implies the following: Suppose that $v$ arises as the blow up of $\{V^j\}_{j=1}^{\infty}$ relative to some sequence of cones $\{\bC^j\}_{j=1}^{\infty}$ for which the conclusions of Theorem \ref{L2} and Corollary \ref{cor:L2Z} hold with $V^j$ and $\bC^j$ in place of $V$ and $\bC$. Fix $Y \in \cD_v$. If there exists a sequence $Y_j \in A(\bC^j) \cap \spt\|V^j\|$ with $\Theta_{V^j}(Y_j) \geq 2$ for all $j$ and such that $Y_j \to Y$, then $\kappa_v(Y) = 0$. This is because in this case $\xi_j := Y_j^{\perp_{A(\bC^j)}} = 0$ for all $j$.
\end{remarks}

\section{Regularity of Blow-Ups} 

The content of this section is that blow-ups satisfy a quantitative $C^{1,\alpha}$ estimate. The first result is a non-concentration estimate for blow-ups.

\begin{lemma} \label{blowupnonconc} Suppose $\bC^{(0)} \in \cC_{n-1}$. For any $v \in \fB(\bC^{(0)})$, there exists a function $\kappa_v : \Omega \to A(\bC^{(0)})^{\perp}$ of the form $\kappa_v(X)= \kappa_v(r_{\bC^{(0)}},X^{\top_{A(\bC^{(0)}}})$ that satisfies $\sup_{\Omega} |\kappa_v|^2 \leq c\int_{\Omega} |v|^2d\cH^n$ for some $c = c(n,k,\bC^{(0)},L) > 0$ and is such that: 
\beq \label{blowupnonconc1}
\int_{\Omega \cap (\cD_v)_{\rho/4}} \frac{|v(X) - \kappa_v(X)^{\perp_{T_X\bC^{(0)}}}|^2}{\dist(X,\cD_v)^{5/2}}d\cH^n(X) \leq c\int_{\Omega} |v(X)|^2 d\cH^n(X) 
\eeq
for every $\rho \in (0,1/4]$, where $c = c(n,k,\bC^{(0)},L) > 0$. 
\end{lemma}

\begin{proof} Work in a basis in which $\bC^{(0)}$ is properly aligned. Fix $\rho \in (0,1/4)$. For each $(r,y) = (r_{\bC^{(0)}},y)$ with $r > 0$ define $\kappa_v(r,y) \in \RR^{1+k}\times\{y\}$ by
\beq \label{blowupnonconc2}
\sum_{j=1}^4 |v(r\omega_j,y) - \kappa_v(r,y)^{\perp_{T_{(\omega_j,0)}\bC^{(0)}}}|^2 = \inf_{\lambda} \sum_{j=1}^4 |v(r\omega_j,y) - \lambda^{\perp_{T_{(\omega_j,0)}\bC^{(0)}}}|^2,
\eeq
where the infimum is taken over $\lambda \in \RR^{l+k}\times\{0\}^{n-1}$ with $|\lambda|^2 \leq c\int_{\Omega} |v|^2$. By using \eqref{B6I} of Theorem \ref{thm:propblowup} and the definition of $\kappa_v$, it follows directly that for $Z \in \cD_v$ and $\sigma \in (0,\rho)$ we have
\beq \label{blowupnonconc3}
\sigma^{-n-7/4}\int_{\Omega \cap B_{\sigma}(Z)} |v - \kappa_v^{\perp_{T_X\bC^{(0)}}}|^2d\cH^n  \leq c\int_{\Omega}|v|^2d\cH^n.
\eeq
Then we cover $(\cD_v)_{\sigma/4} \cap B_{1/2}(0)$ with a collection of at most $c(n,k)\sigma^{-(n-1)}$ balls $\{B_{\sigma}(Z_j)\}$, where $Z_j \in \cD_v$ for each $j$ and sum up the integrals to get that
\beq \label{blowupnonconc4}
\sigma^{-11/4} \int_{\Omega \cap (\cD_v)_{\sigma/4}} |v - \kappa_v^{\perp_{T_X\bC^{(0)}}}|^2d\cH^n \leq c\int_{\Omega}|v|^2d\cH^n.
\eeq
When we multiply by $\sigma^{-3/4}$, integrate in $\sigma$ from $0$ to $\rho$ and use Fubini's theorem to carry out the $\sigma$ integral, a short computation gives \eqref{blowupnonconc1}.  \end{proof}

Boundedness and continuity of blow-ups also follows immediately from the basic properties:

\begin{lemma} \label{blowupbouncont} For any $v \in \fB(\bC^{(0)})$, the following statements hold:
\begin{enumerate}[nolistsep]
\item $\sup_{\Omega \cap B_{1/4}(0)}|v| \leq c$, for some constant $c = c(n,k,\bC^{(0)},L) > 0$.
\item If $\bC^{(0)} \in \cC_{n-1}\setminus \cP_{n-1}$, then for $j=1,..,4$, we have that $v\vert_{\bH_j^{(0)} \cap B_{1/4}(0)}$ extends continuously to the boundary portion $A(\bC^{(0)}) \cap \Omega$. If $\bC^{(0)} \in \cP$, then for $i=1,2$, we have that $v_i := v\vert_{\bP_i^{(0)}}$ extends continuously to the whole of $\bP_i^{(0)} \cap B_{1/4}(0)$.
\end{enumerate} 
\end{lemma}

\begin{proof} Suppose first that $\bC^{(0)} \in \cP$. For $Z \in \cD_v \cap B_{1/4}(0)$ and $0 < \sigma < \rho/2 < 1/4$, \eqref{B6I} of \ref{B6} of Theorem \ref{thm:propblowup} implies that 
\begin{align}
\sigma^{-n} \int_{\Omega \cap B_{\sigma}(Z)}& |v - \kappa_v(Z)^{\perp_{T_X\bC^{(0)}}}|^2 d\cH^n(X) \nonumber \\
& \leq c\left(\frac{\sigma}{\rho}\right)^{7/4} \rho^{-n} \int_{\Omega \cap B_{\rho}(Z)} |v - \kappa_v(Z)^{\perp_{T_X\bC^{(0)}}}|^2 d\cH^n(X). \label{blowupbouncont1}
\end{align}
Fix a rotation $\Gamma$ with the properties that $\Gamma(A(\bC^{(0)})) = A(\bC^{(0)})$ and $\Gamma(\bP_1^{(0)}) = \bP_2^{(0)}$. Now, for any $X_0 \in (\bP_1^{(0)} \setminus \cD_v) \cap B_{1/4}(0)$ and any constant vector $\lambda \in A(\bC^{(0)})^{\perp}$ we have (using the mean value property of harmonic functions) that
\begin{align}
\sigma^{-n}& \Big( \int_{\bP_1^{(0)} \cap B_{\sigma}(X_0)}  |v_1 - v_1(X_0)|^2d\cH^n \nonumber \\
&+ \int_{\bP_2^{(0)} \cap B_{\sigma}(\Gamma(X_0))}  |v_2 - v_2(\Gamma(X_0))|^2d\cH^n \Big)  \nonumber \\
& \leq c\left(\frac{\sigma}{\rho}\right)^2\rho^{-n}\left( \int_{\bP_1^{(0)} \cap B_{\rho}(X_0)}|v_1 - \lambda^{\perp_{\bP_1^{(0)}}}|^2 d\cH^n \right. \nonumber \\
& \hspace{4cm} + \left. \int_{\bP_2^{(0)} \cap B_{\rho}(\Gamma(X_0))}|v_2 - \lambda^{\perp_{\bP_2^{(0)}}}|^2 d\cH^n\right) ,  \label{blowupbouncont2}
\end{align}
for $0 < \sigma \leq \rho/2 \leq \tfrac{1}{2}\min\{\tfrac{1}{4},\dist(X_0,\cD_v)\}$ and where $c = c(n,k) > 0$. By elementary means, \eqref{blowupbouncont1} and \eqref{blowupbouncont2} can be leveraged (see \cite{wickgeneral}, \emph{e.g.} the proof of Lemma 4.3 or the proof of Lemma 12.1 from line (12.5) onwards, for details of such an argument) to yield the estimate
\begin{align}
\rho^{-n}& \left( \int_{\bP_1^{(0)} \cap B_{\rho}(Z)}  |v_1 - v_1(Z)|^2d\cH^n + \int_{\bP_2^{(0)} \cap B_{\rho}(\Gamma(Z))}  |v_2 - v_2(\Gamma(Z))|^2d\cH^n \right)  \nonumber \\
& \leq c\rho^{2\beta}\left( \int_{\bP_1^{(0)} \cap B_{1/2}(Z)}|v_1|^2 d\cH^n + \int_{\bP_2^{(0)} \cap B_{1/2}(\Gamma(Z))}|v_2|^2 d\cH^n\right),  \label{blowupbouncont3}
\end{align}
for any $\rho \in (0,\gamma]$, some fixed $\gamma = \gamma(n,k) > 0$ and $\beta = \beta(n,k,\bC^{(0)},L) \in (0,1)$. From here, it is standard (see \emph{e.g.} Lemma 1 of \cite{simontheorems}) that $v_i \in C^{0,\beta}(\bP_i^{(0)} \cap B_{1/4}(0);\bP_i^{(0)\perp})$ for $i=1,2$. 

If, on the other hand $\bC^{(0)} \in \cC_{n-1}\setminus \cP_{n-1}$, then a similar argument shows that $v\vert_{\bH_j^{(0)} \cap B_{1/4}(0)} \in C^{0,\beta}(\bH_j^{(0)} \cap B_{1/4}(0); \bH_j^{(0)\perp})$. In either case, 1) of the present Lemma also follows.
\end{proof}

\subsection{Homogeneous Degree One Blow-Ups}

Most of the work of this section goes into understanding the structure of a homogeneous degree one blow-up for which $\cH^{n-2}(\cD_v) = \infty$. Firstly, we must gain better information about the way in which $v$ decays to its values on $\cD_v$.

\begin{lemma} \label{decaytoD} Suppose $\bC^{(0)} \in \cC_{n-1}$ is properly aligned. For any $v \in \fB(\bC^{(0)})$ that is homogeneous degree one and satisfies $\cH^{n-2}(\cD_v \cap B_{1/8}(0)) = \infty$, there are vectors $c_p \in \RR^{1+k}\times\{0\}^{n-1}$ for $p=1,...,n-1$ such that for $j=1,..,4$ we have
\beq \label{decaytoD1}
\lim_{\rho \downarrow 0^+} \rho^{-5/4}\int_{\Omega_j \cap (\cD_{v})_{\rho} \cap B_{1/4}(0)}| v(r\omega_j,y) - \Sigma_{i=1}^{n-1}y^pc_p^{\perp_{\bH_j}}|^2\ d\cH^n(X) = 0,
\eeq
where $\Omega_j := \Omega \cap \bH_j^{(0)}$.
\end{lemma}

\begin{proof} This proof is based on the proof of Lemma 4.2 of \cite{simoncylindrical}. In light of \ref{B4}, we may apply the reflection principle for harmonic functions to the function
\beq
(r,y) \mapsto \frac{\partial^2}{\partial r \partial y^p}\sum_{j=1}^4v(r\omega_j,y),
\eeq
which is initially defined on the domain $(0,\infty)\times\RR^{n-1}$. We therefore deduce that the function
\beq \label{decaytoD2}
\Psi_{y_p}(r,y) := \pa{}{y^p}\sum_{j=1}^4 v(r\omega_j,y)
\eeq
extends via even reflection in the $r$-variable to a homogeneous degree zero harmonic function on the whole of $\RR^n$. Since such functions are necessarily constant and since this holds for each $p \in \{1,...,n-1\}$, we deduce that
\beq \label{decaytoD3}
\Psi(r,y) := \sum_{j=1}^4 v(r\omega_j,y) = ra + \sum_{p=1}^{n-1}y^pb_p
\eeq
for some $a, b_p \in \RR^{1+k}\times\{0\}^{n-1}$ (we have also used the fact that $v$ is homogeneous degree one to deduce the form of the dependence on the $r$ variable). From here, \eqref{blowupnonconc1} of Lemma \ref{blowupnonconc} implies that
\beq \label{decaytoD7}
\lim_{\rho \downarrow 0^+}\rho^{-5/2}
\int_{\Omega \cap (\cD_{v})_{\rho} \cap B_{1/4}(0)}\bigl|\Sigma_{p=1}^{n-1}y^pb_p - \sum_{j=1}^4 \kappa_v(r,y)^{\perp_{\bH_j^{(0)}}}\bigr|^2d\cH^n = 0.
\eeq
We claim that this means that each $b_p$ is in the subspace 
\[
T:= \Bigl\{ \sum_{j=1}^4 c^{\perp_{\bH_j^{(0)}}} : c \in \RR^{1+k}\times\{0\}^{n-1} \Bigr\}.
\]
This is equivalent to the claim that $S:= \text{span}\langle b_1,...,b_{n-1}\rangle \subset T$. So, suppose for the sake of contradiction that $S \nsubseteq T$ and write $L(x,y) = \Sigma_{p=1}^{n-1}y^pb_p$. Since $L$ does not depend on the $r$-variable, if $Y = (0,y) \in L^{-1}(T) \cap A(\bC^{(0)})$, then $(r\omega_j,y) \in L^{-1}(T)$ for all $r > 0$ and each $j=1,..,4$. Now, by assumption we have that $L^{-1}(T) \cap A(\bC^{(0)}) \cap B_{1/8}(0) \subsetneq A(\bC^{(0)}) \cap B_{1/8}(0)$, whence $\dim_{\cH}(L^{-1}(T) \cap A(\bC^{(0)})) \leq n-2$ (because $L^{-1}(T) \cap A(\bC^{(0)})$ is a subspace). And since $\cH^{n-2}(\cD_v \cap B_{1/8}(0)) = \infty$, we can find some subset $\cD' \subset \cD_v \cap B_{1/8}(0)$ with $\cH^{n-2}(\cD') > 0$ and $\dist(\cD',L^{-1}(T)) > 0$. Thus we know that $\delta := \inf_{(0,y) \in \cD'} \dist(L(x,y),T)$ is strictly positive. It follows from the definition of Hausdorff measure and the fact that $\cH^{n-2}(\cD') > 0$ that for all sufficiently small $\rho > 0$ we have an estimate $\cH^n(\Omega \cap (\cD')_{\rho} \cap B_{1/4}(0)) \geq c\rho^2$ for some $c = c(\cD',n) > 0$. Moreover, for sufficiently small $\rho > 0$, we have that $\dist(L(x,y),T) \geq \delta/2$ for all $(r,y) \in (\cD')_{\rho}$. Thus we can bound below by integrating only over $(\cD')_{\rho}$ to deduce that
\begin{align*}
& \rho^{-5/2} \int_{\Omega \cap (\cD_{v})_{\rho} \cap B_{1/4}(0)}\bigl|\Sigma_{i=1}^{n-1}y^pb_p - \sum_{j=1}^4 \kappa_v(r,y)^{\perp_{\bH_j^{(0)}}} \bigr|^2d\cH^n \\
& \geq  c\rho^{-5/2}(\delta/2)^2 \rho^2 \\
& \geq  c\rho^{-1/2}\ \to \infty
\end{align*}
as $\rho \downarrow 0^+$, which is a contradiction. Therefore $S \subset T$ and the claim is proved. 

So, we have that for each $p \in \{1,...,n-1\}$, there is some $c_p \in \RR^{1+k}\times\{0\}^{n-1}$ for which $b_p = \sum_{j=1}^4 c_p^{\perp_{\bH_j^{(0)}}}$ and this means that
\beq \label{decaytoD8}
\lim_{\rho \downarrow 0^+} \rho^{-5/2} \int_{\Omega \cap  (\cD_{v})_{\rho} \cap B_{1/4}(0)}\bigl|\sum_{j=1}^4 (\Sigma_{p=1}^{n-1}y^pc_p - \kappa_v(r,y))^{\perp_{\bH_j^{(0)}}}\bigr|^2d\cH^n = 0.
\eeq
From here the argument can be finished exactly as on page 622 of \cite{simoncylindrical}.

\end{proof}

We must introduce one more piece of useful terminology: Suppose that $\bC^{(0)} \in \cC_{n-1}$. Given $v \in \fB(\bC^{(0)})$ that is homogeneous degree one, $Z \in \cD_v \cap B_{1/4}(0)$ and $\rho \in (0,1/4]$ we say that the function $\psi$ \emph{dehomogenizes} $v$ in $B_{\rho}(Z)$ when $\psi(\cdot - Z) \in \cH(\bC^{(0)})$ and
\begin{align*}
\int_{\Omega \cap B_{\rho}(Z)} &|v(X) - \psi(X)|^2d\cH^n(X) \\
&= \inf_{l \in \cH(\bC^{(0)})}\int_{\Omega \cap B_{\rho}(Z) }|v(X) - l(Z + X)|^2d\cH^n(X).
\end{align*}
When $v$ satisfies
\beq \label{dehomo}
\inf_{l \in \cH(\bC^{(0)})}\int_{\Omega \cap B_{\rho}(Z)}|v(X) - l(Z + X)|^2d\cH^n = \int_{\Omega \cap B_{\rho}(Z)}|v(X)|^2d\cH^n(X),
\eeq
we say that $v$ is \emph{dehomogenized} in $B_{\rho}(Z)$. It is straightforward (using orthogonal projection in $L^2(\Omega \cap B_{\rho}(Z),\bC^{(0)\perp})$) to prove the existence of dehomogenizers and one can see (from \eqref{Hdefn}) that it is equivalent to being $L^2$-orthogonal to the functions $(r\omega,y) \mapsto r\varphi(\omega)$ and $(r\omega,y) \mapsto y^pe_j^{\perp_{T_{(r\omega,y)}\bC^{(0)}}}$ for $p=1,..,n-1$, $j=1,...,1+k$. We are now in a position to categorize homogeneous degree one blow-ups. The proof of the following theorem uses a modification of the proof of Proposition 4.2 of \cite{wickgeneral}.

\begin{theorem} \label{homblowup} Fix a properly aligned cone $\bC^{(0)} \in \cC$ and a homogeneous degree one blow-up $v \in \fB(\bC^{(0)})$. Then $v \in \cH(\bC^{(0)})$.
\end{theorem}

\begin{proof} \mbox{}

\noindent \textbf{Step 1.} \emph{The Negligible Part of $\cD_v$.} By \ref{B3}, if $\cH^{n-2}(\cD_v) < \infty$, then $\bC^{(0)} \in \cP$. This means that the set $\cD_v \cap B_{1/8}(0)$ is of zero 2-capacity and is therefore a removable set for the bounded harmonic function $v_i$. So $v_i$ can be extended to a homogeneous degree one harmonic function defined on the whole of $\bP_i^{(0)}$ and such functions are linear. In particular, $v \in \cH(\bC^{(0)})$. Thus from now on we may assume that $\cH^{n-2}(\cD_v \cap B_{1/8}(0)) = \infty$ (which means that $\bC^{(0)} \in \cC_{n-1}$). 

\bigskip

\noindent \textbf{Step 2.} \emph{The Thick Part of $\cD_v$.}  Let $\cT_v$ denote the set of points $Z \in \cD_v \cap B_{1/4}(0)$ for which $\cH^{n-1}(\cD_v \cap B_{\eta}(Z)) > 0$ for every $\eta > 0$. We claim that for all $Z \in \cT_v$, we have that 
\beq \label{homblowup12}
\kappa_v(Z)^{\perp_{\bH_j^{(0)}}} = \Sigma_{p=1}^{n-1}y^pc_p^{\perp_{\bH_j^{(0)}}},
\eeq
with $c_p$ as per \eqref{decaytoD1} of Lemma \ref{decaytoD}. Since
\begin{eqnarray*}
& & \rho^{-5/4}\int_{\Omega_j \cap (\cD_{v})_{\rho} \cap B_{1/4}(0)}| v(r\omega_j,y) - \Sigma_{p=1}^{n-1}y^pc_p^{\perp_{\bH_j^{(0)}}}|^2d\cH^n \\
&\geq & \rho^{-1}\int_{\bH_j^{(0)} \cap (B_{\rho}^{1+k}(0) \times \cT_v) \cap B_{1/4}(0)}| v(r\omega_j,y) - \Sigma_{p=1}^{n-1}y^pc_p^{\perp_{\bH_j^{(0)}}}|^2d\cH^n \\
& \geq & \rho^{-1}\int_0^{\rho} \int_{(B_r^{1+k}(0) \times \cT_v) \cap B_{1/4}(0)}| v(r\omega_j,y) - \Sigma_{p=1}^{n-1}y^pc_p^{\perp_{\bH_j^{(0)}}}|^2\, d\cH^{n-1}(y)\, dr,
\end{eqnarray*}
Lemma \ref{decaytoD} implies that this last expression goes to zero as $\rho \to 0$ and so by Lebesgue differentiation we conclude that 
\[
\lim_{r \to 0}v(r\omega_j,y) = \Sigma_{p=1}^{n-1}y^pc_p^{\perp_{\bH_j^{(0)}}}
\]
for $\cH^{n-1}$-almost every $y \in \cT_v$, \emph{i.e.} \eqref{homblowup12} holds at $\cH^{n-1}$-almost every point of $\cT_v$. But $v\vert_{\Omega_j}$ is continuous along $A(\bC^{(0)})$ (from Lemma \ref{blowupbouncont}) and using the definition of $\cT_v$, we can see that every point of $\cT_v$ is a limit point of a sequence along which \eqref{homblowup12} holds, which implies that \eqref{homblowup12} holds for all $Z \in \cT_v$.

Note that at this stage of the argument, if $\cD_v \cap B_{1/4}(0) = A(\bC^{(0)}) \cap B_{1/4}(0)$, then $\cT_v = \cD_v \cap B_{1/4}(0)$. Thus for any $j \in \{1,..,4\}$, the odd reflection of $v(r\omega_j,y) - \Sigma_{p=1}^{n-1}y^pc_p^{\perp_{\bH_j^{(0)}}}$ in the $r$-variable is entire, homogeneous degree one, harmonic and equal to zero on $\{r=0\}$. It follows that
\[
v(r\omega_j,y) = ra + \Sigma_{p=1}^{n-1}y^pc_p^{\perp_{\bH_j^{(0)}}}
\]
for some $a \in \bH_j^{(0)\perp}$, which proves exactly that $v \in \cH(\bC^{(0)})$. So, in light of Step 2., the remaining case is that in which $\bC^{(0)} \in \cP_{n-1}$, $\cD_v \cap B_{1/4}(0) \neq A(\bC^{(0)}) \cap B_{1/4}(0)$ and $\cH^{n-2}(\cD_v \cap B_{1/8}(0)) = \infty$.

\bigskip

\noindent \textbf{Step 3.} \emph{Setting up the Induction.} For any homogeneous degree one blow up $w$, we will write 
\[
S(w) = \{ Z \in A(\bC^{(0)})\ :\ w(X + Z) = w(X)\ \textit{for all}\ X \in \Omega \}.
\]
It is easy to verify, using the homogeneity of $w$, that $S(w)$ is always a linear subspace of $A(\bC^{(0)})$. We will prove, by induction on $d$, the following statement: If $v \in \fB(\bC^{(0)})$ is homogeneous degree one with $\cH^{n-2}(\cD_v) = \infty$ and has $\dim S(v) = n - d$, then $v \in \cH(\bC^{(0)})$. Note that when $d=1$, \emph{i.e.} when $\dim S(v) = n-1$, then $S(v) = A(\bC^{(0)})$, from which, using the homogeneity of $v$, we immediately deduce that $v \in \cH(\bC^{(0)})$. So, we now fix $d \geq 2$ and using the inductive hypothesis (together with the results of Steps 1. and 2.) we may assume that any homogeneous degree one blow up $w \in \fB(\bC^{(0)})$ with $\dim S(w) > n-d$ belongs to $\cH(\bC^{(0)})$. 

For $Z \in \cD_v$ and $\rho > 0$ let $\psi_{Z,\rho}$ be the function that dehomogenizes $v$ in $B_{\rho}(Z)$. Obviously we may assume that $v \notin \cH(\bC^{(0)})$ (or else there is nothing to prove), so that $v - \psi_{Z,\rho} \not\equiv 0$. And note that that since $\cH^{n-2}(\cD_v) = \infty$, we have that $\cD_v \setminus S(v) \neq \emptyset$. 

\bigskip

\noindent \textbf{Step 4.} \emph{The Reverse Hardt-Simon Inequality.} We claim that for any compact subset $K$ of $(A(\bC^{(0)}) \setminus S(v)) \cap B_{1/4}(0)$, there exists $\eps = \eps(v,K) \in (0,1)$ such that for any $Y \in \cD_v \cap K$ with $\kappa_v(Y) = 0$ and for any $\rho \in (0,\eps]$, we have
\begin{align} 
 \int_{\Omega \cap (B_{\rho}(Y)\setminus B_{\rho/2}(Y))} &R_Y^{2-n} \left|\pa{((v - \psi_{Y,\rho})/R_Y)}{R_Y}\right|^2d\cH^n \nonumber \\
& \geq  \eps \rho^{-n-2}\int_{\Omega \cap B_{\rho}(Y)} |v - \psi_{Y,\rho}|^2d\cH^n. \label{homblowup2}
\end{align}
If this were false, there would exist a sequence of points $\{Y_j\}_{j=1}^{\infty} \in \cD_{v} \cap K$ with $\kappa_{Y_j} = 0$ and $v - \psi_{Y_j,\rho_j} \not\equiv 0$ for all $j$, a sequence $\eps_j \downarrow 0^+$, and a sequence of radii $\rho_j \downarrow 0^+$ such that, writing $\psi^j := \psi_{Y_j,\rho_j}$, we have
\begin{align}
\int_{\Omega \cap (B_{\rho_j}(Y_j)\setminus B_{\rho_j/2}(Y_j))} & R_{Y_j}^{2-n} \left|\pa{((v - \psi^j)/R_{Y_j})}{R_{Y_j}}\right|^2d\cH^n \nonumber \\
& < \eps_j \rho_j^{-n-2}\int_{\Omega \cap B_{\rho_j}(Y_j)} |v - \psi^j|^2d\cH^n. \label{homblowup3}
\end{align}
By property \ref{B5II}, we have that $v^j_{\psi} := \|v - \psi^j\|_{L^2(\Omega)}^{-1}(v - \psi^j) \in \fB(\bC^{(0)})$ for each $j$ and then by property \ref{B5I}, we also have that 
\beq
w^j := \|v^j_{\psi}(Y_j + \rho_j(\cdot)\|^{-1}_{L^2(\Omega)}\, v^j_{\psi}(Y_j + \rho_j(\cdot)) \in \fB(\bC^{(0)})
\eeq
for each $j$. The result of these transformations is that $w^j$ is dehomogenized in $B_{1/4}(0)$ (as one can check using \eqref{dehomo}). Then, using \ref{B5III}, we have that there exists $w \in \fB(\bC^{(0)})$ and a subsequence $\{j'\}$ of $\{j\}$ (to which we pass to without changing notation) for which $w^j\vert_{\bP_i^{(0)}} \to w\vert_{\bP_i^{(0)}}$ in $C^2_{loc}(\bP_i^{(0)} \cap \{ \dist(\cdot,\cD_w) > 0 \} \cap B_1(0);\bP_i^{(0)\perp})$ for $i=1,2$. Assume also, by compactness of $K$, that along this subsequence we have $Y_j \to Y \in \cD_v \cap K$. Dividing \eqref{homblowup3} by $\rho_j^{-n-2}\int_{\Omega \cap B_{\rho_j}(Y_j)} |v - \psi^j|^2d\cH^n$ and making the appropriate substitutions in the integrals we see that
\beq
\int_{\Omega \cap (B_1(0)\setminus B_{1/2}(0))} R^{2-n} \left|\pa{(w^j/R)}{R}\right|^2 d\cH^n < \eps_j,
\eeq
which implies that
\beq \label{homblowup4}
\int_{\Omega \cap (B_1(0) \setminus B_{1/2}(0))} \left|\pa{(w/R)}{R}\right|^2d\cH^n = 0,
\eeq
which means that $w$ is homogeneous degree one on $\Omega \cap (B_1(0) \setminus B_{1/2}(0))$. And note that by unique continuation of harmonic functions, it is equal to its homogeneous degree one extension on $\Omega$. Since $w^j \to w$ weakly in $L^2(\Omega)$, one can also check that $w$ is dehomogenized in $B_1(0)$. 

Let us now see that $\dim S(w) > \dim S(v)$: Write $\mu_j = \|v^j_{\psi}(Y_j + \rho_j(\cdot)\|_{L^2(\Omega)}$. For each $X_0 \in \Omega$, sufficiently small $\sigma > 0$ and sufficiently large $j$, we have:
\begin{eqnarray*}
& & \sigma^{-n}\int_{\Omega \cap B_{\sigma}(X_0)} w^j(X + Y) d\cH^n(X) \\
 &=& \mu_j^{-1}\sigma^{-n}\int_{\Omega \cap B_{\sigma}(X_0)}v^j_{\psi}(Y_j + \rho_j(X +Y)) d\cH^n(X) \\
&=& (1 + \rho_j) \mu_j^{-1}\sigma^{-n} \times \\
 & & \quad \quad \int_{\Omega \cap B_{\sigma}(X_0)}v^j_{\psi}(Y_j + (1 + \rho_j)^{-1}\rho_j(Y-Y_j) + (1 + \rho_j)^{-1}\rho_j X) d\cH^n(X) \\
&=& (1 + \rho_j)^{n+1} \mu_j^{-1}\sigma^{-n} \times \\
& & \quad \quad \int_{\Omega \cap B_{(1+\rho_j)^{-1}\sigma}((1 + \rho_j)^{-1}(Y - Y_j + X_0)}v^j_{\psi}(Y_j + \rho_j X) d\cH^n(X) \\
&=& (1 + \rho_j)^{n+1} \sigma^{-n}\int_{\Omega \cap B_{(1+\rho_j)^{-1}\sigma}((1 + \rho_j)^{-1}(Y - Y_j + X_0)}w^j(X) d\cH^n(X).
\end{eqnarray*}
Letting $j \to \infty$ and $\sigma \downarrow 0^+$, we conclude that $w(X_0 + Y) = w(X_0)$ for every $X_0 \in \Omega$, which implies that $Y \in S(w)$. Since (as one can check) $S(v) \subset S(w)$, we have that $\dim S(w) \geq n-d+1$. Thus by the inductive hypothesis we have that $w \in \cH(\bC^{(0)})$.  However, since $w$ is dehomogenized in $B_1(0)$ we deduce that $w \equiv 0$. But now, following the proof of Lemma 5.7 of \cite{wick08}, we deduce a contradiction: Given any $\eta > 0$, we can choose $j$ sufficiently large so that
\beq \label{homblowup5}
\int_{\Omega \cap B_{1/2}(0)} |w^j|^2 d\cH^n  \leq \eta .
\eeq
Since by construction we have that $\int_{\Omega} |w^j| d\cH^n = 1$, this shows that for any $\eta > 0$, we have 
\beq \label{homblowup6}
\int_{\Omega \cap (B_1(0)\setminus B_{1/2}(0))} |w^j|^2d\cH^n > 1 - \eta
\eeq
for sufficiently large $j$. But now, For $r,s \in (1/4,1)$ and $\omega \in \bC^{(0)} \cap \{ r_{\bC^{(0)}} > 0 \} \cap \partial B_1(0)$, we have
\beq
\left|\frac{w^j(r\omega)}{r} - \frac{w^j(s\omega)}{s}\right| = \left| \int_s^r \pa{(w^j(t\omega)/t)}{t} dt \right| \leq \int_s^r \left|\pa{(w^j(t\omega)/t)}{t}  \right| dt,
\eeq
and so by the triangle inequality, Cauchy-Schwarz and the fact that $|r/s|$ is bounded we have
\beq
|w^j(r\omega)|^2 \leq c\left(|w^j(s\omega)|^2 + \int_{1/4}^{1} t^{n-1}\left|\pa{(w^j(t\omega)/t)}{t}  \right|^2 dt\right)
\eeq
for some constant $c = c(n) > 0$. Now we integrate with respect to $\omega \in \bC^{(0)} \cap \{ r > 0 \} \cap \partial B_1(0)$. Then, we multiply by $r^{n-1}$ and integrate with respect to  $r$ in $(1/4,1)$ and finally we multiply by $s^{n-1}$ and integrate with respect to $s$ in $(1/4,1/2)$ to give (using the coarea formula)
\begin{eqnarray*}
\int_{\Omega \cap (B_1(0) \setminus B_{1/4}(0))} |w^j|^2d\cH^n & \leq & c\left(\int_{\bC^{(0)} \cap (B_{1/2}(0) \setminus B_{1/4}(0))} |w^j|^2d\cH^n + \right. \\
& & \quad \quad \left.  \int_{\Omega \cap (B_1(0) \setminus B_{1/4}(0))}\left|\pa{(w^j/R)}{R}\right|^2 d\cH^n \right) 
\end{eqnarray*}
for some $c = c(n) \geq 1$. Now we add $\int_{\Omega \cap  B_{1/4}(0)} |w^j|^2d\cH^n$ to both sides and use \eqref{homblowup5} and the fact that the final term in the above line tends to zero to deduce that
\beq
\eta > \int_{\Omega \cap B_{1/2}(0)} |w^j|^2d\cH^n \geq c(n) > 0
\eeq
independently of $j$, which is a contradiction. Thus the proof of the claim is complete and the estimate \eqref{homblowup2} indeed holds. 

\bigskip

\noindent \textbf{Step 5.} \emph{$C^1$ Boundary Regularity.} Now with $K$ again any compact subset of $(A(\bC^{(0)}) \setminus S(v)) \cap B_{1/4}(0)$ and $Y \in \cD_v \cap K$ for which $\kappa_v(Y) \neq 0$, we can replace $v$ by $\|v + \kappa_v(Y)^{\perp_{T_{(\cdot)}\bC^{(0)}}}\|_{L^2(\Omega)}^{-1} (v + \kappa_v(Y)^{\perp_{T_{(\cdot)}\bC^{(0)}}})$ (which belongs to $\fB(\bC^{(0)})$ by \ref{B5II}) in order to arrange that $\kappa_v(Y) = 0$. Assuming we have made this replacement, \eqref{homblowup2} together with the fact that $\partial(\psi_{Y,\rho}/R)/\partial R \equiv 0$ implies
\beq \label{homblowup6a}
\eps \rho^{-n-2}\int_{\Omega \cap B_{\rho}(Y) } |v - \psi_{Y,\rho}|^2 d\cH^n \leq \int_{\Omega  \cap (B_{\rho}(Y)\setminus B_{\rho/2}(Y)) }R_Y^{2-n}\left|\pa{(v/R_Y)}{R_Y} \right|^2 d\cH^n.
\eeq
Also, \eqref{B6III} of \ref{B6} applied to $\|v-\psi_{Y,\rho}\|_{L^2(\bC^{(0)} \cap B_1(0))}^{-1}(v-\psi_{Y,\rho})$ tells us that
\beq \label{homblowup7}
\int_{\Omega  \cap B_{\rho/2}(Y)} R_Y^{2-n}\left|\pa{(v/R_Y)}{R_Y} \right|^2d\cH^n \leq c\rho^{-n-2}\int_{\Omega  \cap B_{\rho}(Y)} |v - \psi_{Y,\rho}|^2d\cH^n.
\eeq
Combining these two inequalities we see that
\begin{align}
\frac{\eps}{c} \int_{\Omega  \cap B_{\rho/2}(Y)} R_Y^{2-n}&\left|\pa{(v/R_Y)}{R_Y} \right|^2 d\cH^n \nonumber \\
&\leq \int_{\Omega  \cap (B_{\rho}(Y)\setminus B_{\rho/2}(Y)) }R_Y^{2-n}\left|\pa{(v/R_Y)}{R_Y} \right|^2 d\cH^n,
\end{align}
where the $c$ that appears on the left-hand side here is the constant from the right-hand side of \eqref{homblowup7}. Adding $\int_{\Omega  \cap B_{\rho/2}(Y)}R_Y^{2-n}\left|\pa{(v/R_Y)}{R_Y} \right|^2$ to both sides (`hole-filling') and dividing by $(1 +\eps/c)$ we get that
\beq 
\int_{\Omega  \cap B_{\rho/2}(Y)} R_Y^{2-n}\left|\pa{(v/R_Y)}{R_Y} \right|^2 d\cH^n  \leq \eta\int_{\Omega  \cap B_{\rho}(Y)}R_Y^{2-n}\left|\pa{(v/R_Y)}{R_Y} \right|^2 d\cH^n.
\eeq
for some $\eta \in (0,1)$. Then, by iterating this with $2^{-i}\rho$ in place of $\rho$ and using a standard argument to interpolate between these scales, we deduce that
\beq
\int_{\Omega  \cap B_{\sigma}(Y)} R_Y^{2-n}\left|\pa{(v/R_Y)}{R_Y} \right|^2 d\cH^n  \leq c\left(\frac{\sigma}{\rho}\right)^{\mu} \int_{\Omega  \cap B_{\rho}(Y)}R_Y^{2-n}\left|\pa{(v/R_Y)}{R_Y} \right|^2 d\cH^n,
\eeq
for some $\mu = \mu(n,k,\bC^{(0)},v,K) \in (0,1)$, $c = c(n,k,\bC^{(0)},L) > 0$ and $0 < \sigma \leq \rho/2 \leq \eps/4$. Then using \eqref{homblowup6a} and \eqref{B6III} of \ref{B6} we deduce easily from this that
\beq \label{homblowup8}
\sigma^{-n-2}\int_{\Omega  \cap B_{\sigma}(Y)} |v - \psi_{Y,\sigma}|^2 d\cH^n \leq c\eps^{-1}\left(\frac{\sigma}{\rho}\right)^{\mu}\rho^{-n-2}\int_{\Omega \cap B_{\rho}(Y)} |v - \psi_{Y,\rho}|^2 d\cH^n
\eeq
for some $c = c(n,k,\bC^{(0)},L) > 0$ and $0 < \sigma \leq \rho/2 \leq \eps/8$. Using \eqref{homblowup8} and the triangle inequality, it is then straightforward to check that there exists a single $\psi_Y \in \cH(\bC^{(0)})$ for which
\beq \label{homblowup9}
\sigma^{-n-2}\int_{\Omega  \cap B_{\sigma}(Y)} |v - \psi_Y|^2d\cH^n \leq c\eps^{-1}\left(\frac{\sigma}{\rho}\right)^{\mu}\rho^{-n-2}\int_{\Omega  \cap B_{\rho}(Y)} |v - \psi_Y|^2d\cH^n
\eeq
for all $\sigma \in (0,\rho/2]$. Now, in general by replacing $v$ with 
\[
\|v - \kappa_v(Y)^{\perp_{T_{(\cdot)}\bC^{(0)}}}\|_{L^2(\Omega)}^{-1} (v - \kappa_v(Y)^{\perp_{T_{(\cdot)}\bC^{(0)}}}),
\]
we see that for each $Y \in \cD_v \cap K$, there is $\varphi_Y \in \cH(\bC^{(0)})$ for which we have
\begin{align} 
\sigma^{-n-2}\int_{\Omega  \cap B_{\sigma}(Y)} &|v - \kappa_v(Y)^{\perp_{T_X\bC^{(0)}}} - \varphi_Y|^2 d\cH^n \nonumber\\
&\leq \beta\left(\frac{\sigma}{\rho}\right)^{\mu}\rho^{-n-2}\int_{\Omega  \cap B_{\rho}(Y)} |v - \kappa_v(Y)^{\perp_{T_X\bC^{(0)}}} - \varphi_Y|^2d\cH^n, \label{homblowup10}
\end{align}
for all $0 < \sigma \leq \rho/2 \leq \eps/8$. Here $\beta = \beta(n,k,\bC^{(0)},K,v) > 0$. Using the regularity of $v$ away from $\cD_v$, we know that for any $X_0 = (r\omega_j,y) \in (\overline{\bH_j^{(0)}} \setminus \cD_v) \cap B_{1/4}(0)$ and any affine function $l : \RR^n \to \RR^k$ we have
\begin{align}
\sigma^{-n-2} & \sum_{j=1}^4 \int_{\bH_j^{(0)} \cap B_{\sigma}((r\omega_j,y))}  |v(X) - (v((r\omega_j,y)) + X \cdot Dv((r\omega_j,y))|^2d\cH^n(X) \nonumber \\
& \leq c\left(\frac{\sigma}{\rho}\right)^2\rho^{-n-2}\sum_{j=1}^4  \int_{\bH_j^{(0)} \cap B_{\rho}((r\omega_j,y))}|v - l|^2 d\cH^n \label{homblowup11}
\end{align}
for $0 < \sigma \leq \rho/2 \leq \tfrac{1}{2}\min\{\tfrac{1}{4},\dist(X_0,\cD_v)\}$ and where $c = c(n,k) > 0$. Provided we allow a larger constant in this estimate, a short compactness argument using only basic properties of harmonic functions shows that it is still true if we replace the $l$ that appears on the right-hand side by any $\psi \in \cH(\bC^{(0)})$. Then these two estimates can be used in an elementary way together with a Campanato Lemma (\emph{e.g.} \cite[Theorem 4.4]{rafeiromorreycampanato}) to give that $v\vert_{\bH_j^{(0)} \cap \Omega } \in C^1( (\bH_j^{(0)}\setminus S(v)) ; \bH_j^{(0)\perp})$ for $j=1,..,4$. Recalling that $\{\omega_1,...,\omega_4\} = \{r_{\bC^{(0)}} = 1\} \cap A(\bC^{(0)})^{\perp} \cap \spt \|\bC^{(0)}\|$, for $(r,y) = (r_{\bC^{(0)}},y) \in [0,\infty) \times A(\bC^{(0)})$, we will write $\tilde{v}_j(r,y) = v(r\omega_j,y)$.

\bigskip

\noindent \textbf{Step 6.} \emph{The Thin Part of $\cD_v$.} We claim that $v_i$ is harmonic at points of $\bP_i^{(0)} \setminus \cT_v$. By the definition of $\cT_v$, we have that for each $Y \in \cD_v \setminus \cT_v$, there is $\eta > 0$ such that $\cH^{n-1}(\cD_v \cap B_{\eta}(Y)) = 0$. The previous step shows that $v_i$ is Lipschitz and sets of $\cH^{n-1}$ measure zero are removable for Lipschitz harmonic functions (this follows from a short cut-off argument using the fact that such sets are of zero 1-capacity). Therefore $\cD_v\setminus \cT_v$ is removable for $v_i$. We can obviously now assume that $0 \in \cT_v$.

\bigskip

\noindent \textbf{Step 7.} \emph{Concluding the Argument.} 
Suppose that $\omega_1$, $\omega_3 \in \bP_1^{(0)}$. Notice that $(r,y) \mapsto \tilde{v}_1(r,y) - \tilde{v}_3(r,y)$ is a homogeneous degree one harmonic function that solves the half-space Dirichlet problem with zero boundary values. This means that $\tilde{v}_1(r,y) - \tilde{v}_3(r,y) = ra$ for some constant vector $a \in \bP_1^{(0)\perp}$, but by considering a point $(0,y) \in A(\bC^{(0)})$ at which $v_1$ is smooth, we see that
\beq \label{homblowup13}
\lim_{r \downarrow 0^+}\pa{}{r} \tilde{v}_1(r,y) = \lim_{r \downarrow 0^+}\pa{}{r}\tilde{v}_3(r,y) 
\eeq
which implies that $a = 0$ and hence that the above equation holds for all $(0,y) \in A(\bC^{(0)})\setminus S(v)$. 

We now turn our attention to derivatives in the directions along the axis. By taking derivatives on the set $\mathrm{Int}_{A(\bC^{(0)})}(\cT_v)$ (by which we mean the interior of $\cT_v$ as a subset of or `relative to' $A(\bC^{(0)})$) and using \eqref{homblowup12}, we have that
\beq \label{homblowup14}
D_p (v_1\vert_{\cT_v}) \equiv c_p^{\perp_{\bP_1^{(0)}}}\ \ \text{on}\ \ \mathrm{Int}_{A(\bC^{(0)})}(\cT_v) \setminus S(v)
\eeq
for $p \in \{1,...,n-1\}$. If we write $T_v := \overline{\mathrm{Int}_{A(\bC^{(0)})}(\cT_v)}$, then using the continuity of $D_p \tilde{v}_j$ on $A(\bC^{(0)})\setminus S(v)$ for $j \in \{1,3\}$, we get that \eqref{homblowup14} holds on all of $T_v \setminus S(v)$. In conjunction with \eqref{homblowup13}, this means that $v_1$ is $C^1$ at points of $T_v \setminus S(v)$. Now consider a point $Y = (0,y) \in (\cT_v\setminus T_v) \setminus S(v)$. The general (\emph{i.e.} purely topological) fact that for any set $U$ we have $\overline{U^c} \supset U \setminus \overline{\mathrm{Int}(U)}$ means that there exists a sequence of points $Y_m = (0,y_m) \in (A(\bC^{(0)}) \setminus \cT_v) \setminus S(v)$ with $Y_m \to Y$. By \textbf{Step 6.}, $v_1$ is smooth and harmonic at each $Y_m$, which implies that
\[
\lim_{r \downarrow 0^+} D_p\tilde{v}_1(r,y_m) = \lim_{r \downarrow 0^+}D_p\tilde{v}_3(r,y_m) 
\]
for each $p =1,...,n-1$. The fact that $v\vert_{\bH_j^{(0)} \cap \Omega } \in C^1( (\bH_j^{(0)}\setminus S(v)) ; \bH_j^{(0)\perp})$ for $j=1,..,4$ (from \textbf{Step 5.}) means that we can then let $m \to \infty$ to deduce that
\[
\lim_{r \downarrow 0^+} D_p\tilde{v}_1(r,y) = \lim_{r \downarrow 0^+}D_p\tilde{v}_3(r,y).
\]
In conjunction with \eqref{homblowup13}, this proves that $v_1$ is $C^1$ at $(0,y)$ and hence is $C^1$ on $\bP_1^{(0)}\setminus S(v)$. Now, since $T_v$ is contained in the zero set of $v_1 - \Sigma_{p=1}^{n-1}y^pc_p^{\perp_{\bP_1^{(0)}}}$, we can deduce that $v_1$ is smooth and harmonic at points of $T_v \setminus S(v)$, because the zero set of a $C^1$ harmonic function is removable (see, for example \cite{juutlindremovability}).
we therefore deduce that $v_1$ is smooth and harmonic on $\bP_1^{(0)}\setminus S(v)$. Finally, since $\dim S(v) \leq n-2$, it too is removable and thus $v_1$ is smooth at the origin and therefore linear. Of course the same is true for $v_2$ and this completes the proof.
\end{proof}

Finally we prove the main result of this section.

\begin{theorem} \label{blowupregu} Fix a properly aligned cone $\bC^{(0)} \in \cC$ and $L > 0$. There exists $\bar{\theta}_1  = \bar{\theta}_1(n,k,\bC^{(0)},L) \in (0,1/16)$ and $\mu = \mu(n,k,\bC^{(0)},L) \in (0,1)$ such that the following is true. Let $v \in \fB(\bC^{(0)})$ be the blow-up of a sequence $\{V^j\}_{j=1}^{\infty} \in \cV_L$ with $\Theta_{V^j}(0) \geq 2$ for every $j$. Then there exists $\psi \in \cH(\bC^{(0)})$ with $\sup_{\Omega} |\psi|^2 \leq c\int_{\Omega} |v|^2d\cH^n$ such that for any $\theta \in (0,\bar{\theta}_1)$ we have the estimate:
\beq \label{blowupregu1}
\theta^{-n-2}\int_{\Omega \cap B_{\theta}(0)} |v - \psi|^2d\cH^n \leq c\theta^{2\mu}\int_{\Omega} |v|^2 d\cH^n,
\eeq
for some $c = c(n,k,\bC^{(0)},L) > 0$.
\end{theorem}

\begin{proof} We argue exactly as in Steps 4. and 5. of the proof of Theorem \ref{homblowup}. That is, we first argue by contradiction to prove that there exists $\eps = \eps(n,k,\bC^{(0)}) > 0$ such that for every $\rho \in (0,1/4]$ and $Y \in \cD_v \cap B_{1/4}(0)$, there exists $\varphi_Y \in \cH(\bC^{(0)})$ such that 
\begin{align} 
 \int_{\Omega \cap (B_{\rho}(Y) \setminus B_{\rho/2}(Y))}& R^{2-n}\left|\pa{((v -\kappa_v(Y)^{\perp_{T_X\bC^{(0)}}} - \varphi_Y)/R)}{R}\right|^2d\cH^n \nonumber \\
& \geq \eps  \rho^{-n-2}\int_{\Omega \cap B_{\rho}(Y)} |v -\kappa_v(Y)^{\perp_{T_X\bC^{(0)}}} - \varphi_Y|^2 d\cH^n. \label{blowupregu2} 
\end{align}
Then by the same `hole-filling' argument used in Step 5., we get that there is $\bar{\theta} \in (0,1)$ such that for every $Y \in \cD_v \cap B_{1/4}(0)$, there is $\varphi_Y \in \cH(\bC^{(0)})$ and $\kappa_v(Y) \in \RR^{1+k}\times\{0\}^m$ such that
\begin{align} 
\sigma^{-n-2}&\int_{\bC^{(0)} \cap B_{\sigma}(Y)} |v - \kappa_v(Y)^{\perp_{T_X\bC^{(0)}}} - \varphi_Y|^2 d\cH^n \nonumber \\
& \leq \beta\left(\frac{\sigma}{\rho}\right)^{\mu}\rho^{-n-2}\int_{\bC^{(0)} \cap B_{\rho}(Y)} |v - \kappa_v(Y)^{\perp_{T_X\bC^{(0)}}} - \varphi_Y|^2d\cH^n, \label{blowupregu3}
\end{align}
for all $0 < \sigma \leq \rho/2 \leq \bar{\theta}/2$. Here $\beta = \beta(n,k,\bC^{(0)}) > 0$. And as before, using the regularity of $v$ away from $\cD_v$, we know that for any $X_0 = (r\omega_j,y) \in (\overline{\bH_j^{(0)}} \setminus \cD_v) \cap B_{1/4}(0)$ and any $\varphi \in \cH(\bC^{(0)})$ we have
\begin{align}
\sigma^{-n-2} & \sum_{j=1}^4 \int_{\bH_j^{(0)} \cap B_{\sigma}((r\omega_j,y))}  |v(X) - (v((r\omega_j,y)) + X \cdot Dv((r\omega_j,y))|^2d\cH^n(X) \nonumber \\
& \leq c\left(\frac{\sigma}{\rho}\right)^2\rho^{-n-2}\sum_{j=1}^4  \int_{\bH_j^{(0)} \cap B_{\rho}((r\omega_j,y))}|v - \varphi|^2 d\cH^n \label{blowupregu4}
\end{align}
for $0 < \sigma \leq \rho/2 \leq \tfrac{1}{2}\min\{\tfrac{1}{4},\dist(X,\cD_v)\}$ and where $c = c(n,k) > 0$. And in a similar way to the end of Step 5., using no further properties these estimates can be leveraged to yield \eqref{blowupregu1}. Notice that that the assumption $\Theta_{V^j}(0) \geq 2$ $\forall j$ implies, via the content of Remark \ref{rema:kappa=0}, that $\kappa_v(0) = 0$.
\end{proof}

\section{Proofs of $L^2$ Estimates} \label{S:ProofofL2}

The work of Sections 3, 4 and 5 relied on Theorem \ref{L2} and Corollary \ref{cor:L2Z}. In this section, we prove Theorem \ref{L2} and Corollary \ref{cor:L2Z} using an induction argument on $q_{\bC}$. Given $\bC$ with $q_{\bC} > 0$, we define the following induction hypothesis: 

\smallskip

\noindent \textbf{Induction Hypothesis $H(\bC,\bC^{(0)})$.} \emph{The statements of Theorem \ref{L2} and Corollary \ref{cor:L2Z} both hold with any $\bC^{\prime} \in \cP$ that satisfies $0 \leq q_{\bC^{\prime}} < q_{\bC}$ in place of $\bC$.}

\smallskip

We will prove the $q_{\bC} = 0$ case simultaneously, in a manner which is not circular. Let us briefly outline how this is done. Logically, one must proceed as follows: 
\begin{enumerate}
\item Prove Lemma \ref{q=0smalexcetori}.
\item Prove Theorem \ref{L2} for $q_{\bC} = 0$ (using Section \ref{SS:ProofofTheoremL2} up to but not including Section \ref{SSS:proofofclaimcont}).
\item Prove Corollary \ref{cor:L2Z} for $q_{\bC} = 0$ (Section \ref{SS:ProofofcorL2Z}).
\end{enumerate}
Once this is done, we can consider a general value of $q_{\bC} > 0$ and assume the induction hypothesis $H(\bC,\bC^{(0)})$.  Now, starting from Lemma \ref{L:q>0tori}, the logical order matches the order in which the article is written.

In order to use the induction hypothesis it will be useful to make the following definition for $\bC \in \cP$ with $q_{\bC} > 0$.
\beq
\cE^2_V(\bC) := \inf_{\substack{\bD \in \cC :\\ A(\bD) \supsetneq A(\bC)\\ q_{\bD} \geq 0 }}\int_{B_1(0)}\dist^2(X,\spt\|\bD\|)d\|V\|(X).
\eeq
 Given a subspace $A$ of $\RR^{n+k}$, $\zeta \in A$, $\rho \in (0,1]$ and $r \in (0,\rho)$, we define
\[
T^A_{\rho,r}(\zeta):= \{ (x,y) \in A^{\perp} \times A : (|x| - \rho)^2 + |y - \zeta|^2 < r^2\}.
\]
For any $\bC \in \cC$, we define $T^{\bC}_{\rho,r}(\zeta):= T^{A(\bC)}_{\rho,r}(\zeta)$. Notice that $T^{\bC}_{\rho,r}(\zeta)$ is always a toric region that `goes around' the axis of $\bC$. In particular, $T^{\bC}_{\rho,r}(\zeta) \cap A(\bC) = \emptyset$. More generally, for any set $S \subset \RR^{n+k}$, we define $T^{\bC}(S)$ to be the region of revolution formed by rotating $S$ about $A(\bC)$, \emph{i.e.} $T^{\bC}(S)$ is the set of points $(x,y)$ (in coordinates in which $\bC$ is properly aligned) such that there is some $\omega \in A(\bC)^{\perp} \cap \{r_{\bC} = 1\}$ such that $(|x|\omega,y) \in S$. 

Of crucial importance will be a good understanding of the behaviour that can occur when the varifold exhibits small $L^2$ excess in a region of the form $T^{\bC}(S)$ for some set $S$, where usually $S \cap A(\bC) = \emptyset$. The situation when $q_{\bC} > 0$ is significantly more complicated than when $q_{\bC} = 0$. Let us begin with the easier case. Henceforth in this section we assume that we have fixed $\bC^{(0)} \in \cC$ and $L > 0$.

\subsection{Small One-Sided Excess in Toric Regions when $q_{\bC} = 0$}

\begin{lemma} \label{q=0smalexcetori} There exists $\eps_0 = \eps_0(n,k,\bC^{(0)},L) > 0$ such that the following is true. Suppose that, for some $\eps < \eps_0$, we have that $V \in \cV_L$, $\bC, \bC^{(0)} \in \cC$ satisfy the following hypotheses:
\begin{enumerate}[nolistsep]
\item $\|V\|(B_2^n(0)\times\RR^k) \leq \|\bC^{(0)}\|(B_2^n(0)\times\RR^k) + 1/2$.
\item $A(\bC) = A(\bC^{(0)})$ (in particular, $q_{\bC} = 0$).
\item $\nu_{\bC,\bC^{(0)}} < \eps$ .
\item $E_{V \tiny{\res} T^{\bC}_{1/2,7/16}(0)} (\bC) < \eps$.
\end{enumerate}
Then there are various different possible conclusions. 

If $\dim A(\bC^{(0)}) = n-1$, then one of A) and B) hold. Write $\bC = \sum_{i=1}^4|\bH_i|$ and $\bC^{(0)} = \sum_{i=1}^4|\bH_i^{(0)}|$, labelled so that $d_{\cH}(\bH_i \cap (B_2^n(0) \times \RR^k), \bH_i^{(0)} \cap (B_2^n(0) \times \RR^k)) < \eps$. There exists $I \subset \{1,2,3,4\}$ such that $V \res T_{1/2,6/16}(0) = \sum_{i \in I} V_i$ where each $V_i$ is a stationary varifold in $T^{\bC}_{1/2,6/16}(0)$ for which there exists a domain $\Omega_i \subset \bH_i \cap T^{\bC}_{1/2,6/16}(0)$ such that either 
\begin{enumerate}
\item[\emph{A)}] $V_i$ \emph{is a smooth minimal graph}: There exists $u_i \in C^{\infty}(\Omega_i; \bC^{\perp})$ such that  
\begin{enumerate} [nolistsep]
\item[i)] $V_i = |\gr u_i|$.
\item[ii)] For $X \in \Omega_i$, $\dist(X + u_i(X),\bH_i) = |u_i(X)|$.
\end{enumerate}
Or,
\item[\emph{B)}] $V_i$ \emph{is, up to a small set, a minimal two-valued graph}: There exists a measurable set $\Sigma_i \subset \Omega_i$ and $u_i \in C^{0,1}(\Omega_i; \cA_2(\bH^{\perp}_i))$ such that
\begin{enumerate} [nolistsep]
\item[i)] $V_i \res T^{\bC}_{1/2,6/16}(0)\setminus (\Sigma_i \times \bH_i^{\perp}) = V_{u_i \vert_{\Omega_i\setminus \Sigma_i}}$.
\item[ii)] $\cH^n(\Sigma) + \|V_i \res T_{1/2,6/16}^{\bC}(0)\|(\Sigma_i \times \bH_i^{\perp})$ \\ $\leq c \int_{T^{\bC}_{1/2,7/16}(0)}\dist^2(X,\spt\|\bC|)d\|V\|(X)$
\item[iii)] For $X \in \Omega_i\setminus \Sigma_i$, $\dist(X + u_i(X),\bH_i) = |u_i(X)|$.
\end{enumerate}
\end{enumerate}
If $\dim A(\bC^{(0)}) < n-1$, then one of C) and D) hold. Write $\bC = \sum_{i=1}^2|\bP_i|$ and $\bC^{(0)} = \sum_{i=1}^2|\bP_i^{(0)}|$, labelled so that $d_{\cH}(\bP_i \cap (B_2^n(0) \times \RR^k), \bP_i^{(0)} \cap (B_2^n(0) \times \RR^k)) < \eps$. Then we have either
\begin{enumerate}
\item[\emph{C)}] $V \res T^{\bC}_{1/2,6/16}(0)$ \emph{is a union of smooth minimal graphs}: There exists a domain $\Omega \subset \spt\|\bC\| \cap T^{\bC}_{1/2,6/16}(0)$ and a smooth function $u \in C^{\infty}(\Omega; \bC^{\perp})$. such that
\begin{enumerate} [nolistsep]
\item[i)] $V = |\gr u|$.
\item[ii)] For $X \in \Omega \cap \bP_i$, $\dist(X + u(X),\bP_i) = |u(X)|$.
\end{enumerate}
Or
\item[\emph{D)}] $V \res T^{\bC}_{1/2,6/16}(0)$ \emph{is, up to a small set, a minimal two-valued graph}: There exists $i \in \{1,2\}$, a domain $\Omega \subset \bP_i \cap T^{\bC}_{1/2,6/16}(0)$, a measurable set $\Sigma \subset \Omega$ and $u_i \in C^{0,1}(\Omega;\cA_2(\bP_i^{\perp}))$ such that
\begin{enumerate}[nolistsep]
\item[i)] $V \res T^{\bC}_{1/2,6/16}(0) \setminus (\Sigma \times \bP_i^{\perp}) = V_{u_i\vert_{\Omega\setminus \Sigma}}$.
\item[ii)] $\cH^n(\Sigma) + \|V \res T^{\bC}_{1/2,6/16}(0)\|(\Sigma \times \bP_i^{\perp})$ \\ $ \leq c \int_{T^{\bC}_{1/2,7/16}(0)}\dist^2(X,\spt\|\bC|)d\|V\|(X)$.
\item[iii)] For $X \in \Omega\setminus \Sigma$, $\dist(X + u_i(X),\bP_i) = |u_i(X)|$.
\end{enumerate}
\end{enumerate}
\end{lemma}

\begin{proof}
To prove this Lemma we take a sequence of varifolds and cones satisfying the hypotheses for smaller and smaller choices of $\eps_0$ and show that the conclusions must hold at least along a subsequence. So consider a sequence of real numbers $\eps_j \downarrow 0^+$, sequences $\{V^j\}_{j=1}^{\infty} \in \cV_L$ and $\{\bC^j\}_{j=1}^{\infty} \in \cC$ such that for every $j \geq 1$, the hypotheses 1) to 4) in the statement of the Lemma are satisfied with $V^j$, $\bC^j$ and $\eps_j$ in place of $V$, $\bC$ and $\eps$ respectively.

Now, using the mass bound 1) and the compactness theorem for stationary integral varifolds, we know that there exists a subsequence $\{j'\}$ of $\{j\}$ (to which we pass without changing notation and) along which we have that $V^j \res T^{\bC^j}_{1/2,7/16}(0)$ converges to some stationary integral varifold $W$ in $T^{\bC^{(0)}}_{1/2,7/16}(0)$. By 3) and 4) we have that $\spt \|W\| \subset \spt \|\bC^{(0)}\| \cap T^{\bC^{(0)}}_{1/2,7/16}(0)$. 

Suppose now that $\dim A(\bC) = n-1$. Since $W$ is stationary, the Constancy Theorem \cite[Theorem 41.1]{simongmt} implies that it has constant multiplicity on each connected component of $\spt\|W\|$ and so we see that each connected component of $\spt \|W\|$ must be of the form $\bH_i^{(0)} \cap T^{\bC^{(0)}}_{1/2,7/16}(0)$ for one of the half-planes $\bH_i^{(0)}$ of $\bC^{(0)}$, \emph{i.e.} we can write $W = \sum_{i \in I} \theta_i|\bH_i^{(0)} \cap T_{1/2,7/16}(0)|$ where $\theta_i$ is a positive integer for each $i$. 

Choose a small but fixed $\eta > 0$ so that the regions $(\bH_i^{(0)})_{\eta} \cap T_{1/2,7/16}(0)$ for $i \in \cI$ are disjoint. Then, since $d_{\cH}(\spt \|V^j\| \cap T^{\bC^{(0)}}_{1/2,7/16}(0), \spt \|W\| \cap T^{\bC^{(0)}}_{1/2,7/16}(0)) \to 0$, we know that for sufficiently large $j$,  
\beq \label{q=0smalexcetori1}
\spt\|V^j\| \cap T^{\bC^{(0)}}_{1/2,7/16}(0) \subset (\bH_i^{(0)})_{\eta} \cap T^{\bC^{(0)}}_{1/2,7/16}(0),
\eeq
whence we define $V_i^j := V^j \res (\bH_i^{(0)})_{\eta} \cap T^{\bC^{(0)}}_{1/2,1/4}(0)$. Now given $i \in I$ for which $\theta_i = 1$, for sufficiently large $j$ we may apply Allard's Regularity Theorem to $V^j_i$ in order to write it as a smooth graph and deduce conclusion A). In fact, given any $i \in I$ for which $\Theta_{V_i^j}(X) < 2$ for all $X \in \spt\|V_i^j\|$ and all sufficiently large $j$, we can use Theorem \ref{thm:grapdeco} followed by Allard's Regularity Theorem to get the same conclusions. Given $i \in \cI$ for which $\theta_i = 2$ and such that there exists $Z \in \spt\|V_i^j\|$ with $\Theta_V(Z) \geq 2$, we apply instead Almgren's Lipschitz Approximation Lemma (\cite[Corollary 3.11]{almgrenbig}, or see \cite[Theorem 5.1]{wickgeneral}) in order to deduce B). Finally, if we were to suppose that there were an $i$ for which $\theta_i \geq 3$, one can use the fact that $W$ is a limit of two-valued Lipschitz graphs together with the mass bound ( 1.) to reach a contradiction. This completes the proof in the case $\dim A(\bC) = n-1$. We omit the proof for the case where $\dim A(\bC) < n-1$ as it follows exactly the same method. The slightly different conclusion results naturally from the different geometry. 
\end{proof}

\subsection{Small One-Sided Excess in Toric Regions when $q_{\bC} > 0$}

In this subsection we will prove a lemma that that is analogous to Lemma \ref{q=0smalexcetori}, but in the case where $q_{\bC} > 0$. To do so, we must make use of the induction hypothesis $H(\bC,\bC^{(0)})$ (recall that logically, the use of this lemma comes only after the proofs of Theorem \ref{L2} and Corollary \ref{cor:L2Z} have been established for the $q_{\bC} = 0$ case). 

Notice that if $V^j \to \bC^{(0)}$, the set $\{Z : \Theta_{V^j}(Z) \geq 2\}$ will concentrate near $A(\bC^{(0)})$ - this is how the proof of Lemma \ref{q=0smalexcetori} worked. The next lemma expresses the important fact that if $q_{\bC} > 0$ and if the excess measured relative to $\bC$ is \emph{much smaller} than the excess measured relative to pairs of planes with higher-dimensional axes, (see hypothesis \emph{6)} below), then the set $\{Z : \Theta_{V^j}(Z) \geq 2\}$  will concentrate near the smaller set $A(\bC) \subsetneq A(\bC^{(0)})$.

\begin{lemma} \label{L:q>0tori} There exist numbers $\eps_0 = \eps_0(n,k,\bC^{(0)},L) >0$ and $\gamma_0$ $=$ $\gamma_0(n,k,\bC^{(0)},L)$ $>$ $0$ such that the following is true. If, for some $\eps < \eps_0$ and $\gamma < \gamma_0$, we have that $V \in \cV_L$, $\bC = |\bP_1| + |\bP_2| \in \cP$ and $\bC^{(0)} \in \cP$ satisfy
\begin{enumerate}[nolistsep]
\item $\|V\|(B_2^n(0)\times\RR^k) \leq \|\bC^{(0)}\|(B_2^n(0)\times\RR^k) + 1/2$.
\item $q_{\bC} > 0$ and $0 \in A(\bC) \subsetneq A(\bC^{(0)})$.
\item Hypothesis $H(\bC,\bC^{(0)})$.
\item $\nu_{\bC, \bC^{(0)}} < \eps$.
\item $E_{V \tiny{\res} T^{\bC}_{1/2,7/16}(0)}(\bC) < \eps$
\item $E_{V \tiny{\res} T^{\bC}_{1/2,7/16}(0)}(\bC)  < \gamma\cE_{V \tiny{\res} T^{\bC}_{1/2,7/16}(0)}(\bC)$.
\end{enumerate}
Then there are two possible conclusions, \emph{E)} and \emph{F)}: Either
\begin{enumerate}
\item[\emph{E)}] $V \res T^{\bC}_{1/2,6/16}(0)$ \emph{is a union of single-valued graphs}: $\{X \in \spt \|V\| \cap T^{\bC}_{1/2,6/16}(0) : \Theta_V(X) \geq 2 \} = \emptyset$ and $V \res T^{\bC}_{1/2,6/16}(0) = \sum_{i=1}^2V_i$, where for each $i \in \{1,2\}$, $V_i$ is a stationary varifold for which there exists a domain $\Omega_i \subset \bP_i \cap T^{\bC}_{1/2,6/16}(0)$ and $u_i\in C^{\infty}(\Omega_i; \bC^{\perp})$ such that
\begin{enumerate}[nolistsep, label = \roman*),ref= \roman*)]
\item \label{en:q>0tori1} $V_i = |\gr u_i|$
\item \label{en:q>0tori2} $\int_{T^{\bC}_{1/2,13/32}(0)} \dist^2(X,\bP_i)d\|V_i\|(X)$ \\ $\leq c\int_{T^{\bC}_{1/2,7/16}(0)}\dist^2(X,\spt\|\bC\|)d\|V\|(X)$, for some constant \\ $c$ $=$ $c(n,k,\bC^{(0)},L) > 0$ (for $X \in \Omega_i$, $\dist(X + u_i(X),\bP_i) = |u_i(X)|$). 
\end{enumerate}
\end{enumerate}
Or,
\begin{enumerate}
\item[\emph{F)}] $V \res T^{\bC}_{1/2,6/16}(0)$ \emph{is, up to a small set, a minimal two-valued graph}: There exists $i \in \{1,2\}$, a domain $\Omega \subset \bP_i \cap T^{\bC}_{1/2,6/16}(0)$, a measurable set $\Sigma \subset \Omega$ and $u_i \in C^{0,1}(\Omega;\cA_2(\bP_i^{\perp}))$ such that
\begin{enumerate}[nolistsep, label = \roman*),ref=\roman*)]
\item\label{en:q>0tori3} $V \res T^{\bC}_{1/2,6/16}(0) \setminus (\Sigma \times \bP_i^{\perp}) = V_{u_i\vert_{\Omega\setminus \Sigma}}$ and for $X \in \Omega\setminus \Sigma$, $\dist(X + u_i(X),\bP_i) = |u_i(X)|$, for some constant $c = c(n,k) > 0$.
\item\label{en:q>0tori4} $\cH^n(\Sigma) + \|V \res T^{\bC}_{1/2,6/16}(0)\|(\Sigma \times \bP_i^{\perp})$ \\ $\leq c \int_{T^{\bC}_{1/2,13/32}(0)}\dist^2(X,\bP_i)d\|V\|(X)$, for some constant \\ $c = c(n,k,\bC^{(0)},L) > 0$.
\item\label{en:q>0tori5} $\int_{T^{\bC}_{1/2,13/32}(0)} \dist^2(X,\bP_i)d\|V\|(X)$ \\ $\leq c\int_{T^{\bC}_{1/2,7/16}(0)}\dist^2(X,\spt\|\bC\|)d\|V\|(X)$, for some constant \\ $c = c(n,k,\bC^{(0)},L) > 0$.
\end{enumerate}
\end{enumerate}
\end{lemma}

\begin{proof} Consider to begin with sequences of real numbers $\eps_j \downarrow 0^+$ and $\gamma_j \downarrow 0^+$, sequences $\{V^j\}_{j=1}^{\infty} \in \cV$ and $\{\bC^j\}_{j=1}^{\infty} \in \cP$  such that for every $j \geq 1$, the hypotheses are satisfied with $V^j$, $\bC^j$, $\eps_j$ and $\gamma_j$ in place of $V$, $\bC$, $\eps$ and $\gamma$ respectively. It suffices to prove that the conclusions of the lemma hold along some subsequence. 

\bigskip

\noindent \textbf{Step 1.} \emph{Reducing to transverse case.}
By passing to a subsequence we may assume the following: 
\begin{itemize}[nolistsep]
\item There exists a subspace $A \subsetneq A(\bC^{(0)})$ for which $d_{\cH}(A(\bC^j)\cap B_2(0),A \cap B_2(0)) \to 0$ as $j \to \infty$.
\item $q_{\bC^j} \equiv q > 0$.
\item Using the mass bound (hypothesis 1) of the present lemma) and the compactness theorem for stationary integral varifolds, there exists a stationary integral varifold $W$ in $T^{A}_{1/2,7/16}(0)$ for which $V^j \res T^{\bC^j}_{1/2,7/16}(0) \to W$.
\end{itemize} 
First note that hypotheses 4) and 5) imply that $\spt \|W\| \subset \spt \|\bC^{(0)}\| \cap T^A_{1/2,7/16}(0)$ and the Constancy Theorem (\cite[Theorem 41.1]{simongmt}) tells us that the multiplicity of $W$ is constant on each connected component of $\spt\|\bC^{(0)}\|\setminus A(\bC^{(0)})$. If $\bC^{(0)} \in \cP_{\leq n-2}$, then this means that either $W = \bC^{(0)} \res T^A_{1/2,7/16}(0)$ or  $W = 2|\bP_i^{(0)} \cap T^A_{1/2,7/16}(0)|$ for some $i \in \{1,2\}$. If $\bC^{(0)} \in \cP_{n-1}$, then we have the same conclusions because the stationarity of $W$ rules out the case where $W$ is supported on 3 half-planes of $\bC^{(0)}$.

Suppose that $W = 2|\bP_i^{(0)} \cap T^A_{1/2,7/16}(0)|$. If $\Theta_{V^j}(X) < 2$ for all $X \in \spt\|V^j\| \cap T^{\bC^j}_{1/2,7/16}(0)$ for infinitely many $j$, then conclusion F) follows along this subsequence by combining Lemma \ref{lemm:planlemm} with Theorem \ref{thm:grapdeco} and Allard's Regularity Theorem. If, on the other hand, there exist points $Z_j \in \spt\|V^j\| \cap T^{\bC^j}_{1/2,7/16}(0)$ with $\Theta_{V^j}(Z_j) \geq 2$ for sufficiently large $j$, then we use Almgren's Lipschitz Approximation Lemma (\cite[Corollary 3.11]{almgrenbig}) to also yield conclusion F). Therefore we can suppose that $W = \bC^{(0)} \res T^A_{1/2,7/16}(0)$. This means that $\cQ_{V^j \tiny{\res} T^{\bC^j}_{1/2,7/16}(0)}(\bC^{(0)}) \to 0$ as $j \to \infty$.

\bigskip

\noindent We now assume, for the sake of contradiction, that:
\begin{align} \label{E:q>0toricontra}
\tag{X} &\text{for sufficiently large}\ j\ \text{there exists}\ Z_j \in \spt\|V^j\| \cap T^{\bC^j}_{1/2,6/16}(0)\\
 &\text{with}\ \Theta_{V^j}(Z_j) \geq 2. \nonumber
\end{align}

\bigskip

\noindent \textbf{Step 2.} \emph{Picking optimal coarser cones.} For some constant $\beta > 0$ (that we will show can be determined depending only on $n$, $k$, $\bC^{(0)}$ and $L$) and for each $j = 1,2,...$, we pick a new cone $\bD^j \in \cP$ according to the following algorithm: First set $\bD^j_{(0)} = \bC^j$. Then starting with $p=1$, choose $\bD^j_{(p)}$ inductively to satisfy
\begin{align}
& A(\bD^j_{(p)}) \supsetneq A(\bD^j_{(p-1)}), \label{E:q>0tori6}\\
& \dim A(\bD^j_{(p)}) \leq \dim A(\bC^{(0)}),\ \text{and}\label{E:q>0tori7}\\
& E_{V^j \tiny{\res} T^{\bC^j}_{1/2,7/16}(0)}(\bD^j_{(p)}) \leq \frac{3}{2}\cE_{V^j \tiny{\res} T^{\bC^j}_{1/2,7/16}(0)}(\bD^j_{(p-1)}).\label{E:q>0tori8}
\end{align}
At each step (\emph{i.e.} for each $p$ in turn), if either
\begin{align}
& \dim A(\bD^j_{(p)}) = \dim A(\bC^{(0)}),\  \text{or} \label{E:q>0tori9}\\
 & E_{V^j \tiny{\res} T^{\bC^j}_{1/2,7/16}(0)}(\bD^j_{(p)}) \leq \beta \cE_{V^j \tiny{\res} T^{\bC^j}_{1/2,7/16}(0)}(\bD^j_{(p)}) \label{E:q>0tori10}
\end{align}
hold, then stop and set $\bD^j = \bD^j_{(p)}$ (in fact, if \eqref{E:q>0tori9} holds, then we can say that \eqref{E:q>0tori10} holds vacuously). Notice that it must eventually be the case that \eqref{E:q>0tori9} holds, because the sequence $\dim A(\bD^j_{(p)})$ is strictly increasing in $p$. Notice also from the construction that
\beq \label{E:q>0tori11}
E_{V^j \tiny{\res} T^{\bC^j}_{1/2,7/16}(0)}(\bD^j) \leq c \cE_{V^j \tiny{\res} T^{\bC^j}_{1/2,7/16}(0)}(\bC^j).
\eeq
for some constant $c = c(n,k,\bC^{(0)}, L,\beta) > 0$. This follows from the fact that if $\bD^j = \bD^j_{(p')}$, then we have the opposite inequality in \eqref{E:q>0tori10} for all $p = 1,\dots,p'-1$. Using this in conjunction with \eqref{E:q>0tori8} for $p=1,\dots,p'-1$ establishes \eqref{E:q>0tori11}.

By passing to another subsequence we may further assume that: 
\begin{itemize}[nolistsep]
\item There exists a subspace $A'$ with $A \subsetneq A' \subset A(\bC^{(0)})$ for which $d_{\cH}(A(\bD^j)\cap B_2(0),A' \cap B_2(0)) \to 0$ as $j \to \infty$.
\item $0 \leq q_{\bD^j} \equiv q' < q$.
\item $\bD^j \to \bC^{(0)}$.
\item $Z_j \to Z \in A(\bC^{(0)}) \cap T^A_{1/2,6/16}(0)$.
\end{itemize}

\bigskip

\noindent \textbf{Step 3.} \emph{Blowing Up.} Now let $\cR$ be the set of all rotations $R$ of $\RR^{n+k}$ such that
\begin{enumerate}
\item $R$ fixes $A(\bC^j)$.
\item $R$ maps $T^{\bC^j}_{1/2,7/16}(0)$ to itself, and
\item $R(Z_j) \in A(\bD^j)$. 
\end{enumerate}
And then, for each $j$, choose $R^j \in \cR$ satisfying
\beq 
|R^j - \mathrm{id}_{\RR^{n+k}}| \leq \frac{3}{2}\inf_{R \in \cR}|R - \mathrm{id}_{\RR^{n+k}}|.
\eeq
Next we let $\cG$ be the set of all rotations $\Gamma$ of $\RR^{n+k}$ such that
\begin{enumerate}
\item $\Gamma$ fixes $A(\bC^j)$.
\item $\Gamma$ maps $T^{\bC^j}_{1/2,7/16}(0)$ to itself, and
\item $(A(\bC^j) \cup \{Z\}) \subset A(\Gamma_*\bD^j) \subset A(\bC^{(0)}) $. 
\end{enumerate} 
And for each $j$ we choose $\Gamma^j \in \cG$ satisfying
\beq 
|\Gamma^j - \mathrm{id}_{\RR^{n+k}} | \leq \frac{3}{2}\inf_{\Gamma \in \cG}|\Gamma - \mathrm{id}_{\RR^{n+k}}|.
\eeq
We claim that
\beq \label{E:q>0tori12}
\nu^2_{R^j_*\bD^j,\bD^j} \leq c\; \int_{T_{1/2,7/16}^{\bC^j}(0)} \dist^2(X,\spt\|\bD^j\|)d\|V^j\|(X).
\eeq
We will need to use Corollary \ref{cor:L2Z} (applied with $\bD^j$ in place of $\bC$); this is valid because of the induction hypothesis $H(\bC,\bC^{(0)})$ and because $q' < q$. If $q' = 0$, then \ref{en:L2Z2} of Corollary \ref{cor:L2Z} implies easily that $\dist(Z_j, A(\bD^j)) < 1/32$, whence some elementary geometric considerations show that for sufficiently large $j$,
\beq \label{E:q>0tori13}
\nu_{R^j_*\bD^j,\bD^j}  \leq c\;\nu_{T_{Z_j *}\bD^j, \bD^j}
\eeq
for some constant $c = c(n,k) > 0$ (here we have used the fact that $Z_j$ is bounded away from $A(\bC^j)$, by virtue of belonging to $T^{\bC^j}_{1/2,6/16}(0)$).  We see that \eqref{E:q>0tori12} now follows readily from \eqref{E:CtoC_Z5} of Remark \ref{R:CtoC_Z} applied with $(\eta_{Z_j^{\top_{A(\bC^{(0)})}},1/32})_*V^j$ in place of $V$. If, on the other hand $q' > 0$, then more work is needed because \ref{en:L2Z2} of Corollary \ref{cor:L2Z} gives only that
\beq \label{E:q>0tori14}
\dist^2(Z_j, A(\bD^j)) \leq c\bigl(\nu^2_{\bD^j,\bC^{(0)}}\bigr)^{-1}\int_{T_{1/2,7/16}^{\bC^j}(0)} \dist^2(X,\spt\|\bD^j\|)d\|V^j\|(X).
\eeq 
But, by using the triangle inequality and \eqref{E:CtoC_Z1} we have that
\begin{align} 
\int_{T_{1/2,7/16}^{\bC^j}(0)} \dist^2(X,&\spt\|\bC^{(0)}\|)d\|V^j\|(X) \label{E:q>0tori15}\\
&\leq 2\int_{T_{1/2,7/16}^{\bC^j}(0)} \dist^2(X,\spt\|\bD^j\|)d\|V^j\|(X) + c\; \nu^2_{\bD^j,\bC^{(0)}} \nonumber
\end{align}
for some absolute constant $c > 0$. And so, by using \eqref{E:q>0tori10}, we deduce from this that
\beq \label{E:q>0tori16}
c^{-1}(1-2\beta^2)\; \int_{T_{1/2,7/16}^{\bC^j}(0)} \dist^2(X,\spt\|\bC^{(0)}\|)d\|V^j\|(X)  \leq \nu^2_{\bD^j,\bC^{(0)}}.
\eeq
Combining this with \eqref{E:q>0tori14} and using \eqref{E:q>0tori10} again, we have that
\beq \label{E:q>0tori17}
\dist^2(Z_j, A(\bD^j)) \leq c\; \frac{\beta^2}{1-2\beta^2},
\eeq
for $c = c(n,k,\bC^{(0)},L)  > 0$. Therefore, by single choice of $\beta$ depending only on $n$, $k$, $\bC^{(0)}$ and $L$, we can ensure once again that $\dist^2(Z_j, A(\bD^j))  < 1/32$, whence for sufficiently large $j$ we have that $\nu_{R^j_*\bD^j,\bD^j}  \leq c\;\nu_{T_{Z_j *}\bD^j, \bD^j}$. From here, \eqref{E:CtoC_Z5} of Remark \ref{R:CtoC_Z} applied with $(\eta_{Z_j^{\top_{A(\bC^{(0)})}},1/32})_*V^j$ completes the proof of \eqref{E:q>0tori12}.

Moreover, since $d_{\cH}(A(\bD^j) \cap B_2(0),A' \cap B_2(0)) \to 0$ and $|Z_j - Z| \to 0$, we have that $\nu_{\Gamma^j_*\bD^j,(R^j)^{-1}_*\bD^j} \to 0$, whence using triangle inequality and the fact that $\bD^j \to \bC^{(0)}$, we get that $\nu_{\Gamma^j_*\bD^j,\bC^{(0)}} \to 0$.

We will now blow up $\{\Gamma^j_*R^j_*V^j$ $\res$ $T^{\bC^j}_{1/2,7/16}(0)\}_{j=1}^{\infty}$ off $\bC^{(0)}$ relative to the sequence of cones $\{\Gamma^j_*\bD^j\}_{j=1}^{\infty}$ in the region $T^A_{1/2,7/16}(0)$. Because of the different domain, this is slightly different from the general blow-up procedure used to construct $\fB(\bC^{(0)})$ in Section 4, but we will only need a few properties of the blow-up so let us briefly outline the procedure here: Fix $\tau_0 \in (0,1/64)$ and suppose that $\Gamma^j_*\bD^j = |\Gamma^j\bQ_1^j| + |\Gamma^j\bQ_2^j|$, labelled so that $|\Gamma^j\bQ_i^j| \to |\bP_i^{(0)}|$ as $j \to \infty$ for $i=1,2$. Then, for sufficiently large $j$, we can represent $\Gamma^j\bQ_i^j \cap \{r_{\bC^{(0)}} > \tau_0\}$ as the graph of $d^j$ (defined on a domain in $\Omega := B_1(0) \cap (\spt\|\bC^{(0)}\|\setminus A(\bC^{(0)}))$). Furthermore, we can represent $\Gamma^j_*R^j_*V^j \res (T^{\bC^j}_{1/2,(1-\tau_0)7/16}(0) \cap \{r_{\bC^{(0)}} > \tau_0\})$ as the graph of $d^j + u^j$, where $u^j$ is defined on a domain in $\Omega \cap T^A_{1/2,7/16}(0)$. Then, since $\tau_0$ is arbitrary, if we set
\beq
E_j := \left(\int_{T^{\bC^j}_{1/2,7/16}(0)} \dist^2(X,\spt\|\Gamma^j_*\bD^j\|)d\|\Gamma^j_*R^j_*V^j\|(X)\right)^{1/2},
\eeq
we have that $E_j^{-1}u^j$ converges on compact subsets of $\Omega \cap T^A_{1/2,7/16}(0)$ to a harmonic function $v$.

\bigskip

\noindent \textbf{Step 4.} \emph{The Structure of the Blow-Up.} Using the facts that both $R^j$ and $\Gamma^j$ are rigid motions of $T^{\bC^j}_{1/2,7/16}(0)$ together with hypothesis 6) of the present lemma, we see that
\begin{align} 
& \int_{T_{1/2,7/16}^{\bC^j}(0)}\dist^2(X,\spt\|\Gamma^j_*R^j_*\bC^j\|)d\|\Gamma^j_*R^j_*V^j\|(X) \nonumber \\
= & \int_{T_{1/2,7/16}^{\bC^j}(0)}\dist^2(X,\spt\|\bC^j\|)d\|V^j\|(X) \nonumber \\
< & \gamma_j  \int_{T^{\bC^j}_{1/2,7/16}(0)} \dist^2(X,\spt\|(R^j)^{-1}_*\bD^j\|)d\|V^j\|(X) \nonumber \\ 
= & \gamma_j E_j^2. \label{E:q>0tori18}
\end{align}
This shows that on $\Omega \cap T^A_{1/2,7/16}(0)$, the function $v$ coincides with the function obtained by blowing up $\Gamma^j_*R^j_*\bC^j$ relative to $\Gamma^j_*\bD^j$ using the same excess. This latter blow-up is the blow-up of a sequence of pairs of planes and so we deduce that $v$ is given by linear functions defined on $(\bP_i^{(0)}\setminus A(\bC^{(0)})) \cap T^A_{1/2,7/16}(0)$ for $i=1,2$. 

Let us check that $v \not \equiv 0$. Relying again on the fact that $R^j$ is a rigid motion of $T^{\bC^j}_{1/2,7/16}(0)$, and then using the triangle inequality and \eqref{E:CtoC_Z1} we have that
\begin{align}
E_j^2 &= \int_{T^{\bC^j}_{1/2,7/16}(0)} \dist^2(X,\spt\|\bD^j\|)d\|R^j_*V^j\|(X) \label{E:q>0tori19} \\
& \leq \int_{T_{1/2,7/16}^{\bC^j}(0)} \dist^2(X,\spt\|R^j_*\bC^j\|)d\|R^j_*V^j\|(X) + c\; \nu^2_{\bD^j,R^j_*\bC^j} \nonumber 
\end{align}
for some absolute constant $c > 0$. Then by changing variables in the integral and using hypothesis 6) of the present lemma, we have that this is less than
\begin{align}
&\gamma_j \int_{T_{1/2,7/16}^{\bC^j}(0)} \dist^2(X,\spt\|(R^j)^{-1}_*\bD^j\|)d\|V^j\|(X) + c\; \nu^2_{\bD^j,R^j_*\bC^j} \nonumber \\
= &\gamma_j E_j^2 + c\; \nu^2_{\bD^j,R^j_*\bC^j} \label{E:q>0tori20} 
\end{align}
Thus by absorbing the first term in \eqref{E:q>0tori20} back into \eqref{E:q>0tori19}, we deduce that
\beq \label{E:q>0tori21}
0 <  c  \leq E_j^{-1}\nu_{\Gamma^j_*\bD^j,\Gamma^j_*R^j_*\bC^j} 
\eeq
for some absolute constant $c > 0$, which indeed implies that $v \not\equiv 0$. 

Now on $(\bP_i^{(0)}\setminus A(\bC^{(0)})) \cap T^A_{1/2,7/16}(0)$ set $l_i^j := d_i^j + E_j v$ and let $\bL_i^j$ denote the unique plane containing $\gr l^j_i$. We define $\bF^j := |\bL_1^j| + |\bL_2^j|$. Since $v \not \equiv 0$, notice that $\bF^j \neq \bD^j$.

We claim that $A(\bF^j) \supsetneq A(\bC^j)$ . Since $A(\bC^j) \subsetneq A(\Gamma^j_*\bD^j) \subset  A(\bC^{(0)})$ for every $j$ we have that $d^j_i \equiv 0$ on $A(\bC^j)$. And since $v$ is the blow-up of $\Gamma^j_*R^j_*\bC^j$ and $A(\Gamma^j_*R^j_*\bC^j) = A(\bC^j)$ for every $j$, we also have that $v \equiv 0$ on $A(\bC^j)$. This means that $l_i^j \equiv 0$ on $A(\bC^j)$ for $i=1,2$ which means that 
\beq\label{E:q>0tori22}
A(\bC^j) \subset (\bL_1^j \cap \bL_2^j) = A(\bF^j).
\eeq
Now, since $Z \in A(\Gamma^j_*\bD^j) \subset A(\bC^{(0)})$ for every $j$, we have that $d^j_i(Z) = 0$ for every $j$. And , since $|(\Gamma^j\circ R^j)(Z_j)^{\perp_{A(\Gamma^j_*\bD^j)}}| = 0$ for every $j$, we have that $\kappa_v(Z) = 0$ (by Remark \ref{rema:kappa=0}). This means that $l_i^j(Z) = 0$ for all $j$ and for $i=1,2$, from which we can conclude that
\beq\label{E:q>0tori23}
Z \in (\bL_1^j \cap \bL_2^j) = A(\bF^j).
\eeq
Since we also have $\dist(Z_j,A(\bC^j)) \geq 1/16$, the combination of \eqref{E:q>0tori22} and \eqref{E:q>0tori23} shows that $A(\bF^j) \supsetneq A(\bC^j)$ as claimed.

\bigskip

\noindent \textbf{Step 5.} \emph{Establishing a Contradiction}. The $L^2$ convergence to the blow-up implies that
\beq  \label{E:q>0tori24}
E_j^{-1} \int_{T_{1/2,7/16}^{\bC^j}(0)} \dist^2(X,\spt\|(\Gamma^j\circ R^j)^{-1}_*\bF^j\|)d\|V^j\|(X) \to 0.
\eeq 
And, using the triangle inequality and \eqref{E:CtoC_Z1} we have
\begin{align} 
E_j^2 &= \int_{T_{1/2,7/16}^{\bC^j}(0)} \dist^2(X,\spt\|(R^j)^{-1}_*\bD^j\|)d\|V^j\|(X) \label{E:q>0tori25}\\
& \leq \int_{T_{1/2,7/16}^{\bC^j}(0)} \dist^2(X,\spt\|\bD^j\|)d\|V^j\|(X)  + c\; \nu^2_{\bD^j,(R^j)^{-1}_*\bD^j}
\end{align}
for some absolute constant $c > 0$. So, since $\nu_{\bD^j,(R^j)^{-1}_*\bD^j} = \nu_{R^j_*\bD^j,\bD^j}$, using \eqref{E:q>0tori12} shows that
\beq \label{E:q>0tori26}
 E_j \leq c\; \int_{T_{1/2,7/16}^{\bC^j}(0)} \dist^2(X,\spt\|\bD^j\|)d\|V^j\|(X)
\eeq
for some constant $c = c(n,k,\bC^{(0)},L) > 0$. Combining \eqref{E:q>0tori24} with \eqref{E:q>0tori26} contradicts \eqref{E:q>0tori11}. This contradiction establishes that \eqref{E:q>0toricontra} does not hold. Therefore, along a subsequence we have that $\Theta_{V^j}(Z) < 2$ for all $Z \in \spt\|V^j\| \cap T^{\bC^j}_{1/2,6/16}(0)$. From here, the specific conclusions of E) follow by applying Theorem \ref{thm:grapdeco}, Allard's Regularity Theorem and \eqref{planlemm1} of Lemma \ref{lemm:planlemm} to get E) ii).
\end{proof}

The next lemma records the observations that one can easily make in the situation in which Hypothesis $6)$ of Lemma \ref{L:q>0tori} does not hold. This is the setting for the most intricate part of the proof of Theorem \ref{L2}, namely Section \ref{SSS:proofofclaimcont}.

On more than on occasion in the sequel we will need the following construction: Given $\bD \in \cC$, we define $\bH_{\bD}$ as follows. If $\dim A(\bD) = n-1$, then we have that $\bD = \sum_{i=1}^4 |\bH_i|$ for half-planes $\bH_i$ and we just set $\bH_{\bD} = \bH_1$. If instead $\dim A(\bD) < n-1$, then $\bD  = |\bQ_1| + |\bQ_2| \in \cP$ and we pick a unit length vector $\omega_{\bD}$ that lies in $\bQ_1$ and is orthogonal to $A(\bD)$. Then let $A_1$ be the $d := (\dim A(\bD)+1)$ - dimensional subspace of $\bQ_1$ that is spanned by $\omega_{\bD}$ and $A(\bD)$. Now define $\bH_{\bD}$ to be the connected component of $A_1 \setminus A(\bD)$ that contains $\omega_{\bD}$. Observe that if $X = (x,y)$ is written in a basis in which $\bD$ is properly aligned, then $\omega_{\bD}$ is the unique direction such that $(|x|\omega_{\bD},y) \in \bH_{\bD}$. 

\begin{corollary} \label{C:q>0fine=coarse} Fix $\eta$ and $\alpha > 0$. There exists a number $\eps_0$ $=$ $\eps_0(n,k,$ $\bC^{(0)},L,\eta) >0$ such that the following is true. If, for some $\eps < \eps_0$, we have that $V \in \cV_L$, $\bC = |\bP_1| + |\bP_2| \in \cP$ and $\bC^{(0)} \in \cP$ satisfy hypotheses $1)$ to $5)$ of Lemma \ref{L:q>0tori}, then: Either one of the conclusions \emph{E)} of \emph{F)} of Lemma \ref{L:q>0tori} holds, or we have:
\begin{enumerate}
\item[\emph{G)}] There exists $\bD = |\bQ_1| + |\bQ_2| \in \cP$ and an open set $\cO \subset \bH_{\bD}$ such that 
\begin{enumerate}[nolistsep, label = \roman*),ref=\roman*)]
\item \label{en:q>0fine=coarse1} $A(\bC) \subsetneq A(\bD)$, 
\item \label{en:q>0fine=coarse2} $E_{V \tiny{\res} T^{\bC}_{1/2,7/16}(0)}(\bD) \leq c \cE_{V \tiny{\res} T^{\bC}_{1/2,7/16}(0)}(\bC)$\\ for some $c = c(n,k,\bC^{(0)},L,\alpha) > 0$,
\item \label{en:q>0fine=coarse3} Either $q_{\bD} = 0$,\\
Or $q_{\bD} > 0$ and we have $c \nu_{\bD,\bC^{(0)}}$ $\geq$ $E_{V \tiny{\res} T^{\bC}_{1/2,7/16}(0)}(\bC^{(0)})$ for some $c$ $=$ $c(n,k,\bC^{(0)},L)$ $>$ $0$ and $E_{V \tiny{\res} T^{\bC}_{1/2,7/16}(0)}(\bD) \leq \alpha \cE_{V \tiny{\res} T^{\bC}_{1/2,7/16}(0)}(\bD)$,
\item \label{en:q>0fine=coarse4} $V \res T^{\bD}(\cO)$ consists only of smooth graphs defined over $\bQ_1$ or $\bQ_2$,
\item \label{en:q>0fine=coarse5} $\bH_{\bD} \cap T^{\bC}_{1/2,6/16}(0) \subset \cO \cup (\underline{\cD}(V))_{\eta}$,
\item \label{en:q>0fine=coarse6} $\cO \cap (\underline{\cD}(V))_{\eta/4} = \emptyset$, 
\end{enumerate}
where 
\[
\underline{\cD}(V) := \bigl\{  X \in \bH_{\bD} \cap T^{\bC}_{1/2,6/16}(0) : \exists Z \in T^{\bD}(\{X\})\ \text{with}\ \Theta_{V}(Z) \geq 2\bigr\}
\]
\end{enumerate}
\end{corollary}

\begin{proof} Consider to begin with a sequence of real numbers $\eps_j \downarrow 0^+$ and sequences $\{V^j\}_{j=1}^{\infty} \in \cV_L$ and $\{\bC^j\}_{j=1}^{\infty} \in \cP$  such that for every $j \geq 1$, the hypotheses are satisfied with $V^j$, $\bC^j$ and $\eps_j$ in place of $V$, $\bC$ and $\eps$ respectively. It suffices to prove that the conclusions of the lemma hold along some subsequence. By arguing as in \textbf{Step 1.} of the proof of Lemma \ref{L:q>0tori}, we see that either we have conclusion F) of Lemma \ref{L:q>0tori} or we have that $\cQ_{V^j \tiny{\res} T^{\bC^j}_{1/2,7/16}(0)}(\bC^{(0)}) \to 0$ as $j \to \infty$. If
\begin{align} \label{E:q>0fine=coarse7}
E_{V^j \tiny{\res} T^{\bC^j}_{1/2,7/16}(0)}(\bC^j)  < \gamma_0\cE_{V^j \tiny{\res} T^{\bC^j}_{1/2,7/16}(0)}(\bC^j)
\end{align}
holds for infinitely many $j$,  where $\gamma_0$ is as in the statement of Lemma \ref{L:q>0tori}, then we can of course get conclusion E) of Lemma \ref{L:q>0tori}. So we will now prove $G)$ under the assumption that \eqref{E:q>0fine=coarse7} does not hold.

Begin by constructing new cones $\{\bD^j\}_{j=1}^{\infty}$ exactly as in \textbf{Step 2.} of the proof of Lemma \ref{L:q>0tori}. The construction is such that we can immediately verify \ref{en:q>0fine=coarse1} and (bearing in mind the negation of \eqref{E:q>0fine=coarse7}), \ref{en:q>0fine=coarse2} of the present lemma. Conclusion \ref{en:q>0fine=coarse3} is obtained by arguing as per \eqref{E:q>0tori14} - \eqref{E:q>0tori16} in Lemma \ref{L:q>0tori}. Set
\[
\underline{\cD}_j := \bigl\{  X \in \bH_{\bD^j} \cap T^{\bC^j}_{1/2,6/16}(0) : \exists Z \in T^{\bD^j}(\{X\})\ \text{with}\ \Theta_{V^j}(Z) \geq 2\bigr\}.
\]
By correct choice of $\bH_{\bD_j}$ and by passing to a subsequence we may assume that: 
\begin{itemize}[nolistsep]
\item There exists a subspace $A$ with $A \subset A(\bC^{(0)})$ for which $d_{\cH}(A(\bC^j)\cap B_2(0),A \cap B_2(0)) \to 0$ as $j \to \infty$.
\item There exists a subspace $A'$ with $A \subsetneq A' \subset A(\bC^{(0)})$ for which $d_{\cH}(A(\bD^j)\cap B_2(0),A' \cap B_2(0)) \to 0$ as $j \to \infty$.
\item There exists a half-space $\bH'$ with $A' \subset \bH' \subset \spt\|\bC^{(0)}\|$ for which $d_{\cH}(\bH_{\bD^j} \cap B_2(0), \bH' \cap B_2(0)) \to 0$ as $j \to \infty$.
\item $0 \leq q_{\bD^j} \equiv q' < q$.
\item $d_{\cH}(\underline{\cD}_j, \underline{\cD}) \to 0$ as $j \to \infty$ for some closed set $\underline{\cD} \subset \bH'$.
\end{itemize}
For sufficiently large $j$ (depending on $\eta > 0$) we have that $(\underline{\cD}_j)_{\eta/4} \subset (\underline{\cD})_{\eta/2} \subset (\underline{\cD}_j)_{\eta}$. Write
\[
\cO_j := \bH^j \cap \bigl( T^{A'}(\bH' \cap (\underline{\cD})_{\eta/2}^c \cap T^{\bC^j}_{1/2,13/32}(0)) \bigr).
\]
Since $(\underline{\cD}_j)_{\eta/4} \subset (\underline{\cD})_{\eta/2}$, we can use Theorem \ref{thm:grapdeco} and Allard's regularity theorem to conclude that $V^j \res T^{A'}(\bH' \cap (\underline{\cD})_{\eta/2}^c \cap T^{\bC^j}_{1/2,6/16}(0))$ consists of a finite collection of smooth graphs, which for sufficiently large $j$ can be taken to be defined over the planes of $\bD^j$. This shows that $\cO_j$ satisfies \ref{en:q>0fine=coarse4}. It is also straightforward to check from this definition that \ref{en:q>0fine=coarse5} and \ref{en:q>0fine=coarse6} hold. 
\end{proof}

\begin{remarks}\label{R:q>0torirema} In \textbf{Step 3.} of Section \ref{SSS:proofofclaimcont}, we will have a situation where 
\[
e(K)^{-n-2}E_{V \tiny{\res} T^{\bC}(K)}(T_{Z*}\bC) < \epsilon
\]
 for sufficiently small $\eps > 0$, where $K \subset \bH_{\bC}$ is a cube with edge length $e(K)$ and that is adjacent to $A(\bC)$ and where $r_{\bC}(Z) < \tfrac{1}{10}e(K)$. This is different from the other situations analysed thus far, but arguments very similar to those that we have seen in \textbf{Step 1.} of the proof of Lemma \ref{L:q>0tori} and  again in the proof of Corollary \ref{C:q>0fine=coarse} show that one of the following situations must occur:
\begin{enumerate}[leftmargin=1cm,nolistsep]
\item[$E)_{\text{adj}}$] $\Theta_V(Z) < 2$ for all $Z \in \spt\|V\| \cap T^{\bC}(K)$. In this case, $V \res T^{\bC}(K')$ (for a smaller cube $K' \subset K$) is a union of smooth graphs over the planes of $T_{Z*}\bC$.
\item[$F)_{\text{adj}}$] $V \res T^{\bC}(K')$ is, up to a small set, a two-valued graph over one of the planes of $T_{Z*}\bC$. In this case, conclusions analogous to \emph{F)} of Lemma \ref{L:q>0tori} hold with $T^{\bC}(K'')$ and $T^{\bC}(K')$ for smaller cubes $K'' \subset K' \subset K$ in place of the regions $T^{\bC}_{1/2,6/16}(0)$ and $T^{\bC}_{1/2,13/32}(0)$ .
\item[$G)_{\text{adj}}$] None of the above: $V \res T^{\bC}(K)$ \emph{is `transverse'}, by which we mean that
\[
\int_{T^{\bC}(K) \cap \{r_{\bC^{(0)}} > \tfrac{1}{8}e(K)\}} \dist^2(X,\spt\|\tau_{\xi*}V\|) d\| \bC\|(X) < c\eps,
\]
for some constant $c = c(n,k,\bC^{(0)},L) > 0$, where $\xi := Z^{\perp_{A(\bC)}}$, \emph{and} there are still singular points present, \emph{i.e.} $\{ Z : \Theta_V(Z) \geq 2\} \cap T^{\bC}(K) \neq \emptyset$.
\end{enumerate} 
 \end{remarks}

\subsection{Proof of Theorem \ref{L2}} \label{SS:ProofofTheoremL2}
Now we begin the proof of Theorem \ref{L2}.  Suppose we have a sequence of numbers $\{\eps_j\}_{j=1}^{\infty}$ with $\eps_j \downarrow 0^+$ and sequences $\{V^j\}_{j=1}^{\infty} \in \cV_L$ and $\{\bC^j\}_{j=1}^{\infty} \in \cC$ satisfying the hypotheses of the theorem with $V^j$, $\bC^j$ and $\eps_j$ in place of $V$, $\bC$ and $\eps$. In this case, using the mass bound ($1)$ of Hypotheses A) and the compactness theorem for stationary integral varifolds, we can extract a subsequence $\{j'\}$ of $\{j\}$ (to which we pass without changing notation) and a stationary integral $n$-varifold $W$ in $B^n_2(0)\times\RR^k$ for which $V^j \to W$. We get from $3)$ and $4)$ of Hypotheses A that $\spt \|W\|  \subset  \spt \|\bC^{(0)}\|$ and $\spt \|W\| \setminus \{r_{\bC^{(0)}} < 1/8\} = \spt \|\bC^{(0)}\| \setminus \{r_{\bC^{(0)}} < 1/8\}$. Furthermore, the Constancy Theorem (\cite[$\S$ 41]{simongmt}) implies that $W$ has constant integer multiplicity on each of the connected components of $\spt \|\bC^{(0)}\| \setminus A(\bC^{(0)})$. Using this together with the mass bound for $W$ (the same mass bound is inherited), we see that this multiplicity must in fact be one everywhere and hence that $W = \bC^{(0)}$.

Now, the upper semicontinuity of $\Theta_{V^j}(\cdot)$ with respect to both the spatial variable and varifold convergence implies that for sufficiently large $j$ (depending on $\tau$, $n$, $k$ and $\bC^{(0)})$) we have that $\{Z : \Theta_{V^j}(Z) \geq 2\} \subset \{r_{\bC^{(0)}} < \tau/4\}$. This means that for any $X \in \spt \|\bC^{(0)}\| \cap B_2(0)\cap \{r_{\bC^{(0)}} \geq \tau \}$, we may apply Allard's Regularity Theorem to $V^j \res B_{\tau/2}(X)$ to deduce that $V^j \res B_{\tau/4}(X) = |\gr (u_X + c)|$, where $u_X \in C^{\infty}(U_X;\bC^{(0)\perp})$ for some domain $U_X \subset \spt \|\bC\| \cap \{r_{\bC^{(0)}} > 0\}$. Since we may do this at each point of the compact set $\spt \|\bC^{(0)}\| \cap \overline{B_{15/8}(0)} \cap \{r_{\bC^{(0)}} \geq \tau/2\}$, we may invoke unique continuation of smooth solutions to the minimal surface system to deduce that provided $j$ is sufficiently large (depending only on $\tau$, $n$, $k$ and $\bC^{(0)}$), we indeed have a function $u$ satisfying conclusion \ref{en:L21}. For the remainder of the proof we drop the index $j$.

\bigskip

\noindent The proof of the estimates \ref{en:L22} to \ref{en:L25} are based on the proof of Lemma 3.4 of \cite{simoncylindrical}, but require substantial modification. Some of these modifications are in the spirit of the proof of Theorem 10.1 of \cite{wickgeneral} and some are new. As per the derivation of (2) and (3) in the proof of Lemma 3.4 of \cite{simoncylindrical}, we let $\psi : [0,\infty) \to \RR$ be a non-increasing smooth function with $\psi \equiv 1$ on $[0,13/16]$ and $\psi \equiv 0$ on $(29/32,\infty)$ and such that $\psi'$ and $\psi''$ are bounded by some absolute constant. Then, using the monotonicity formula and a computation with the first variation formula, we establish the estimates
\begin{align} 
\int_{B_{5/8}(0)} \frac{|X^{\perp_{T_XV}}|^2}{|X|^{n+2}} d\|V\|(X) \leq & c \left( \int_{B_1(0)} \psi^2(R) d\|V\|(X)\right. \nonumber \\ 
& \hspace{1cm} \left. - \int_{B_1(0)} \psi^2(R) d\|\bC\|(X) \right) , \label{L26}
\end{align}
and 
\begin{align}
&\int_{B_1(0)} \left(l + \frac{1}{2}\sum_{j=1}^{m} |e^{\perp_{T_XV}}_{l+k+j}|^2\right)\psi^2(R) d\|V\|(X) \nonumber \\
&  \hspace{3cm}  \leq  c\int_{B_1(0)} | (x,0)^{\perp_{T_XV}}|^2 \left((\psi'(R))^2 + \psi(R)^2\right)d\|V\|(X)  \nonumber \\
&  \hspace{6cm} - 2\int_{B_1(0)} r_{\bC}^2R^{-1} \psi(R)\psi'(R) d\|V\|(X), \label{L27} 
\end{align}
for some constant $c = c(n) > 0$ and where $(x,y)$ is written in a basis in which $\bC$ is properly aligned. Also (as in (6) of the same proof in \cite{simoncylindrical}), it can be verified via a computation using the coarea formula and integration by parts that
\beq \label{L28}
l\int_{B_1(0)} \psi^2(R) d\|\bC\|(X) = -2\int_{B_1(0)} r_{\bC}^2R^{-1} \psi(R) \psi'(R) d\|\bC\|(X). 
\eeq
Subtracting \eqref{L28} from \eqref{L27} gives
\begin{align}
& \frac{1}{2}\int_{B_1(0)}\sum_{j=1}^{m} |e^{\perp_{T_XV}}_{l+k+j}|^2\psi^2(R) d\|V\|(X) \nonumber \\
& + l\int_{B_1(0)} \psi^2(R) d\|V\|(X) - l\int_{B_1(0)} \psi^2(R) d\|\bC\|(X)  \nonumber \\
& \leq  c\int_{B_1(0)} | (x,0)^{\perp_{T_XV}}|^2 \left((\psi'(R))^2 + \psi(R)^2\right)d\|V\|(X) \nonumber \\
&  - 2\int_{B_1(0)} r^2_{\bC}R^{-1} \psi(R)\psi'(R) d\|V\|(X) + 2\int_{B_1(0)} r^2_{\bC}R^{-1} \psi(R) \psi'(R) d\|\bC\|(X) . \label{L29}
\end{align}
We bound the right-hand side of the above line using a covering argument. Set $\bH := \bH_{\bC}$ and define $\cY := \bH \cap \{0 < r_{\bC} < 1/28\} \cap B_{15/16}(0)$. Now pick a countable collection $\cI$ of points $(x,y) \in \cY$ (points here are in a basis in which $\bC$ is properly aligned) such that
\begin{enumerate}[nolistsep, label = C\arabic*), ref = C\arabic*)]
\item \label{en:C1} $\cY \subset \bigcup_{(x,y) \in \cI} B_{|x|/8}(x,y)$ 
\item \label{en:C2} $\cB := \{ B_{15|x|/16}(x,y)\}_{(x,y) \in \cI}$ can be decomposed into the pairwise disjoint sub-collections $\cB_1$,...,$\cB_N$ for some $N = N(n)$. 
\end{enumerate} 
This can be done exactly as in the proof of Theorem 10.1 of \cite{wickgeneral}, immediately preceding (10.18) therein. Then let $\cJ$ be a collection of $J = J(n)$ points $Y \in \cZ := \bH \cap (B_{15/16}(0)\setminus \{r_{\bC} < 1/28\}) $ such that $\cZ \subset \bigcup_{Y \in \cJ} B_{1/64}(Y)$ and define 
\[ \label{L210}
\Psi := \{B_{3|x|/16}(x,y) \cap \bH \}_{(x,y) \in \cI} \cup \{B_{3/128}(Y) \cap \bH \}_{Y \in \cJ}.
\]
Then apply \cite[3.1.13]{federergmt} to the covering $\Psi$ and with the function
\[
h(X) := \frac{1}{20}\sup_{B \in \Psi} \min\{1,\dist(X,B^c)\}
\]
for $X \in \bigcup \Psi$. The result is that we obtain a family of smooth functions $\{\varphi_s\}_{s \in \cS}$, for which
\begin{enumerate}[nolistsep]
\item $\cS$ is a countable subset of $\bigcup \Psi$ and $\varphi_s : \bigcup \Psi \to [0,1]$ for all $s \in \cS$
\item $\{B_{h(s)}(s)\}_{s \in \cS}$ is pairwise disjoint and for each $s \in \cS$, there exists $B \in \Psi$ such that $B_{h(s)}(s) \subset \spt \varphi_s \subset B_{10h(s)}(s) \subset B$.
\item $\sum_{s \in \cS} \varphi_s(X) = 1$ for all $X \in \bigcup \Psi$.
\item $|D \varphi_s(X)| \leq Ch(X)^{-1}$ for each $s \in \cS$ and each $X \in \bigcup \Psi$, where $C = C(n) \in (0,\infty)$.
\end{enumerate}
It follows from 4) and the definition of $h$ that for each $s \in \cS$,
\beq \label{L211}
|D\varphi_s(\tilde{X})| \leq cr_{\bC}(\tilde{X})^{-1},
\eeq
whenever $\tilde{X} \in \bigcup_{(x,y) \in\cI} (B_{5|x|/32} \cap \bH) \cup \bigcup_{Y \in \cJ} (B_{5/256}(Y) \cap \bH)$. For each $s \in \cS$, extend $\varphi_s$ to the rest of $\bH$ by setting $\varphi_s(X) = 0$ for $X \in \bH \setminus \bigcup \Psi$ and let $\tilde{\varphi}_s$ be the smooth extension of $\varphi_s$ to $\RR^{n+k}$ defined by $\tilde{\varphi}_s(x,y) = \varphi_s(|x|\omega_{\bC},y)$. As per (10.22) and (10.23) of \cite[Theorem 10.1]{wickgeneral}, it can now be shown, by only elementary considerations, that there is a fixed constant $M = M(n,k)$ such that for each $(x,y) \in \cI$,
\beq \label{L212}
\#\{s \in \cS : \spt\tilde{\varphi}_s \subset T^{\bC}_{|x|,3|x|/16}(x,y)\} \leq M
\eeq
and for each $Y \in \cJ$,
\beq \label{L213}
\#\{s \in \cS : \spt\tilde{\varphi}_s \subset T^{\bC}(B_{3/128}(Y))\} \leq M.
\eeq
Note (by the construction of $\bH$) that $\varphi_s(X)$ only depends on $r_{\bC}(X) = |x|$ and $X^{\top_{A(\bC)}} = y$. The main claim is then the following:

\begin{claim} \label{L2claim}
Suppose $(\xi,\zeta) \in \cI$ and $s \in \cS$ is such that $\spt \tilde{\varphi}_s \subset T^{\bC}_{|\xi|,3|\xi|/16}(\zeta)$. Then we have the following estimate:
\begin{align}
& \int_{B_1(0)} \tilde{\varphi}_s\; |(x,0)^{\perp_{T_XV}}|^2d\|V\|(X) \nonumber \\
& - 2\int_{B_1(0)}\tilde{\varphi}_s\; r_{\bC}^2R^{-1}\psi(R)\psi'(R)d\|V\|(X) \nonumber \\
&   + 2\int_{B_1(0)} \tilde{\varphi}_s\;  r_{\bC}^2R^{-1}\psi(R)\psi'(R)d\|\bC\|(X) \nonumber \\
& \hspace{3cm} \leq\ c\int_{T^{\bC}_{|\xi|,15|\xi|/16}(0)} \dist^2(X,\spt\|\bC\|) d\|V\|(X). \label{L214}
\end{align}

\end{claim}
It will be clear from the proof of this claim that the following corresponding estimate for points of $\cJ$ also holds: For any $Y \in \cJ$ and any $s \in \cS$ with $\spt \tilde{\varphi}_s \subset T^{\bC}(B_{3/128}(Y))$, we have
\begin{align}
&  \int_{B_1(0)}\tilde{\varphi}_s\; | (x,0)^{\perp_{T_XV}}|^2 d\|V\|(X) \nonumber \\
& - 2\int_{B_1(0)} \tilde{\varphi}_s\; r_{\bC}^2R^{-1} \psi(R)\psi'(R) d\|V\|(X)  \nonumber \\
&   + 2\int_{B_1(0)}\varphi_s\; r_{\bC}^2R^{-1} \psi(R) \psi'(R) d\|\bC\|(X) \nonumber \\
& \hspace{3cm} \leq  c\int_{T^{\bC}(B_{1/32}(Y))} \dist^2(X,\spt\|\bC\|)d\|V\|(X). \label{L215}
\end{align}
Before we move on to the proof of Claim \ref{L2claim}, let us see how it implies the conclusions \ref{en:L22} to \ref{en:L25} in the statement. First note that
\beq \label{L216}
B_{15|x|/16}(x,y) \cap B_{15|x_0|/16}(x_0,y_0) \neq \emptyset \Leftrightarrow T^{\bC}_{|x|,15|x|/16}(y) \cap T^{\bC}_{|x_0|,15|x_0|/16}(y_0) \neq \emptyset,
\eeq
which implies (in light of \ref{en:C2}) that $\{T^{\bC}_{|x|,15|x|/16}(y)\}_{(x,y) \in \cI}$ can be decomposed into $N$ pairwise disjoint sub-collections. Now choose enumerations $\cJ = \{Y_j\}_{j=1}^J$ and $\cI = \{(x_{J+j},y_{J+j})\}_{j=1}^{\infty}$ and for $1 \leq j \leq J$ let 
\beq
\cS_j := \{s \in \cS :\spt\tilde{\varphi}_s \subset T^{\bC}(B_{3/128}(Y_j))\ \text{and}\ \spt\tilde{\varphi}_s \cap \spt\|V\| \neq \emptyset\}
\eeq
and for $j \geq J+1$ let
\beq
\cS_j := \{s \in \cS :\spt\tilde{\varphi}_s \subset T^{\bC}_{|x_j|,3|x_j|/16}(y_j)\ \text{and}\ \spt\tilde{\varphi}_s \cap \spt\|V\| \neq \emptyset\}.
\eeq
And write $\{s \in \cS : \spt\tilde{\varphi_s} \cap \spt\|V\| \neq \emptyset \} = \bigcup_{j=1}^{\infty}\cS^{\prime}_j$, where $\cS^{\prime}_1 = \cS_1$ and $\cS_j^{\prime} = \cS_j\setminus \bigcup_{i=1}^{j-1}\cS_i^{\prime}$. The collections $\cS^{\prime}_j$ are pairwise disjoint and (by \eqref{L212} and \eqref{L213}) we have that $\text{card}\ \cS^{\prime}_j \leq M$ for every $j$. So now, in \eqref{L214} and \eqref{L215} we sum first over $s \in \cS^{\prime}_j$ and then over $j \in \{1,...J\}$ in \eqref{L215} and $j \geq J+1$ in \eqref{L214}. Then adding the two resulting inequalities, using the fact that $\sum_{s \in \cS} \varphi_s(X) = 1$ and the fact that $\{T^{\bC}_{|x|,15|x|/16}(y)\}_{(x,y) \in \cI}$ can be decomposed into $N$ pairwise disjoint sub-collections, we achieve the estimate
\begin{eqnarray}
& & \int_{B_{15/16}(0)}  | (x,0)^{\perp_{T_XV}}|^2 d\|V\|(X) - 2\int_{B_{15/16}(0)}  r_{\bC}^2R^{-1} \psi(R)\psi'(R) d\|V\|(X)  \nonumber \\
& & \hspace{1.5cm} + 2\int_{B_{15/16}(0)}r_{\bC}^2R^{-1} \psi(R) \psi'(R) d\|\bC\|(X) \nonumber \\
& & \hspace{3cm} \leq  c\int_{B_1(0)} \dist^2(X,\spt\|\bC\|)d\|V\|(X) \label{L217}
\end{eqnarray}
We therefore deduce (in light of \eqref{L29} and \eqref{L26}), conclusions \ref{en:L22} and \ref{en:L24}. The other conclusions (\emph{i.e.} \eqref{en:L23} and \eqref{en:L25}) can now be derived exactly as they are in the proof of Lemma 3.4 of \cite{simoncylindrical}. We therefore shift our attention to the proof of Claim \ref{L2claim}.

\subsubsection{Proof of Claim \ref{L2claim}} \label{SSS:proofofclaim} Let $m' := \dim A(\bC) + 1$, \emph{i.e.} the dimension of $\bH$. For some $\delta$ to eventually be determined depending only on $n$, $k$, $\bC^{(0)}$ and $L$, we write $\cY = \cU \cup \cW$, where $\cW$ is the set of points $(\xi,\zeta) \in \cY$ where
\beq \label{L218}
(15|\xi|/16)^{-m'-2}E_{V \tiny{\res} T^{\bC}_{|\xi|,7|\xi|/8}(\zeta)} (\bC) \geq \delta,
\eeq
and $\cU := \cY\setminus \cW$. If $q_{\bC} >  0$, then for some $\beta$ to eventually be determined depending only on $n$, $k$, $\bC^{(0)}$ and $L$, we write $\cU = \cU_1 \cup \cU_2$, where $\cU_1$ is the set of points $(\xi,\zeta) \in \cU$ where
\beq \label{L219}
E_{V \tiny{\res} T^{\bC}_{|\xi|,7|\xi|/8}(\zeta)} (\bC) < \beta \cE_{V \tiny{\res} T^{\bC}_{|\xi|,7|\xi|/8}(\zeta)} (\bC) 
\eeq
and
$\cU_2$ is the set of points $(\xi,\zeta) \in \cU$ where
\beq \label{L220}
E_{V \tiny{\res} T^{\bC}_{|\xi|,7|\xi|/8}(\zeta)} (\bC) \geq \beta \cE_{V \tiny{\res} T^{\bC}_{|\xi|,7|\xi|/8}(\zeta)} (\bC). 
\eeq
Suppose to begin with that $(\xi,\zeta) \in \cW$. In this case, using the monotonicity formula and the fact that $r^2_{\bC}\vert_{T^{\bC}_{|\xi|,3|\xi|/16}(\zeta)} \leq c(n)|\xi|^2$, we easily have that 
\begin{align}
& \int_{B_1(0)} \tilde{\varphi}_s\; |(x,0)^{\perp_{T_XV}}|^2d\|V\|(X) - 2\int_{B_1(0)}\tilde{\varphi}_s\; r_{\bC}^2R^{-1}\psi(R)\psi'(R)d\|V\|(X) \nonumber \\
&   + 2\int_{B_1(0)} \tilde{\varphi}_s\; r_{\bC}^2R^{-1}\psi(R)\psi'(R)\|\bC\|(X)\ \leq\ c|\xi|^{m'+2}, \label{L221}
\end{align}
and so the required estimate follows immediately from \eqref{L218}. Suppose now that $(\xi,\zeta) \in \cU$. Assume to begin with that either we have $q_{\bC} = 0$ or we have both $q_{\bC} > 0$ and $(\xi,\zeta) \in \cU_1$. Under these hypotheses we will be able to appeal to Lemma \ref{q=0smalexcetori} (if $q_{\bC} = 0$) or Lemma \ref{L:q>0tori} (if $q_{\bC} > 0$ and $(\xi,\zeta) \in \cU_1$). Let us do exactly that. Define $\tilde{V} := ((\eta_{(0,\zeta),2|\xi|})_* V) \res (B^n_2(0)\times \RR^k)$. Since the negation of \eqref{L218} holds, a change of variables shows that
\beq \label{L222}
E_{\tilde{V} \tiny{\res} T^{\bC}_{1/2,7/16}(0)}(\bC) < c\delta,
\eeq
for some constant $c = c(n) > 0$ and similarly, if $q_{\bC} > 0$, we get from \eqref{L219} that
\beq
E_{\tilde{V} \tiny{\res} T^{\bC}_{1/2,7/16}(0)}(\bC) < c\beta \cE_{\tilde{V} \tiny{\res} T^{\bC}_{1/2,7/16}(0)}(\bC) .
\eeq
So, by correct choice of $\delta$ and $\beta$, we have that $\tilde{V}$, $\bC$ and $\bC^{(0)}$ satisfy the hypotheses of Lemma \ref{q=0smalexcetori} (if $q_{\bC} = 0$) or Lemma \ref{L:q>0tori} (if $q_{\bC} > 0$ and $(\xi,\zeta) \in \cU_1$). After applying the relevant Lemma, there are various possible different conclusions (\emph{i.e.} conclusions A), B), C) or D) of Lemma \ref{q=0smalexcetori} and conclusions E) and F) of Lemma \ref{L:q>0tori}). These possibilities may be summarized by saying that we may now write $V \res T^{\bC}_{|\xi|,|\xi|/2}(\zeta) = \sum_{i \in I} V_i$ where for each $i \in I$: $V_i$ is a stationary varifold in $T^{\bC}_{|\xi|,|\xi|/2}(\zeta)$ and there exists a plane $\bQ_{(i)}$ and a domain $\Omega_i \subset \bQ_{(i)}$ such that either 
\begin{itemize}
\item $\bQ_{(i)} = \bP_i$ for some $i \in \{1,2\}$ (if $\bC = |\bP_1| + |\bP_2| \in \cP$), \\or 
\item $\bQ_{(i)}\supset \bH_j \supset \Omega_i$ for some $j \in \{1,\dots,4\}$ (if $\bC = \sum_{j=1}^4|\bH_j| \in \cC_{n-1}\setminus \cP_{n-1}$)
\end{itemize}
and such that either
\begin{itemize}
\item $V_i$ is a smooth single-valued graph over $\Omega_i$ (conclusions A),C) and E)), or
\item $V_i$ is (`up to a small set') a two-valued Lipschitz graph over $\Omega_i$ (conclusions B),D) and F))
\end{itemize} 
In describing this part of the argument, let us assume that $\bC \in \cP$ and that $i \in I$ is such that $V_i$ is, up to a small set, a two-valued graph. The other cases here are strictly simpler. Let $\Sigma_i \subset \Omega_i$ and $u_i$ be the measurable set and function the existences of which are asserted by conclusions B),D) or F). Now, for $X = (x,y) \in ((\Omega_i\setminus \Sigma_i) \times \bQ_{(i)}^{\perp}) \cap \spt \|V_i\| \cap T^{\bC}_{|\xi|,|\xi|/2}(\zeta)$, write $X' = (x',y)$ for the nearest point projection of $X$ to $\bQ_{(i)}$, so that $u_i(x',y) = (x-x',0)$. Then notice that
\begin{eqnarray*}
(x,0)^{\perp_{T_XV_i}} &=& u_i^{\alpha}(x',y) + (x,0)^{\perp_{T_XV_i}} - u_i^{\alpha}(x',y) \\
&=& u_i^{\alpha}(x',y) + (\bp^{\perp}_{T_XV_i} - \bp^{\perp}_{\bQ_{(i)}})(x,0)
\end{eqnarray*}
(where $\alpha \in \{1,2\}$ for $u_i(X) = \{u_i^1(X),u_u^2(X)\}$), and hence
\beq \label{L223}
|(x,0)^{\perp_{T_XV_i}}|^2 \leq 2(|u_i^{\alpha}(X')|^2 + \|\bp^{\perp}_{T_XV_i} - \bp^{\perp}_{\bQ_{(i)}}\|^2r^2_{\bC}(X)).
\eeq
Recall also that in each of the conclusions B), D) and F) we get the estimate
\begin{align} 
\cH^n(\Sigma_i) + \|V_i\|(\Sigma_i \times &\bQ_{(i)}^{\perp}) \nonumber \\
& \leq c|\xi|^{-2}\int_{T^{\bC}_{|\xi|,15|\xi|/16}(\zeta)} \dist^2(X,\spt\|\bC\|) d\|V\|(X). \label{L224}
\end{align}
So, we use \eqref{L224} on the non-graphical set. And on the graphical set $((\Omega_i\setminus \Sigma_i) \times \bQ_{(i)}^{\perp}) \cap \spt \|V_i\|$ we use \eqref{L223} and \cite[Lemma 22.2]{simongmt} (the standard estimate for tilt-excess in term of height-excess). And in both cases, we also use the point-wise estimate $r^2_{\bC}\vert_{\Omega_i} \leq |\xi|^2$ to get that:
\begin{align} 
\int_{T^{\bC}_{|\xi|,|\xi|/4}(\zeta)}\tilde{\varphi}_s\; |(x,0)^{\perp_{T_XV_i}}|^2 &d\|V_i\|(X) \nonumber \\
&\leq c \int_{T^{\bC}_{|\xi|,|\xi|/2}(\zeta)} \dist^2(X,\spt\|\bQ_{(i)}\|) d\|V_i\|(X). \label{L225}
\end{align}
When $V_i$ is a single-valued graph, one can use elliptic estimates for the supremum of $|u_i|$ and $|Du_i|$ in terms of $L^2$ norm to get this estimate. Finally we can conclude that
\begin{align} 
\int_{B_1(0)}\tilde{\varphi}_s\; |(x,0)^{\perp_{T_XV}}|^2 &d\|V\|(X) \nonumber \\
& \leq c \int_{T^{\bC}_{|\xi|,15|\xi|/16}(\zeta)} \dist^2(X,\spt\|\bC\|) d\|V\|(X), \label{L226}
\end{align}
where, in the case $q_{\bC} > 0$, we have used the estimate \ref{en:q>0tori5} of F) of Lemma \ref{L:q>0tori}. This completely handles the first term of the estimate in Claim \ref{L2claim} when $q_{\bC} = 0$ and $(\xi,\zeta) \in \cU$ and when $q_{\bC} > 0$ and $(\xi,\zeta) \in \cU_1$. 

\bigskip

\noindent Now for the other terms of the estimate in Claim \ref{L2claim}. This part is similar to lines (10.24)-(10.27) of \cite{wickgeneral}. First, we can once again discard the non-graphical part of $V_i$, \emph{i.e.} $\spt\|V_i\| \cap (\Sigma_i \times \bQ_{(i)}^{\perp})$ and subsume it into an error term, the size of which is controlled by \eqref{L224}. Then, we use the area formula to write the graphical part $\spt\|V_i\| \cap ((\Omega_i\setminus \Sigma_i) \times \bQ_{(i)}^{\perp})$ as an integral over $\Omega_i\setminus \Sigma_i$: 
\begin{eqnarray*} 
& & \int_{T^{\bC}_{|\xi|,|\xi|/2}(\zeta)} \tilde{\varphi}_s(X)r^2_{\bC}R^{-1} \psi(R)\psi'(R) d\|V_i\|(X) \\
& = &  \sum_{\alpha=1,2}\int_{\Omega_i\setminus \Sigma_i} \varphi_s(r_{\bC,u_i^{\alpha}},y)\; r^2_{\bC,u_i^{\alpha}} R_{u_i^{\alpha}}^{-1} \psi(R_{u_i^{\alpha}})\psi'(R_{u_i^{\alpha}}) |\cJ(u_i^{\alpha})| d\cH^n(X) + E_1
\end{eqnarray*}
where $r_{\bC,u_i^{\alpha}}(x,y) := \sqrt{|x|^2 + |u_i^{\alpha}(x,y)|^2}$, $R_{u_i^{\alpha}}(X)$ $:=$ $\sqrt{|X|^2 + |u_i^{\alpha}(X)|^2}$ and $|\cJ(u_i^{\alpha})|$ $:=$  $\det(\delta_{\alpha\beta} + \Sigma_{\kappa}D_{\alpha}u_i^{\alpha,\kappa}D_{\beta}u_i^{\alpha,\kappa})^{1/2}$, where $u_i^{\alpha}(X) = (u_i^{\alpha,1},...,u_i^{\alpha,k})$. Continuing with this estimate we have that
\begin{align*}
&  \sum_{\alpha=1,2}\int_{\Omega_i\setminus \Sigma_i} \varphi_s(r_{\bC,u_i^{\alpha}},y)\; r^2_{\bC,u_i^{\alpha}} R_{u_i^{\alpha}}^{-1} \psi(R_{u_i^{\alpha}})\psi'(R_{u_i^{\alpha}}) |\cJ(u_i^{\alpha})| d\cH^n(X) + E_1 \\
& =   \sum_{\alpha=1,2}\int_{\Omega_i\setminus \Sigma_i} \varphi_s(r_{\bC,u_i^{\alpha}},y)\; r^2_{\bC} R_{u_i^{\alpha}}^{-1} \psi(R_{u_i^{\alpha}})\psi'(R_{u_i^{\alpha}}) d\cH^n(X) + E_2 \\
& =  \sum_{\alpha=1,2}\int_{\Omega_i\setminus \Sigma_i} \varphi_s(r_{\bC,u_i^{\alpha}},y)\; r^2_{\bC} R^{-1} \psi(R)\psi'(R) d\cH^n(X) \\
& + \sum_{\alpha=1,2}\int_{\Omega_i\setminus \Sigma_i}\varphi_s(r_{\bC,u_i^{\alpha}},y)\; r_{\bC}^2 \Bigl[\; R_{u_i^{\alpha}}^{-1} \psi(R_{u_i^{\alpha}})\psi'(R_{u_i^{\alpha}})  \\
& \hspace{5cm}  - R^{-1}\psi(R)\psi'(R)\; \Bigr]d\cH^n(X) + E_2, \\
& =  \sum_{\alpha=1,2}\int_{\Omega_i\setminus \Sigma_i} \varphi_s(r_{\bC,u_i^{\alpha}},y)\;r_{\bC}^2 R^{-1}\psi(R)\psi'(R) d\cH^n(X) + E_3,
\end{align*}
where, by using the estimate \eqref{L224}, the mean-value inequality, the bound for tilt-excess in terms of height-excess, the pointwise identity $|u^{\alpha}_i(X)| = \dist(X+u^{\alpha}_i(X),\bQ_{(i)})$ and $r^2_{\bC}\vert_{\Omega_i} \leq c|\xi|^2$, and  \ref{en:q>0tori5} of F) of Lemma \ref{L:q>0tori} (if $q_{\bC} > 0$), we have that
\beq \label{L227}
|E_1|, |E_2|, |E_3| \leq c\int_{T^{\bC}_{|\xi|,15|\xi|/16}(\zeta)} \dist^2(X,\spt\|\bC\|)d\|V\|(X)
\eeq
for some constant $c = c(n,k,\bC^{(0)},L) > 0$. Then finally
\begin{align*}
&   \sum_{\alpha=1,2}\int_{\Omega_i\setminus \Sigma_i} \varphi_s(r_{\bC,u_i^{\alpha}},y)\;r_{\bC}^2 R^{-1}\psi(R)\psi'(R) d\cH^n(X) + E_3\\
&=  2\int_{\Omega_i\setminus \Sigma_i} \varphi_s(r_{\bC},y)\;r_{\bC}^2 R^{-1}\psi(R)\psi'(R) d\|\bC\|(X) \\
&   + \sum_{\alpha=1,2} \int_{\Omega_i\setminus \Sigma_i} \left(\varphi_s(r_{\bC,u_i^{\alpha}},y)-\varphi_s(r_{\bC},y)\right)\; r_{\bC}^2 R^{-1}\psi(R)\psi'(R) d\cH^n(X) + E_3.
\end{align*}
We can again use the mean value inequality together with \eqref{L211}, to bound the absolute value of the second term by 
\[
c\int_{T^{\bC}_{|\xi|,15|\xi|/16}(\zeta)} \dist^2(X,\spt\|\bC\|)d\|V\|(X).
\]
Thus we achieve the estimate
\begin{align}  \label{L228}
 -2\int_{T^{\bC}_{|\xi|,|\xi|/2}(\zeta)} \tilde{\varphi}_s\; r^2_{\bC} & R^{-1} \psi(R)\psi'(R) d\|V_i\|(X) \nonumber \\
 + 2\int_{\Omega_i} & \varphi_s\; r_{\bC}^2 R^{-1}\psi(R)\psi'(R) d\|\bC\|(X) \nonumber \\
&  \leq  c\int_{T^{\bC}_{|\xi|,15|\xi|/16}(\zeta)} \dist^2(X,\spt\|\bC\|)d\|V\|(X).
\end{align}
By a similar argument for the case in which $V_i$ is a single-valued graph and then by summing over $i$, we get that
\begin{align}  \label{L229}
 -2\int_{T^{\bC}_{|\xi|,|\xi|/2}(\zeta)} \tilde{\varphi}_s\; r^2_{\bC}& R^{-1} \psi(R)\psi'(R) d\|V\|(X) \nonumber \\
 + 2\int_{T^{\bC}_{|\xi|,|\xi|/2}(\zeta)} & \varphi_s\; r_{\bC}^2 R^{-1}\psi(R)\psi'(R) d\|\bC\|(X) \nonumber \\
&  \leq  c\int_{T^{\bC}_{|\xi|,15|\xi|/16}(\zeta)} \dist^2(X,\spt\|\bC\|)d\|V\|(X),
\end{align}
which handles the second two terms in \eqref{L214} and therefore completes the proof of Claim \ref{L2claim} when $q_{\bC} = 0$ and $(\xi,\zeta) \in \cU$ and when $q_{\bC} > 0$ and $(\xi,\zeta) \in \cU_1$. 

\medskip

\noindent This completes the proof of Theorem \ref{L2} in the $q_{\bC} = 0$ case. It remains to complete the proof of Claim \ref{L2claim} in the case where $q_{\bC} > 0$ and $(\xi,\zeta) \in \cU_2$.

\subsubsection{Proof of Claim \ref{L2claim} continued: Estimates In Terms of a Coarser Excess} \label{SSS:proofofclaimcont}
So here we suppose that $q_{\bC} > 0$ and $(\xi,\zeta) \in \cU_2$. Define $\tilde{V} := ((\eta_{(0,\zeta),2|\xi|})_* V) \res (B^n_2(0)\times \RR^k)$. Notice that we have
\beq \label{proofofclaimcont1}
E_{\tilde{V} \tiny{\res} T^{\bC}_{1/2,7/16}(0)}(\bC) < c\delta,
\eeq
for some constant $c = c(n) > 0$ and we note (from \eqref{L220}) that
\beq \label{proofofclaimcont2}
E_{\tilde{V} \tiny{\res} T^{\bC}_{1/2,7/16}(0)}(\bC) \geq c\beta \cE_{\tilde{V} \tiny{\res} T^{\bC}_{1/2,7/16}(0)}(\bC) .
\eeq
Given values for the parameters $\eta$ and $\alpha > 0$, an appropriate choice of $\delta $ will ensure that that $\tilde{V}$, $\bC$ and $\bC^{(0)}$ satisfy the hypotheses of Corollary \ref{C:q>0fine=coarse} in $T^{\bC}_{1/2,7/16}(0)$ with these values. Therefore either E) or F) of Lemma \ref{L:q>0tori} holds or G) of Corollary \ref{C:q>0fine=coarse} holds. If E) or F) of Lemma \ref{L:q>0tori} holds, then we can argue exactly as we have just done for the $(\xi,\zeta) \in \cU_1$ case in order to establish the estimate of the claim. So assume that G) of Corollary \ref{C:q>0fine=coarse} holds. This implies in particular, via the induction Hypothesis $H(\bC,\bC^{(0)})$, that the estimates of Corollary \ref{cor:L2Z} hold with $\bD$ in place of $\bC$. The idea now is to bound the left-hand side of \eqref{L214} by $E_{\tilde{V}\tiny{\res} T^{\bC}_{1/2,7/16}(0)}(\bD)$, which (by \ref{en:q>0fine=coarse2} of Corollary \ref{C:q>0fine=coarse} and \eqref{proofofclaimcont2}) is controlled by $E_{\tilde{V} \tiny{\res} T^{\bC}_{1/2,7/16}(0)}(\bC)$. Let us write
\beq \label{proofofclaimcont3}
g(X) := 2 \tilde{\varphi}_s(X') r_{\bC}^2(X')R(X')^{-1}\psi(R(X'))\psi'(R(X')),
\eeq
where $X' := 2|\xi| X + (0,\zeta)$. Using the fact that $A(\bD) \supset A(\bC)$ and the fact that $g$ depends only on $R(X)$ and $X^{\top_{A(\bC)}}$, we have
\beq \label{proofofclaimcont4}
\int_{T^{\bC}_{1/2,7/16}(0)} g(X) d\|\bC\|(X) = \int_{T^{\bC}_{1/2,7/16}(0)} g(X) d\|\bD\|(X). 
\eeq
Therefore it suffices to prove the estimate
\begin{align} \label{proofofclaimcont5}
\int_{T^{\bC}_{1/2,5/16}(0)} & |(x,0)^{\perp_{T_X\tilde{V}}} |^2 d\|\tilde{V}\|(X) \nonumber \\
 - \int_{T^{\bC}_{1/2,5/16}(0)} & g(X)d\|\tilde{V}\|(X) + \int_{T^{\bC}_{1/2,5/16}(0)} g(X) d\|\bD\|(X) \nonumber \\
 & \leq \int_{T^{\bC}_{1/2,7/16}(0)} \dist^2(X,\spt\|\bD\|)d\|\tilde{V}\|(X).
\end{align}
Let us use the shorthand
\[
\tilde{E}^2 := \int_{T^{\bC}_{1/2,7/16}(0)} \dist^2(X,\spt\|\bD\|)d\|\tilde{V}\|(X)
\]
and, for any region $S \subset \bH_{\bD}$, 
\begin{align*}
\cI(S) := \int_{T^{\bD}(S)} & |(x,0)^{\perp_{T_X\tilde{V}}}|^2 d\|\tilde{V}\|(X)  - \int_{T^{\bD}(S)}  g(X)d\|\tilde{V}\|(X) \\
&+ \int_{T^{\bD}(S)} g(X) d\|\bD\|(X).
\end{align*}

\noindent \textbf{Step 1.} \emph{A cubical decomposition.} Write
\begin{align*}
\cD(\tilde{V}) &:= \{ Z \in \spt\|\tilde{V}\| \cap T^{\bC}_{1/2,6/16}(0) : \Theta_{\tilde{V}}(Z) \geq 2\}, \\
\underline{\cD}(\tilde{V}) &:= \{X \in \bH_{\bD} \cap T^{\bC}_{1/2,6/16}(0) : T^{\bD}(\{X\}) \cap \cD(\tilde{V}) \neq \emptyset\}.
\end{align*}
We will define a partition of a region of the form $(\underline{\cD}(\tilde{V}))_{\eta} \cap \bH_{\bD}$ into cubes that is reminiscent of the Whitney decomposition of the complement of $\underline{\cD}(\tilde{V}) $. Broadly speaking, the initial part of the construction is a modification of \cite[pp. 167-169]{steinsingular}. There will be various parameters involved, all of which can eventually be chosen depending only on $n$, $k$, $\bC^{(0)}$ and $L$. 

Firstly observe that there is a constant $c = c(n,k,\bC^{(0)},L) > 0$, so that the parameter $t > 0$ defined by
\beq \label{proofofclaimcont7}
t^2 := \begin{cases} c\tilde{E}^2 & \text{if}\ q_{\bD} = 0 \\
c\bigl( \nu_{\bD,\bC^{(0)}}^2\bigr)^{-1} \tilde{E}^2 & \text{if}\ q_{\bD} > 0, 
\end{cases}
\eeq
is such that
\begin{align} \label{proofofclaimcont8}
\cD(\tilde{V}) \subset (A(\bD))_t.
\end{align}
This follows directly from the estimate of \ref{en:L22} of Corollary \ref{cor:L2Z}. 

Write $d := \dim A(\bD) + 1 = \dim \bH_{\bD}$. Since $\bH_{\bD}$ is a half-space we can write it in the form $\{(x_1,\dots,x_d) : x_d > 0\}$ so that it makes sense to talk of the lattice of points in $\bH_{\bD}$ whose coordinates are integral. This lattice determines a collection of open cubes $\cM_0$ of unit edge length, the vertices of which are points of the lattice. Then $\cM_0$ defines a sequence $\cM_j := 2^{-j}\cM_0$ for $j=1,2,\dots$ of collections of cubes, where each cube $Q \in \cM_p$ determines $2^d$ open cubes in $\cM_{p+1}$ by bisecting the sides of $Q$. We then set
\beq
\Omega_j := \{X \in \bH_{\bD} : 2^{-j+1}\sqrt{n} < \dist(X,\underline{\cD}(\tilde{V})) \leq 2^{-j+2}\sqrt{n}\},
\eeq
Now set
\[
\mathscr{F}_{\text{initial}} := \cM_{J_2} \cup \bigcup_{j=J_1}^{J_2 - 1} \{ Q \in \cM_j : Q \cap \Omega_j \neq \emptyset\}.
\]
where $J_1 = \min \{j : 2^{-j} < s\}$ and $J_2 = \max \{j : 2^{-j} \geq 10t\}$. We see that $J_1$ and $J_2$ are are defined so that $10t \leq e(Q) < s$ for every $Q \in \mathscr{F}_{\text{initial}}$, where $e(Q)$ is the edge length of $Q$. Notice that if $Q_1$, $Q_2 \in \mathscr{F}_{\text{initial}}$ have $e(Q_1) \neq e(Q_2)$ and $Q_1 \cap Q_2 \neq \emptyset$, then one of them contains the other. This implies that for every cube $Q \in \mathscr{F}_{\text{initial}}$, there is a unique largest cube in $\mathscr{F}_{\text{initial}}$ that contains $Q$. Call such a cube maximal and observe that any two maximal cubes have pairwise disjoint interiors. Let $\mathscr{F}$ denote the collection of all maximal cubes $Q \in \mathscr{F}_{\text{initial}}$ that intersect $\bH_{\bD} \cap T^{\bC}_{1/2,6/16}(0)$ and are such that
\[
Q \not\subset \bigl\{X \in \bH_{\bD} :  r_{\bD}(X) < \tilde{E}^2 \bigr\}.
\]
Let $\mathscr{F}_{\text{adj}}$ denote those cubes in $\mathscr{F}$ that are adjacent to $A(\bD)$. 

We build up a number of observations about $\mathscr{F}$. The first observation we make is that 
\begin{align}\tag*{$\mathscr{F}$(i)} \label{Q1}
d_{\cH}(Q, \underline{\cD}(\tilde{V})) \leq 4\diam(Q) \quad \text{for every}\ Q \in \mathscr{F}.
\end{align}
To see this: If $Q \in \mathscr{F} \cap \cM_j$ for $j < J_2$, then $\diam Q = 2^{-j}\sqrt{n}$ and there exists $X \in Q \cap \Omega_j$. Thus $d_{\cH}(Q,\underline{\cD}(\tilde{V}))\leq \dist(X,\underline{\cD}(\tilde{V})) \leq 2^{-j+2}\sqrt{n} = 4\diam(Q)$. So suppose $Q \in \mathscr{F} \cap \cM_{J_2}$. If $Q' \cap Q \neq \emptyset$ for some $Q' \in \mathscr{F}_{\text{initial}} \cap \cM_j$ for some $j \leq J_2-1$, then $Q' \supset Q$, in which case $Q$ would not be maximal. Therefore $Q \subset \cup_{j \geq J_2}\Omega_j$ and so there is $X \in \Omega_j \cap Q$ for some $j \geq J_2$, which means that $d_{\cH}(Q,\underline{\cD}(\tilde{V})) \leq \dist(X,\underline{\cD}(\tilde{V})) \leq 2^{-J_2+2}\sqrt{n} = 4\diam (Q)$ as before.

For each cube $Q \in \mathscr{F}$, pick $Z_Q \in \spt\|\tilde{V}\|\cap T^{\bC}_{1/2,6/16}(0)$ with $\Theta_{\tilde{V}}(Z_Q) \geq 2$ satisfying
\beq
\dist(Z_Q,T^{\bD}(Q)) = \inf_{\substack{Z \in \spt\|\tilde{V}\| \cap T^{\bC}_{1/2,6/16}(0) \\ :\Theta_{\tilde{V}}(Z) \geq 2}}\dist(Z,T^{\bD}(Q)).
\eeq
Then we claim that there exists a constant $\tilde{c}$ $=$ $\tilde{c}(n,k,L,\bC^{(0)})$ $>$ $0$ such that 
\begin{align} \tag*{$\mathscr{F}$(ii)} \label{Q4}
&T^{\bD}(\tfrac{14}{10}Q) \subset B_{\tilde{c} e(Q)}(Z_Q) \subset T^{\bC}_{1/2,6/16+1/200}(0)\\ 
&\text{for every}\ Q \in \mathscr{F}, \nonumber
\end{align}
where, for $\rho > 0$, the cube $\rho Q$ is defined to be the cube with the same centre point as $Q$ but with $e(\rho Q) = \rho e(Q)$. To see this: Notice that from \ref{Q1} we have $d_{\cH}(T^{\bD}(Q),T^{\bD}(\{Z_Q\})) \leq c e(Q)$ for $c = c(n) > 0$. Then, using \eqref{proofofclaimcont8} and the fact that $10 t \leq e(Q)$ we have that $\diam T^{\bD}(\{Z_Q\}) \leq c t \leq c e(Q)$. Combining these facts gives $T^{\bD}(\tfrac{14}{10}Q) \subset B_{\tilde{c}e(Q)}(Z_Q)$ for a constant $\tilde{c}$ as claimed. Then, since $e(Q) < s$, we can suppose that $s > 0$ is such that $\tilde{c} e(Q) < 1/200$ which proves \ref{Q4}.

\wl

\noindent \begin{remarks} \label{R:tchoice} Notice that in checking \ref{Q4}, we chose $s$ to be small (depending on the allowed parameters $n$, $k$ $\bC^{(0)}$ and $L$). Since $10 t \leq e(Q) \leq s$ for every cube $Q \in \mathscr{F}_{\text{initial}}$, this means that we need to be able to choose $t$ smaller than a fixed factor times $s$. In the $q_{\bD} > 0$ case, we have that $t \leq c\alpha$ (from \ref{en:q>0fine=coarse3} of Corollary \ref{C:q>0fine=coarse}) and so this is achieved by choosing $\alpha$ small. In the $q_{\bD} = 0$ case, we have $t \leq c\delta$, so it follows just because we will choose $\delta$ small. \end{remarks}

\wl

Finally, we claim that
\begin{align}  \tag*{$\mathscr{F}$(iii)} \label{Q5}
d_{\cH}(A(\bD),Q) \leq c e(Q) \quad \text{for every}\ Q \in \mathscr{F}.
\end{align}
To see this: Notice that $d_{\cH}(A(\bD),T^{\bD}(\{Z_Q\})) \leq \diam T^{\bD}(\{Z_Q\})$, which in the explanation of \ref{Q4} above we showed was at most $c e(Q)$. Then observe that \ref{Q1} implies that $d_{\cH}(T^{\bD}(\{Z_Q\},Q) \leq 4\sqrt{n}e(Q)$. Thus \ref{Q5} is proved by using the triangle inequality $d_{\cH}(A(\bD),Q)$ $\leq$ $d_{\cH}(A(\bD),T^{\bD}(\{Z_Q\}))$ $+$ $d_{\cH}(T^{\bD}(\{Z_Q\},Q)$. 

\bigskip

\noindent \textbf{Step 2.} \emph{Regions  in which $\tilde{V}$ has large excess}.
Recall that $\tau_Y(X) = X - Y$ and $T_Y(X) = X+Y$ and write $\xi_Q := Z_Q^{\perp_{A(\bD)}}$. First we suppose that $Q \in \mathscr{F}$ is such that
\begin{align} \label{proofofclaimcont9}
e(Q)^{-n-2}\int_{T^{\bD}(\tfrac{13}{10}Q)} \dist^2(X,\spt\|\bD\|)d\|\tau_{\xi_Q *}\tilde{V}\|(X) \geq \delta',
\end{align}
for some $\delta' > 0$. Using \eqref{proofofclaimcont8} together with the fact that $e(Q) \geq 10t$, we have that $T^{\bD}(\tfrac{14}{10}Q) \supset T_{\xi_Q}\bigl(T^{\bD}(\tfrac{13}{10}Q)\bigr)$ whence
\begin{align} \label{proofofclaimcont9a}
e(Q)^{-n-2}\int_{T^{\bD}(\tfrac{14}{10}Q)} \dist^2(X,\spt\|T_{Z_Q *}\bD\|)d\|\tilde{V}\|(X) \geq \delta'.
\end{align}
In this case, we can use a crude supremum bound for the integrands to get that
\begin{align} \label{proofofclaimcont10}
\cI\bigl(\tfrac{14}{10}Q\bigr) \leq ce(Q)^d. 
\end{align} 
And, for such a cube $Q$, using \eqref{proofofclaimcont9a} followed by the first inclusion of \ref{Q4}, we can estimate thus:
\begin{align}
& ce(Q)^d \nonumber \\
& \leq  c\delta'^{-1} e(Q)^{-2-n+d}\int_{T^{\bD}(\tfrac{14}{10}Q)} \dist^2(X,\spt\|T_{Z_Q*}\bD\|)d\|\tilde{V}\|(X) \nonumber \\
& \leq  c \delta'^{-1} e(Q)^{d-1/4}  \int_{B_{\tilde{c} e(Q)}(Z_Q)} \frac{\dist^2(X,\spt\|T_{Z_Q*}\bD\|)}{|X-Z_Q|^{n+7/4}} d\|\tilde{V}\|(X). \nonumber 
\end{align}
Now, since $B_{\tilde{c}e(Q)}(Z_Q)$ $\subset$ $B_{1/200}(Z_Q)$, we can use \ref{en:L2Z3} of Corollary \ref{cor:L2Z} at a fixed scale (\emph{i.e.} with $\rho = 1/100$) to see that this is at most
\begin{align} \label{proofofclaimcont11}
c \delta'^{-1} e(Q)^{d-1/4} \int_{T^{\bC}_{1/2,6/16}(0)} \dist^2(X,\spt\|T_{Z_Q*}\bD\|)d\|\tilde{V}\|(X) 
\end{align}
And using \eqref{E:CtoC_Z1} and Remark \ref{R:CtoC_Z}, we achieve the estimate
\begin{align} \label{proofofclaimcont12}
\cI\bigl(\tfrac{14}{10}Q\bigr)  \leq c e(Q)^{d-1/4}\tilde{E}^2
\end{align}
where $c = c(n,k,L,\bC^{(0)},\delta') > 0$. 

\bigskip

\noindent \textbf{Step 3.} \emph{Regions in which $\tilde{V}$ is graphical}. Now suppose that $Q \in \mathscr{F}$ is such that
\begin{align} \label{proofofclaimcont13}
e(Q)^{-n-2}\int_{T^{\bD}(\tfrac{13}{10}Q)} \dist^2(X,\spt\|\bD\|)d\|\tau_{\xi_Q *}\tilde{V}\|(X) < \delta'.
\end{align}
Since $T^{\bD}(\tfrac{12}{10}Q) \subset T_{\xi_Q}\bigl(T^{\bD}(\tfrac{13}{10}Q)\bigr)$, this implies that
\begin{align} 
e(Q)^{-n-2}\int_{T^{\bD}(\tfrac{12}{10}Q)} \dist^2(X,\spt\|T_{Z_Q *}\bD\|)d\|\tilde{V}\|(X) < \delta'.
\end{align}
If $Q \notin \mathscr{F}_{\text{adj}}$, then using \eqref{proofofclaimcont8} and the fact that $e(Q) \geq 10t$, we have that $T^{\bD}(\tfrac{12}{10}Q) \cap \cD(\tilde{V}) = \emptyset$. Therefore we can apply Theorem \ref{thm:grapdeco} in $T^{\bD}(\tfrac{12}{10}Q)$ and, after making an appropriate choice of $\delta' > 0$, apply Allard's Regularity Theorem in $T^{\bD}(\tfrac{11}{10}Q)$, in order to deduce that $\tilde{V} \res T^{\bD}(\tfrac{11}{10}Q)$ is equal to a union of smooth graphs over the planes of $T_{Z_Q*}\bD$. If, on the other hand $Q \in \mathscr{F}_{\text{adj}}$, then we appeal to the observations of Remark \ref{R:q>0torirema}. Thus one of the following conclusions holds:
\begin{enumerate}[nolistsep,label=$\cG$\arabic*),ref=$\cG$\arabic*)]
\item \label{en:G1} $\tilde{V} \res T^{\bD}(\tfrac{11}{10}Q)$ is a union of smooth graphs over the planes of $T_{Z_Q*}\bD$ (either $Q \notin \mathscr{F}_{\text{adj}}$ or $Q \in \mathscr{F}_{\text{adj}}$ and conclusion E$)_{\text{adj}}$ of Remark \ref{R:q>0torirema} holds). 
\item \label{en:G2} $Q \in \mathscr{F}_{\text{adj}}$, and $\tilde{V} \res T^{\bD}(\tfrac{11}{10}Q)$ is, up to a small set, a two-valued graph over one of the planes of $T_{Z_Q*}\bD$ (conclusion F$)_{\text{adj}}$ of Remark \ref{R:q>0torirema} holds).
\end{enumerate}
Or,
\begin{enumerate}[nolistsep,label=`trouble',font=\scriptsize, ref=(`trouble')]
\item \label{en:trouble} $Q \in \mathscr{F}_{\text{adj}}$ and G$)_{\text{adj}}$ of  Remark \ref{R:q>0torirema} holds.
\end{enumerate}
If \ref{en:G1} or \ref{en:G2} hold, then the arguments of Section \ref{SSS:proofofclaim} can easily be repeated with only very minor modifications and combined with \ref{Q4} to achieve the estimate
\begin{align}
\int_{T^{\bD}(\tfrac{11}{10}Q)} & |(x,0)^{\perp_{T_X\tilde{V}}}|^2 d\|\tilde{V}\|(X) \nonumber \\
-  \int_{T^{\bD}(\tfrac{11}{10}Q)} & g(X)d\|\tilde{V}\|(X) + \int_{T^{\bD}(\tfrac{11}{10}Q)} g(X) d\|T_{Z_Q*}\bD\|(X) \nonumber \\
&\leq  c e(Q)^{-2}\int_{B_{\tilde{c} e(Q)}(Z_Q)} \dist^2(X,\spt\|T_{Z_Q*}\bD\|)d\|\tilde{V}\|(X).
\end{align} 
(notice that the third term on the left-hand side of the above estimate is with respect to $\|T_{Z_Q*}\bD\|$ and not $\|\bD\|$). Then, arguing as per the series of estimates at the end of \textbf{Step 2.} that led to \eqref{proofofclaimcont12}, we bound this term above by
\beq \label{proofofclaimcont14}
c e(Q)^{n-1/4}\int_{T^{\bC}_{1/2,6/16}(0)} \dist^2(X,\spt\|T_{Z_Q*}\bD\|)d\|\tilde{V}\|(X).
\eeq
Again combining this with \eqref{E:CtoC_Z1} and Remark \ref{R:CtoC_Z} we get that
\begin{align} \label{proofofclaimcont15}
\int_{T^{\bD}(\tfrac{11}{10}Q)} & |(x,0)^{\perp_{T_X\tilde{V}}}|^2 d\|\tilde{V}\|(X) \nonumber \\
-  \int_{T^{\bD}(\tfrac{11}{10}Q)} & g(X)d\|\tilde{V}\|(X) + \int_{T^{\bD}(\tfrac{11}{10}Q)} g(X) d\|T_{Z_Q*}\bD\|(X) \nonumber \\
&\leq  c e(Q)^{n-1/4}\tilde{E}^2.
\end{align}
Also, by expressing the planes of $T_{Z_Q*}\bD$ as graphs over the respective planes of $\bD$ and again arguing as per the series of steps leading to \eqref{L228} in Section \ref{SSS:proofofclaim}, we achieve the estimate
\begin{align}
-\int_{T^{\bD}(\tfrac{11}{10}Q)} &g(X)d\|T_{Z_Q*}\bD\|(X) \nonumber \\
& + \int_{T^{\bD}(\tfrac{11}{10}Q)} g(X)d\|\bD\|(X) \leq c e(Q)^d \nu_{\bD,T_{Z_Q*}\bD}^2. \label{proofofclaimcont16}
\end{align}
By \eqref{E:CtoC_Z2} (if $q_{\bD} = 0$) or \eqref{E:CtoC_Z4} (if $q_{\bD} > 0$) and \ref{en:L2Z2} of Corollary \ref{cor:L2Z}, this is at most
\beq \label{proofofclaimcont17}
c e(Q)^d\tilde{E}^2.
\eeq
Thus, by combining \eqref{proofofclaimcont15}, \eqref{proofofclaimcont16} and \eqref{proofofclaimcont17} and noting that $d \leq n$, we achieve the estimate
\begin{align} \label{proofofclaimcont18}
\cI \bigl( \tfrac{11}{10}Q \bigr)  \leq  c\; e(Q)^{d-1/4}\tilde{E}^2.
\end{align} 

\bigskip

\noindent \textbf{Step 4.} \emph{Handling Troublesome Cubes.} It remains to deal with cubes $Q \in \mathscr{F}_{\text{adj}}$ such that \eqref{proofofclaimcont13} and \ref{en:trouble} hold. We call such cubes \emph{troublesome} for $\mathscr{F}$ and we let $\mathscr{F}_{\text{trouble}}$ denote the collection of such cubes. 

Since $\tau_{\xi_Q *}\tilde{V} \res T^{\bD}(\tfrac{12}{10}Q)$ is transverse in the sense of G$)_{\text{adj}}$ of Remark \ref{R:q>0torirema}, we have that (by choice of $\delta'$) the estimates of Corollary \ref{cor:L2Z} hold for $\tau_{\xi_Q *}\tilde{V}$ at the scale of $T^{\bD}(\tfrac{12}{10}Q)$. 

\begin{remarks} Roughly speaking, troublesome cubes are cubes over which the varifold looks no simpler: $\tilde{V} \res T^{\bD}(Q)$ has the same structure as $\tilde{V} \res T^{\bC}_{1/2,6/16}(0)$, \emph{i.e.} the varifold is transverse, passes close to $A(\bD)$ and contains singular points. We will turn this to our advantage by `restarting' the decomposition algorithm inside these cubes.\end{remarks}

Write $\tilde{V}_Q := \tau_{\xi_Q *}\tilde{V}$. If $q_{\bD} = 0$, then set
\beq \label{proofofclaimcont19}
t_Q^2 :=   ce(Q)^{-n}\int_{T^{\bD}(\tfrac{12}{10}Q)}\dist^2(X,\spt\|\bD\|)d\|\tilde{V}_Q\|(X).
\eeq
In this case, for the same reasons as before (except this time at the new scale $T^{\bD}(\tfrac{12}{10}Q)$), we have that
\begin{align} \label{proofofclaimcont20}
\bigl(\underline{\cD}(\tilde{V}_Q) \cap T^{\bD}(\tfrac{11}{10}Q)\bigr)\ \subset\ (A(\bD))_{t_Q}.
\end{align}
Notice also that $t_Q < ce(Q)\sqrt{\delta'}$.  If, on the other hand, $q_{\bD} > 0$, then we suppose for the time being that
\begin{align} \label{proofofclaimcont21}
& E_{\tilde{V}_Q \tiny{\res} T^{\bD}(\tfrac{12}{10}Q)}(\bD) \leq \alpha  \cE_{\tilde{V}_Q \tiny{\res} T^{\bD}(\tfrac{12}{10}Q)}(\bD)
\end{align}
and set
\beq \label{proofofclaimcont22}
t_Q^2 :=   ce(Q)^{-n}\Bigl( \nu_{\bD,\bC^{(0)}}^2 \Bigr)^{-1} \int_{T^{\bD}(\tfrac{12}{10}Q)}\dist^2(X,\spt\|\bD\|)d\|\tilde{V}_Q\|(X).
\eeq
Then we get \eqref{proofofclaimcont20} again, and (by using \eqref{proofofclaimcont21} in the same way that \eqref{E:q>0tori16} was derived in the proof of Lemma \ref{L:q>0tori}) we have that $t_Q <\leq ce(Q)\sqrt{\alpha}$. This shows that we can insist that $t_Q$ is smaller than a fixed, small factor of $e(Q)$ (this is analogous to the issue discussed in Remark \ref{R:tchoice}). So assume now that either $q_{\bD} = 0$ or $q_{\bD}  > 0$ \emph{and} \eqref{proofofclaimcont21} holds.  The final case, in which $q_{\bD} > 0$ and \eqref{proofofclaimcont21}  does not hold, we postpone to \textbf{Step 6.}

Following a similar procedure as was done in \textbf{Step 1.}, set
\[
\mathscr{F}_{\text{initial}}(Q) := \cM_{J_{2,Q}} \cup \bigcup_{j=J_1}^{J_{2,Q} - 1} \{ L \in \cM_j : L \cap \Omega_j \neq \emptyset\}.
\]
where $J_1$ is as it was before and $J_{2,Q} = \max \{j : 2^{-j} \geq 10t_Q\}$. Note that $10t_Q \leq e(L) < s$ for every $L \in \mathscr{F}_{\text{initial}}(Q)$. Let $\mathscr{F}(Q)$ denote the collection of all maximal cubes $L$ in $\mathscr{F}_{\text{initial}}(Q)$ that intersect $\bH_{\bD} \cap T^{\bC}_{1/2,6/16}(0)$ and that satisfy
\begin{align} \label{proofofclaimcont23}
L \not\subset \bigl\{ X \in \bH_{\bD} : r_{\bD}(X) < \tilde{E}^2 \bigr\}.
\end{align}
We perform this construction for every cube $Q \in \mathscr{F}_{\text{trouble}}$.

It is straightforward to see that \ref{Q1} holds with $\mathscr{F}(Q)$ in place of $\mathscr{F}$. Let us explain the property of $\mathscr{F}(Q)$ that is analogous to \ref{Q4}: For each $L \in \mathscr{F}(Q)$, choose $Z_L \in \spt\|\tilde{V}_Q\| \cap T^{\bD}(Q)$ with $\Theta_{\tilde{V}_Q}(Z_L) \geq 2$ satisfying
\beq \label{proofofclaimcont24}
\dist(Z_L,T^{\bD}(L)) = \inf_{\substack{Z \in \spt\|\tilde{V}_Q\| \cap T^{\bD}(Q) \\ :\Theta_{\tilde{V}_Q}(Z) \geq 2}}\dist(Z,T^{\bD}(L)).
\eeq
From \ref{Q1} we again have $d_{\cH}(T^{\bD}(L),T^{\bD}(\{Z_L\})) \leq c e(L)$ for $c = c(n) > 0$ and from \eqref{proofofclaimcont20}, we have $\diam T^{\bD}(\{Z_L\}) \leq c t_Q \leq c e(L)$. Combining these facts gives $T^{\bD}(\tfrac{13}{10}L) \subset B_{\tilde{c} e(L)}(Z_L)$. Once again we have that $\tilde{c} e(L) \leq \tilde{c} s < 1/200$, which ensures that $B_{\tilde{c} \alpha e(L)}(Z_L) \subset T^{\bC}_{1/2,6/16+1/200}(0)$. Now we also get that \ref{Q5} holds for $\mathscr{F}(Q)$ by the same argument as before. Finally, we record that we can easily insist that
\beq \label{proofofclaimcont25}
e(L) \leq \frac{1}{2}e(Q)\ \text{for every}\ L \in \mathscr{F}(Q).
\eeq
\wl

\noindent Suppose first that $L \in \mathscr{F}(Q)$ is such that
\begin{align} \label{proofofclaimcont26}
e(L)^{-n-2}\int_{T^{\bD}(\tfrac{13}{10}L)} \dist^2(X,\spt\|\bD\|)d\|\tau_{\xi_L *}\tilde{V}_Q\|(X) \geq \delta'.
\end{align}
Here we can follow the arguments of \textbf{Step 2.} to get that
\begin{align} \label{proofofclaimcont27}
\cI\bigl( \tfrac{13}{10}L \bigr) \leq c e(L)^{d-1/4} \int_{B_{1/200}(Z_L)} \frac{\dist^2(X,\spt\|T_{Z_L*}\bD\|)}{|X-Z_L|^{n+7/4}}d\|\tilde{V}_Q\|(X).
\end{align}
At this point, notice that Remark \ref{R:CtoC_Z} implies that
\begin{align} \label{proofofclaimcont28}
\int_{T^{\bC}_{1/2,13/32}(0)} \dist^2(X,\spt\|\bD\|)d\|\tilde{V}_Q\|(X) \leq c\tilde{E}^2 < c\delta,
\end{align}
for $c = c(n,k,L,\bC^{(0)}) > 0$, which means that the hypotheses of Corollary \ref{cor:L2Z} are satisfied at unit scale with $\tilde{V}_Q$ in place of $\tilde{V}$. So combining \eqref{proofofclaimcont27} with \ref{en:L23} of Corollary \ref{cor:L2Z} (for $\tilde{V}_Q$) gives us that
\begin{align} \label{proofofclaimcont29}
\cI\bigl( \tfrac{13}{10}L \bigr) \leq c e(L)^{d-1/4} \int_{T^{\bC}_{1/2,13/32}(0)} \dist^2(X,\spt\|\bD\|)d\|\tilde{V}_Q\|(X).
\end{align}
Applying \eqref{proofofclaimcont28} we in particular have that
\begin{align} \label{proofofclaimcont30}
\cI\bigl( \tfrac{11}{10}L \bigr) \leq  c e(L)^{d-1/4} \tilde{E}^2.
\end{align}

Now suppose that $L \in \mathscr{F}(Q)$ is such that
\begin{align} \label{proofofclaimcont31}
e(L)^{-n-2}\int_{T^{\bD}(\tfrac{13}{10}L)} \dist^2(X,\spt\|\bD\|)d\|\tau_{\xi_L *}\tilde{V}_Q\|(X) < \delta'.
\end{align}
In this case, we can similarly mimic the arguments of \textbf{Step 3.} The result is that we achieve the estimate \eqref{proofofclaimcont30}, except in the case where $L$ is troublesome for $\mathscr{F}(Q)$, \emph{i.e.} except when $L \in \mathscr{F}_{\text{adj}}(Q)$ and $L$ satisfies both \eqref{proofofclaimcont31} and \ref{en:trouble}. Let $\mathscr{F}_{\text{trouble}}(Q)$ denote the collection of such cubes: Cubes that are troublesome for $\mathscr{F}(Q)$. Now, with $\mathscr{F}^{(1)} := \mathscr{F}$ and $\mathscr{F}^{(1)}_{\text{trouble}} := \mathscr{F}_{\text{trouble}}$, we inductively define
\begin{align*}
\mathscr{F}^{(j)} &=  \bigcup_{Q \in \mathscr{F}^{(j-1)}_{\text{trouble}}} \{L : L \in \mathscr{F}(Q) \}. \\
\mathscr{F}^{(j)}_{\text{trouble}} &:= \bigcup_{Q \in \mathscr{F}^{(j-1)}_{\text{trouble}}} \{L : L \in \mathscr{F}_{\text{trouble}}(Q)) \}.
\end{align*}
for $j = 2,3,...$. And we claim that there exists some $J \geq 1$ for which $\mathscr{F}^{(J)}_{\text{trouble}} = \emptyset$. To see this: Notice that \eqref{proofofclaimcont25} implies that $
\sup_{L \in \mathscr{F}^{(j)}} e(L) \leq c 2^{-j}.$ So for sufficiently large $J$, any cube $L$ that is adjacent to $A(\bD)$ and that comes from subdividing some $Q \in \mathscr{F}^{(J)}$ will satisfy
\[
L \subset \bigl\{X \in \bH_{\bD} : r_{\bD}(X) < \tilde{E}^2 \bigr\}.
\]
By construction, such cubes will not be included in any collection $\mathscr{F}^{(j')}$ for $j' > J$. In particular, this means that $\mathscr{F}^{(J)}_{\text{trouble}} = \emptyset$. Finally, write
\[
\mathscr{F}_{\text{final}} := \bigcup_{j=1}^J \mathscr{F}^{(j)} \setminus \mathscr{F}^{(j)}_{\text{trouble}}.
\]
Our construction is such that \ref{Q4} and \ref{Q5} hold for $\mathscr{F}_{\text{final}}$.

\bigskip

\noindent \textbf{Step 5.} \emph{Completing the argument}. We can now pick $\eta > 0$ in our application of Corollary \ref{C:q>0fine=coarse} so that 
\[
\bH_{\bD} \cap T^{\bC}_{1/2,5/16}(0) \subset \biggl( \{ r_{\bD} < \tilde{E}^2\} \cup \cO \cup \bigcup_{Q \in \mathscr{F}_{\text{final}}}\tfrac{11}{10}Q  \biggr)
\]
To do this, $\eta$ must depend on $s$, because $s$ determines the region that is covered by $\mathscr{F}_{\text{final}}$. Having done this, we can finally fix the $\delta > 0$ in \eqref{L218} (the definition of the set $\cU$) so that the conclusions of Corollary \ref{C:q>0fine=coarse} hold with the values of $\alpha$ and $\eta$ that we have established here. 

By arguing as in the $(\xi,\zeta) \in \cU_1$ case in Section \ref{SSS:proofofclaim}, we achieve the estimate:
\begin{align} \label{proofofclaimcont32}
\cI\bigl( \cO \bigr) \leq  c\; \int_{T^{\bD}((\cO)_{1/200} \cap \bH_{\bD})} \dist^2(X,\spt\|\bD\|)d\|\tilde{V}\|(X), 
\end{align} 
where $c = c(n,k,L,\bC^{(0)}) > 0$. Elementary geometric considerations tell us that 
\begin{align}
&T^{\bC}_{1/2,5/16}(0) \subset T^{\bD}(\bH_{\bD} \cap T^{\bC}_{1/2,5/16}(0)) \nonumber \\
 \subset &\biggl(\{ r_{\bD} < \tilde{E}^2\} \cup T^{\bD}(\cO) \cup \bigcup_{Q \in \mathscr{F}_{\text{final}}}T^{\bD}(\tfrac{11}{10}Q) \biggr),
\end{align}
and we trivially have
\begin{align} \label{proofofclaimcont33}
\cI \Bigl( \bigcup_{Q \in \mathscr{F}_{\text{final}}} \tfrac{11}{10}Q \Bigr)  \leq \sum_{Q \in \mathscr{F}_{\text{final}}} \cI\bigl(\tfrac{11}{10}Q\bigr).
\end{align}
Suppose that we have the estimate \eqref{proofofclaimcont30} for every cube in $\mathscr{F}_{\text{final}}$. This would mean that the right-hand side of \eqref{proofofclaimcont33} were at most
\begin{align} \label{proofofclaimcont34}
c\biggr(\sum_{Q \in \mathscr{F}_{\text{final}}} e(Q)^{d-1/4} \biggl)\tilde{E}^2 
\end{align}
for $c = c(n,k,L,\bC^{(0)}) > 0$. Now, for each $Q \in \mathscr{F}_{\text{final}}$, property \ref{Q5} implies that $\dist(X,A(\bD))^{-1/4} \geq c e(Q)^{-1/4}$ for all $X \in Q$ whence $e(Q)^{d-1/4} \leq c \int_Q \dist(X,A(\bD))^{-1/4} d\cH^d(X)$. This means that
\begin{align} \label{proofofclaimcont35}
\sum_{Q \in \mathscr{F}_{\text{final}}} e(Q)^{d-1/4} \leq \int_{\bH_{\bD} \cap B_1(0)} \dist(X,A(\bD))^{-1/4} d\cH^d(X) \leq C,
\end{align}
for some absolute constant $C  > 0$. Observing lastly that
\begin{align} \label{proofofclaimcont36}
\|\tilde{V}\| \bigl( \{ r_{\bD} < \tilde{E}^2 \} \bigr) &\leq c\tilde{E}^2,
\end{align}
then \eqref{proofofclaimcont5} follows by putting together \eqref{proofofclaimcont33} - \eqref{proofofclaimcont37}.

\wl

\textbf{Step 6.} \emph{Handling Troublesome Cubes when $q_{\bD} > 0$.} In light of \eqref{proofofclaimcont12}, \eqref{proofofclaimcont18} and \eqref{proofofclaimcont30}, it remains to prove the estimate $\cI(\tfrac{11}{10}Q) \leq ce(Q)^{d-1/4}\tilde{E}^2$ when $q_{\bD} > 0$ and $Q \in \mathscr{F}_{\text{trouble}}$ is such that \eqref{proofofclaimcont21} does not hold. In this case, by using the same process used in \textbf{Step 2.} of the proof of Lemma \ref{L:q>0tori} and in the proof of Corollary \ref{C:q>0fine=coarse}, but carried out in the domain $T^{\bD}(\tfrac{12}{10}Q)$ and with the varifold $\tilde{V}_Q$, we can find a new cone $\bD^Q \in \cC$ so that $A(\bD^Q) \supset A(\bD)$, $\dim A(\bD^Q) \leq \dim A(\bC^{(0)})$, 
\begin{align} \label{proofofclaimcont37}
E_{\tilde{V}_Q \tiny{\res} T^{\bD}(\tfrac{12}{10}Q)}(\bD^Q) \leq c E_{\tilde{V}_Q \tiny{\res} T^{\bD}(\tfrac{12}{10}Q)}(\bD).
\end{align}
and such that either $q_{\bD^Q} = 0$ or
\begin{align} \label{proofofclaimcont38}
& E_{\tilde{V}_Q \tiny{\res} T^{\bD}(\tfrac{12}{10}Q)}(\bD^Q) \leq \alpha  \cE_{\tilde{V}_Q \tiny{\res} T^{\bD}(\tfrac{12}{10}Q)}(\bD^Q).
\end{align}
Cubes $Q$ for which this is necessary are such that at unit scale $\tilde{V}$ resembles $\bD$, but at the scale of $Q$, the varifold $\tilde{V}_Q$ resembles a cone with a larger dimensional axis. Now, if we are able to argue that
\begin{align} \label{proofofclaimcont39}
\cI(\tfrac{12}{10}Q) \leq \int_{T^{\bD}(\tfrac{12}{10}Q)} \dist^2(X,\spt\|\bD^Q\|)d\|\tilde{V}_Q\|(X),
\end{align}
then by using \eqref{proofofclaimcont37}, the non-concentration estimate \ref{en:L24} of Corollary \ref{cor:L2Z}, and \eqref{proofofclaimcont28} we would get the result that 
\[
\cI(\tfrac{12}{10}Q) \leq ce(Q)^{d-1/4}\tilde{E}^2,
\]
which is exactly the estimate required to sum over the cubes as is done at the end of the previous step. To justify \eqref{proofofclaimcont39}, we observe that it is of the same form as \eqref{proofofclaimcont5} except that $q_{\bD^Q} < q_{\bD}$. Therefore we argue inductively with respect to this parameter: Relying on the scale invariance of the arguments we have made thus far, we can repeat \textbf{Steps 1.} to \textbf{4.} with $T^{\bD}(\tfrac{12}{10}Q)$ taking the role of $T^{\bC}_{1/2,7/16}(0)$, $\bD^Q$ taking the role $\bD$ and with $\tilde{V}_Q$ taking the role of $\tilde{V}$ and do so in a such a way that the parameters $\delta'$, $s$, $\eta$ and $\alpha$ can be chosen once and for all depending only on $n$, $k$, $\bC^{(0)}$ and $L$. Moreover the strict inequality $q_{\bD^Q} < q_{\bD}$ ensures that this process terminates; eventually we are either working with a cone $\bD^L$ for which $q_{\bD^L} = 0$ (where $L$ has come from subdividing $Q$) or else every cube that satisfies \eqref{proofofclaimcont13} also satisfies \eqref{proofofclaimcont21}. In these cases, the estimate has already been justified by the work of the other steps. And in this way, we justify \eqref{proofofclaimcont39}. 

\wl

\noindent Notice that despite the complexity of the full construction, there are only ever finitely many cubes. This finishes the proof of \eqref{L214} of Claim \ref{L2claim}, which completes the proof of the induction step for Theorem \ref{L2}.

\subsection{Proof of Corollary \ref{cor:L2Z}} \label{SS:ProofofcorL2Z}

This proof relies on being able to apply Theorem \ref{L2} with $\tau = 1/16$ and with $\bar{V} := \eta_{Z,\rho/2 *}(V \res B_{\rho}(Z))$ in place of $V$. To do this, we need to show that it is possible to choose $\eps$ in such a way that all of the hypotheses of Theorem \ref{L2} are satisfied with $\bar{V}$ in place of $V$. Let $\eps_0$ be the constant the existence of which is asserted by Theorem \ref{L2} when we take $\tau = 1/16$ therein. We claim that the following list of statements can be satisfied:
\begin{enumerate}[nolistsep]
\item $\|\bar{V}\|(B_2^n(0)\times\RR^k) \leq \|\bC^{(0)}\|(B_2^n(0)\times\RR^k) + 1/2$.
\item $0 \in A(\bC) \subset A(\bC^{(0)})$.
\item $\nu_{\bC, \bC^{(0)}} < \eps$.
\item $Q_{\bar{V}}(\bC^{(0)}) < \eps$.
\item $\Theta_{\bar{V}}(0) \geq 2$. 
\end{enumerate}
Firstly observe that 2) and  3) hold trivially because they only concern $\bC$ and $\bC^{(0)}$, which are unchanged and that 5) is immediate because $\Theta_V(Z) \geq 2$. Now, 4) follows by choosing $\eps < \rho^{n+2}\eps_0$ and 1) also follows, using varifold convergence, by choice of $\eps$ sufficiently small depending on $\rho$ and $\sigma$. Thus we may apply Theorem \ref{L2} to $\bar{V}$. This establishes \ref{en:L23} directly. The other conclusions rest on being able to prove \ref{en:L22}.

Writing $\xi = Z^{\perp_{A(\bC)}}$, notice that for fixed $\rho_0 = \rho_0(n,k,\bC^{(0)},L) > 0$, the argument of \cite[Lemma 6.21]{wickrigidity} shows that we can choose $\eps$ sufficiently small so that
\beq
\|V\|(\{X \in B_{\rho_0}(Z) : |\xi^{\perp_{T_{x'}\bC}}| \geq c\nu_{\bC,\bC^{(0)}}|\xi| \}) \geq c\rho_0^n,
\eeq
for some constant $c \in (0,1)$, where $x'$ is the nearest point projection of $X^{\perp_{A(\bC)}}$ onto $\spt\|\bC\|$. This means that
\beq \label{L2Z6}
\nu_{\bC,\bC^{(0)}}^2|\xi|^2 \leq c\rho_0^{-n}\int_{B_{\rho_0}(Z)} |\xi^{\perp_{T_{x'}\bC}}|^2d\|V\|(X),
\eeq
which implies that
\beq \label{L2Z7}
\nu_{\bC,\bC^{(0)}}^2|\xi^{\top_{\bP_i^{(0)}}}|^2 \leq c\rho_0^{-n}\int_{B_{\rho_0}(Z)} |\xi^{\perp_{T_{x'}\bC}}|^2d\|V\|(X).
\eeq
Also, since for any $X \in \spt\|V\| \cap B_{\rho_0}(Z)$ and $i=1,2$, the triangle inequality implies that
\beq
|\xi^{\perp_{\bP_i^{(0)}}}|^2  \leq 2|\xi^{\perp_{\bP_i^{(0)}}} - \xi^{\perp_{T_{x'}\bC}}|^2 + 2| \xi^{\perp_{T_{x'}\bC}}|^2,
\eeq 
we have that
\begin{align*}
|\xi^{\perp_{\bP_i^{(0)}}}|^2 \leq c\nu_{\bC,\bC^{(0)}}^2|\xi|^2 + c\rho_0^{-n}\int_{B_{\rho_0}(Z)} |\xi^{\perp_{T_{x'}\bC}}|^2d\|V\|(X) 
\end{align*}
and so using \eqref{L2Z6} we get that
\beq \label{L2Z8}
|\xi^{\perp_{\bP_i^{(0)}}}|^2 \leq c\rho_0^{-n}\int_{B_{\rho_0}(Z)} |\xi^{\perp_{T_{x'}\bC}}|^2d\|V\|(X).
\eeq
Using \eqref{L2Z7}, \eqref{L2Z8} and the fact that
\beq
|\xi^{\perp_{T_{x'}\bC}}|^2 \leq 2\dist^2(X,\spt\|T_{Z *}\bC\|) + 2\dist^2(X,\spt\|\bC\|),
\eeq
we have
\begin{align}
|\xi^{\perp_{\bP_i^{(0)}}}|^2 + \nu_{\bC,\bC^{(0)}}^2&|\xi^{\top_{\bP_i^{(0)}}}|^2 \nonumber \\
  \leq c\rho_0^{-n}\int_{B_{\rho_0}(Z)}& |\xi^{\perp_{T_{x'}\bC}}|^2d\|V\|(X) \nonumber \\
\leq c\rho_0^{-n}\int_{B_{\rho_0}(Z)} &\dist^2(X,\spt\|T_{Z *}\bC\|)d\|V\|(X) \nonumber \\
 &  + c\rho_0^{-n}\int_{B_{\rho_0}(Z)} \dist^2(X,\spt\|\bC\|)d\|V\|(X). \label{L2Z9}
\end{align}
Then, using \ref{en:L2Z3} of the present Corollary applied to $\eta_{Z,\rho_0 *}V$ (which we have already established), we bound this last expression by
\begin{align}
c\rho_0^{7/4}\int_{B_1(0)} \dist^2(X,&\spt\|T_{Z *}\bC\|)d\|V\|(X) \nonumber \\
& + c\rho_0^{-n}\int_{B_{\rho_0}(Z)} \dist^2(X,\spt\|\bC\|)d\|V\|(X). \label{L2Z10}
\end{align}
By first summing over $i=1,2$ and then using the triangle inequality on the integrand of the first term of \eqref{L2Z10}, followed by \eqref{E:CtoC_Z3} and \eqref{E:CtoC_Z4}, we conclude that
\begin{align}
 \max_{i=1,2}\bigl[ |\xi^{\perp_{\bP_i^{(0)}}}|^2 + \nu_{\bC,\bC^{(0)}}^2|&\xi^{\top_{\bP_i^{(0)}}}|^2\bigr] \nonumber \\
 \leq c\rho_0^{7/4}\int_{B_1(0)} \dist^2(X,&\spt\|\bC\|)d\|V\|(X) \nonumber \\
+ c\rho_0^{7/4}&\max_{i=1,2}\big[|\xi^{\perp_{\bP_i^{(0)}}}|^2 + \nu_{\bC,\bC^{(0)}}^2|\xi^{\top_{\bP_i^{(0)}}}|^2\big] \nonumber \\
 & + c\rho_0^{-n}\int_{B_{\rho_0}(Z)} \dist^2(X,\spt\|\bC\|)d\|V\|(X).
\end{align}
From here was see that by choosing $\rho_0$ sufficiently small depending only on the allowed parameters $n$, $k$, $\bC^{(0)}$ and $L$, we can absorb the middle term to the left-hand side to get that
\beq
|\xi^{\perp_{\bP_i^{(0)}}}|^2 + \nu_{\bC,\bC^{(0)}}^2|\xi^{\top_{\bP_i^{(0)}}}|^2 \leq c\int_{B_1(0)} \dist^2(X,\spt\|\bC\|)d\|V\|(X),
\eeq
for some constant $c = c(n,k,\bC^{(0)},L) > 0$. This completes the proof of (ii) of Corollary \ref{cor:L2Z}, from which (i) follows by taking $\bC = \bC^{(0)}$ and from which one can now deduce (iv) and (v) by using (iii) of Corollary \ref{cor:L2Z} and (v) of Theorem \ref{L2} applied to $\eta_{Z,\rho *}V$ respectively. This completes the proof of Corollary \ref{cor:L2Z} and completes the induction for $q_{\bC}$.

\section{Proofs of Main Results} 

In this chapter we prove the main Excess Improvement Lemma (\ref{lemm:exceimprlemm}) and the main Theorems 1 - 4.

\subsection{Excess Improvement}

Firstly we must prove a lemma that allows us  - when $\bC^{(0)} \in \cP_{n-1}$ - to only work with blow-ups taken relative to sequences of pairs of planes. 

\begin{lemma}\label{blowupoffplan} Fix $\bC^{(0)} \in \cP_{n-1}$, $L > 0$ and $\delta > 0$. There exists $\eps_0 = \eps_0(n,k,\bC^{(0)},\delta) > 0$ and $\eta = \eta(n,k,\bC^{(0)},\delta) > 0$ such that the following is true. Suppose that for some $\eps < \eps_0$, we have that $V \in \cV_L$, $\bC^{(0)}$ and $\bC \in \cC_{n-1}\setminus \cP_{n-1}$ satisfy Hypotheses A and suppose that there exists $Y \in A(\bC^{(0)}) \cap B_1(0)$ for which
\beq \label{blowupoffplan1}
B_{\delta}(Y) \cap \{ X : \Theta_{V}(X) \geq 2\} = \emptyset.
\eeq
This is called a ``$\delta$-gap''. Then
\beq \label{blowupoffplan2}
E_V(\bC) \geq \eta E_V(\tilde{\bC}),
\eeq 
where $\tilde{\bC} \in \cP$ is chosen so that
\beq \label{blowupoffplan3}
E_V(\tilde{\bC}) \leq (3/2) \inf_{\bC^{\prime} \in \cP}E_V(\bC^{\prime}).
\eeq
\end{lemma}

\begin{proof} If the lemma is false then there are sequences of numbers $\{\eps_j\}_{j=1}^{\infty}$, $\{\eta_j\}_{j=1}^{\infty}$ with $\eps_j,\eta_j \downarrow 0^+$, points $Y_j \in A(\bC^{(0)}) \cap B_1(0)$ and $\{\bC^j\}_{j=1}^{\infty} \in \cP$, $\{V^j\}_{j=1}^{\infty} \in \cV_L$ satisfying all of the hypotheses but with $V^j$, $\bC^j$, $\eps_j$, $\eta_j$ and $Y_j$ in place of $V$, $\bC$, $\eps$, $\eta$ and $Y$ respectively and for which
\beq \label{blowupoffplan4}
E_{V^j}(\bC^j) < \eta_jE_{V^j}(\tilde{\bC}^j)
\eeq
for all $j$, where the $\tilde{\bC}^j \in \cP$ are such that
\beq \label{blowupoffplan5}
E_{V^j}(\tilde{\bC}^j) \leq (3/2) \inf_{\bC^{\prime} \in \cP}E_{V^j}(\bC^{\prime}).
\eeq
Begin by passing to a subsequence for which $Y_j \to Y \in A(\bC^{(0)}) \cap \overline{B_1(0)}$ as $j \to \infty$. Using the definition of $\tilde{\bC}^j$, the fact that $\bC^{(0)} \in \cP$ and 3) of Hypotheses A, we see that $\tilde{\bC}^j \to \bC^{(0)}$. Choose a sequence $\Gamma^j$ of rigid motions of $\RR^{n+k}$ for which $\Gamma^j(A(\tilde{\bC}^j)) \subset A(\bC^{(0)})$ and such that
\beq
|\mathrm{id}_{\RR^{n+k}} - \Gamma^j| \leq \frac{3}{2}\inf|\mathrm{id}_{\RR^{n+k}} - \Gamma|,
\eeq
where the infimum is taken over all rigid motions $\Gamma$ for which $\Gamma(A(\tilde{\bC}^j)) \subset A(\bC^{(0)})$. Then, using \eqref{blowupoffplan4} and the triangle inequality, we have that 
\beq \label{blowupoffplan6}
\nu_{\Gamma^j_*\tilde{\bC}^j,\Gamma^j_*\bC^j} \leq cE_{\Gamma^j_*V^j}(\Gamma^j_*\tilde{\bC}^j),
\eeq
for some absolute constant $c > 0$. Let $\tilde{c}^j + c^j$ be the function that represents $\Gamma^j_*\bC^j$ as a graph over $\bC^{(0)}$ in a such a way that $\tilde{c}^j$ represents $\Gamma^j_*\tilde{\bC}^j$ as a graph over $\bC^{(0)}$ (at least away from a small neighbourhood of $A(\bC^{(0)})$). Now we blow up $\Gamma^j_*\bC^j$ off $\bC^{(0)}$ relative to $\Gamma^j_*\tilde{\bC}^j$ using the excess $\tilde{E}_j := E_{\Gamma^j_*V^j}(\Gamma^j_*\tilde{\bC}^j)$, \emph{i.e.} using \eqref{blowupoffplan6}, we deduce that along a subsequence $\tilde{E}_j^{-1}c^j$ converges locally uniformly in $\spt\|\bC^{(0)}\| \cap \{r_{\bC^{(0)}} > 0 \} \cap B_1(0)$ to some function $w$, say. 

Now, if we let $v$ be a blow-up of $\Gamma^j_*V^j$ off $\bC^{(0)}$ relative to $\Gamma^j_*\tilde{\bC}^j$, then dividing \eqref{blowupoffplan4} by $\tilde{E}_j^2$, letting $j \to \infty$ and using the smooth convergence to the blow-up together with the non-concentration estimate \eqref{8j1} shows that $w = v$. Now let us see that $v \not\equiv 0$: From \eqref{blowupoffplan1} we have that $B_{\delta/2}(Y) \cap \cD_v = \emptyset$ and so the $C^2_{loc}$ convergence to the blow-up and the fact that $v=w$ imply that $\gr v$ is a pair of planes. Notice again now that by a pointwise triangle inequality we have that
\begin{align*}
\tilde{E}_j^2 \leq E^2_{\Gamma^j_*V^j}(\Gamma^j_*\bC^j) + c\, \nu^2_{\Gamma^j_*\tilde{\bC}^j,\Gamma^j_*\bC^j},
\end{align*}
where $c$ is a positive absolute constant, from which, using \eqref{blowupoffplan4}, we get that
\beq
0 < c  \leq \tilde{E}_j^{-1}\nu_{\Gamma^j_*\tilde{\bC}^j,\Gamma^j_*\bC^j}.
\eeq
This implies that $v\not\equiv 0$. But now, if we write $\hat{\bC}^j$ for the unique pair of planes containing $\gr(\tilde{c}^j + \tilde{E}_jv)$, we have that 
\[
\tilde{E}_j^{-1}E_{\Gamma^j_*V^j}(\hat{\bC}^j) \to 0
\]
as $j \to \infty$. Thus for sufficiently large $j$, the pair of planes $(\Gamma^j)^{-1}_*\hat{\bC}^j$ contradicts \eqref{blowupoffplan5} and this completes the proof of the Lemma. \end{proof}

We now come to the main lemma. 

\begin{lemma}[Excess Improvement] \label{lemm:exceimprlemm} Let $\bC^{(0)} \in \cC$ and $L > 0$. There exists $\eps_0 = \eps_0(n,k,\bC^{(0)},L) > 0$ such that the following is true. If, for some $\eps < \eps_0$, we have $V \in \cV_L$ and $\bC$, $\bC^{(0)} \in \cC$ satisfying Hypotheses A and $\Theta_V(0) \geq 2$, then there exists $\theta = \theta(n,k,\bC^{(0)},L) > 0$, $c_1 = c_1(n,k,\bC^{(0)}L) \geq 1$, $\bC^{\prime} \in \cC$ and a rotation $\Gamma$ of $\RR^{n+k}$ such that 
\begin{enumerate}[nolistsep]
\item $0 \in A(\bC^{\prime}) \subset A(\bC^{(0)})$,
\item $|\Gamma - \mathrm{id}_{\RR^{n+k}}| \leq c_1\cQ_V(\bC^{(0)})$,
\item $\nu_{\bC^{\prime},\bC^{(0)}} \leq c_1\cQ_{V}(\bC^{(0)})$,
\end{enumerate}
and such that
\begin{align} 
\hspace{0.5cm}\theta^{-n-2}&\int_{B_{\theta}^n(0)\times\RR^k}\dist^2(X,\spt \|\Gamma_*\bC^{\prime}\|) d\|V\|(X) \nonumber \\
&  + \theta^{-n-2}\int_{\Gamma\big((B_{\theta}^n(0)\times\RR^k) \setminus \{r_{\bC^{(0)}} < \theta/8\}\big) }\dist^2(X,\spt \|V\|) d\|\Gamma_*\bC^{\prime}\|(X) \nonumber \\ \nonumber \\
& \hspace{4cm} \leq \frac{1}{2} \cQ^2_V(\bC),  \label{exceimprlemm1}
\end{align}  
 where
 \begin{align*}
\cQ_V&(\bC) := \left(\int_{B_2^n(0)\times\RR^k}\dist^2(X,\spt \|\bC\|) \, d\|V\|(X)\right. \\
&+ \left.\int_{(B_2^n(0)\times\RR^k) \setminus \{r_{\bC^{(0)}} < 1/8\}}\dist^2(X,\spt \|V\|) \, d\|\bC\|(X)\right)^{1/2}
\end{align*} 
Moreover, if $\bC^{(0)} \in \cP_{\leq n-2}$, then $\bC^{\prime} \in \cP_{\leq n-2}$. And if $\bC^{(0)} \in \cC_{n-1}\setminus \cP_{n-1}$, then $\bC^{\prime} \in \cC_{n-1}\setminus \cP_{n-1}$.
\end{lemma}

\begin{proof} We first prove a weakened version of the lemma that is analogous to Lemma 1 of \cite{simoncylindrical}, in which a dichotomy is established. We claim that when all the hypotheses are satisfied for sufficiently small $\eps$, there exists $\delta_0 = \delta_0(n,k,\bC^{(0)},L) > 0$ such that \emph{either} $V$ has a $\delta_0$-gap (in the language of the statement of Lemma \ref{blowupoffplan}) \emph{or} the conclusions of the present lemma hold. So take a sequence $\{\eps_j\}_{j=1}^{\infty}$ of positive numbers with $0 < \eps_j \downarrow 0^+$ as $j \to \infty$ and arbitrary sequences $\bC^j$, $V^j$ satisfying the hypotheses with $V^j$, $\bC^j$ and $\eps_j$ in place of $V$, $\bC$ and $\eps$ respectively. We will prove that the conclusions of this claim hold along a subsequence.

If there is a fixed $\delta_0 > 0$ such that for sufficiently large $j$, $V^j$ has a $\delta_0$ gap, then we are of course done. So assume that this is not the case. This means that given $\{\delta_j\}_{j=1}^{\infty}$ with $\delta_j \downarrow 0^+$, we can pass to a subsequence for which it is the case that $V^j$ has no $\delta_j$-gap. To be precise, this means that $B_{\delta_j}(Y) \cap \cD_j \neq \emptyset$ for every $Y \in A(\bC^{(0)}) \cap B_1(0)$. 
 
Note that $V^j$, $\bC^j$ and $\bC^{(0)}$ satisfy Hypotheses A and let $\cD_j$ be as in \ref{en:dagger2} of Hypotheses $\dagger$. Using the fact that $\cD_j \cap \overline{B_2(0)}$ is closed, together with the sequential compactness of the Hausdorff metric on the space of closed subsets of a compact space, we have that there exists a closed subset $\cD \subset A(\bC^{(0)}) \cap B_2(0)$ such that (along a further subsequence to which we pass without changing notation), we have $d_{\cH}(\cD_j \cap \overline{B_2(0)},\cD \cap \overline{B_2(0)}) \to 0$. Since there are no $\delta_j$ gaps, we have that $\cD \cap B_1(0) = A(\bC^{(0)}) \cap B_1(0)$. So \ref{en:dagger3} is vacuously satisfied. And since $0 \in \cD_j$ for all $j$ by assumption, we have that \ref{en:dagger2} is satisfied. Now pass to a subsequence along which $q_{\bC^j} \equiv q$. Then, since $V^j$, $\bC^j$ and $\bC^{(0)}$ satisfy Hypotheses A, we have that \ref{en:dagger1} is satisfied. 

Now let $v \in \fB(\bC^{(0)})$ denote a blow-up of $V^j$ off $\bC^{(0)}$ relative to $\bC^j$ and let $\psi$ be as in \eqref{blowupregu1} of Theorem \ref{blowupregu} and pass to a subsequence along which we have convergence to $v$. Remark \ref{rema:psiasblowup} shows that $\|v - \psi\|_{L^2(\Omega)}^{-1}(v - \psi)$ is a blow-up of $\tilde{V}^j := R^j_*V^j$ off $\bC^{(0)}$ relative to a new sequence $\{\hat{\bC}^j\}_{j=1}^{\infty} \in \cC$, for some sequence of rotations $R^j$ satisfying $|R^j - \mathrm{id}_{\RR^{n+k}}| \leq c\cQ_{V^j}(\bC^j)$. 

Now note that for sufficiently large $j$ we have $\spt\|V^j\| \cap (B_{\theta}^n(0)\times \RR^k) \subset B_{2L\theta}(0)$. Let $\bar{\theta}_1$ be as in Theorem \ref{blowupregu} and pick $\theta < \bar{\theta}_1/(2\min\{L,1\})$. So, using \eqref{blowupregu1} of Theorem \ref{blowupregu}, the non-concentration estimate \eqref{nonconcnearD1} of Remark \eqref{rema:dagger} and the strong $L^2$ convergence to the blow-up, we have
\begin{align}
\theta^{-n-2}\int_{B^n_{\theta}(0)\times\RR^k} \dist^2(X,&\spt \|(R^j)_*^{-1}\hat{\bC}^j\|)d\|V^j\|(X) \nonumber \\
& \leq  c_2\theta^{2\mu}\int_{B_2^n(0)} \dist^2(X,\spt\|(R^j)_*^{-1}\bC^j\|) d\|V^j\|(X)  \label{exceimprlemm3}
\end{align}
for sufficiently large $j$ and for some $c_2 = c_2(n,k,\bC^{(0)},L) > 0$.

We define $\Gamma^j$ to be the rigid motion of $\RR^{n+k}$ which minimizes $|\Gamma^j - \mathrm{id}_{\RR^{n+k}}|$ subject to the constraint that $\Gamma^j(A(\bC^{(0)})) \supset A((R^j)^{-1}_*\hat{\bC}^j))$. Then we set $\bC^{\prime\; j} = (\Gamma^j)^{-1}_*(R^j)^{-1}_*\hat{\bC}^j$. It is easy to see that by construction, conclusions 1), 2) and 3) are satisfied. Using the fact that $V^j$ is graphical outside of a small neighbourhood of $A(\bC^{(0)})$ to control the second term on the left-hand side of \eqref{exceimprlemm1} and choosing $\theta$ so that $c_2\theta^{2\mu} = 1/2$, we get \eqref{exceimprlemm1}. Conclusion 1) shows that if $\bC^{(0)} \in \cP_{\leq n-2}$ then $\bC^{\prime\; j} \in \cP_{\leq n-2}$. And if $\bC^{(0)} \in \cC_{n-1}\setminus \cP_{n-1}$, then the final conclusion follows from 3), Remark \ref{R:C0inP} and choice of $\eps_0$ sufficiently small depending on $\bC^{(0)}$, $n$ and $k$. This completes the proof of the claim.

Then to establish the full lemma, we return to an arbitrary sequence satisfying the hypotheses. By applying the claim, we either already have the conclusions of the lemma, or else there is a fixed $\delta_0$-gap for every $j$. Now, if we are in the situation where $\bC^{(0)} \in \cP$ and $\bC^j \notin \cP$ for sufficiently large $j$, then we can appeal to Lemma \ref{blowupoffplan} in order to replace each $\bC^j$ by $\tilde{\bC}^j \in \cP$ such that $E_{V^j}(\tilde{\bC}^j) \leq cE_{V^j}(\bC^j)$, whence it suffices to improve the excess relative to this new sequence. Moreover, this replacement process means that \ref{en:dagger3} is now satisfied. Thus we can proceed with the proof of the current lemma as in the proof of the claim and again yield \eqref{exceimprlemm1} and conclusions 1), 2) and 3). This completes the proof of the lemma. 
\end{proof}

\subsection{Proofs of Theorems 1,2, and 3}

We begin by making arguments that are common to the proof of all three theorems: We claim that by iterating Lemma \ref{lemm:exceimprlemm} we can produce a sequence $\{\bC^{(j)}\}_{j=1}^{\infty} \in \cC$ and a sequence $\{\Gamma^j\}_{j=1}^{\infty}$ of rotations of $\RR^{n+k}$ such that
\begin{enumerate}[nolistsep,label = \arabic*),ref= \arabic*)]
\item \label{en:main1} $0 \in A(\bC^{(j)}) \subset \bC^{(0)}$.
\item \label{en:main2} $\nu_{\Gamma^j_*\bC^{(j)}, \Gamma^{j-1}_*\bC^{(j-1)}} \leq c2^{-j}\cQ_V(\bC^{(0)})$. 
\item \label{en:main3} $|\Gamma^j - \mathrm{id}_{\RR^{n+k}}|  \leq c\cQ_V(\bC^{(0)})$.
\item \label{en:main4} $\theta^{-j(n+2)} \int_{B_{\theta^j}^n(0)\times\RR^k} \dist^2(X, \spt\|\Gamma^j_*\bC^{(j)}\|)\, d\|V\|(X)  \leq 2^{-j} \cQ^2_V(\bC^{(0)})$.
\item \label{en:main5} $\theta^{-j(n+2)} \int_{\Gamma^j\big((B_{\theta^j}^n(0)\times \RR^k)\setminus \{ r_{\bC^{(0)}} < \theta^j/8\}\big)} \dist^2(X,\spt\|V\|)\, d\|\Gamma^j_*\bC^{(j)}\|(X)$ \\ $\leq$ $2^{-j}\cQ^2_V(\bC^{(0)})$.
\end{enumerate}
To prove this claim, we construct the sequence inductively: Let $\eps_0$ be as in Lemma \ref{lemm:exceimprlemm}. By choice of $\eps$ in the hypotheses of the present theorems and by applying Lemma \ref{lemm:exceimprlemm} with $\bC^{(0)}$ in place of $\bC$, we produce $\bC^{(1)} \in \cC$ and $\Gamma^1$ which, by the conclusions of Lemma \ref{lemm:exceimprlemm}, show that \ref{en:main1} to \ref{en:main5} hold with $j=1$. Now suppose we have constructed $\{\bC^{(j)}\}_{j=1}^J$ and $\{\Gamma^j\}_{j=1}^J$ satisfying \ref{en:main1} to \ref{en:main5}. By choice of $\eps$, we can insist that $\eps c_1 (1 + \tfrac{1}{2} + (\tfrac{1}{2})^2 + ... ) < \eps_0$. Then note that \ref{en:main1} to \ref{en:main5} imply that the hypotheses of Lemma \ref{lemm:exceimprlemm} are satisfied with $\bC^{(J)}$ in place of $\bC$ and with $\eta_{0,\theta^J\, *}(\Gamma^J)^{-1}_*V$ in place of $V$. Applying the Lemma produces $\bC^{(J+1)}$ and a rigid motion $\Gamma_*^{J+1}$ which satisfy the listed properties and this shows that we indeed have the sequence as claimed.

\wl

\noindent Now observe that by choosing $\eps$ sufficiently small, we can repeat the proof of the claim but starting with $(\eta_{Z,1/8})_*V$ in place of $V$ for any $Z \in \spt \|V\| \cap B_{3/4}(0)$ with $\Theta_V(Z) \geq 2$ (to initially satisfy the hypotheses of Lemma \ref{lemm:exceimprlemm} here we need to use the translation invariance of $\bC^{(0)}$ along its axis and the estimate \ref{en:L2Z1} of Corollary \ref{cor:L2Z}). Then \ref{en:main2} implies that for each such $Z$, the sequence $\{\Gamma_{Z\; *}^j\bC_Z^{(j)}\}_{j=1}^{\infty}$ whose existence is asserted by the claim converges. The result is that there exists some $\bC_Z \in \cC$, a rotation $\Gamma_Z$ and $\alpha = \alpha(n,k,\bC^{(0)},L) \in (0,1)$ for which (writing $V_Z := (\eta_{Z,1/8})_*V$ ) we have
\begin{enumerate}[label = \Roman*),ref= \Roman*)]
\item \label{en:main6} $|\Gamma_Z - \mathrm{id}_{\RR^{n+k} }|  \leq c\cQ_{V_Z}(\bC^{(0)})$   
\item \label{en:main7} $\nu_{\bC_Z,\bC^{(0)}}  \leq c\cQ_{V_Z}(\bC^{(0)})$ 
\item \label{en:main8} $\rho^{-(n+2)} \int_{B_{\rho}^n(0)\times\RR^k} \dist^2(X, \spt\|\Gamma_{Z *}\bC_Z\|)\, d\|V_Z\|(X)  \leq c\rho^{2\alpha} \cQ^2_{V_Z}(\bC^{(0)})$
\item \label{en:main9} $\rho^{-(n+2)} \int_{\Gamma_{Z}\big((B_{\rho}^n(0)\times \RR^k)\setminus \{ r_{\bC^{(0)}} < \rho/8\}\big)} \dist^2(X,\spt\|V_Z\|)\, d\|\Gamma_{Z *}\bC_Z\|(X)$ \\ $\leq c\rho^{2\alpha}\cQ^2_{V_Z}(\bC^{(0)})$ 
for all $\rho \in (0,\theta)$,
\end{enumerate}
where the last two points are proved by using a standard argument to interpolate between the scales $\theta^j$ for $j = 1,2,...$. Observe that \ref{en:main8} implies that $\Gamma_{Z *}\bC_Z \in \cC$ is the unique tangent cone to $V$ at $Z$.

\wl

\noindent Write $\cD(V) = \{Z : \Theta_V(Z) \geq 2\}$. Suppose that $\bC^{(0)}$ is properly aligned and suppose for the sake of contradiction that there exists $Y \in A(\bC^{(0)}) \cap B_{1/4}(0)$ for which $(\RR^{l+k} \times \{Y\}) \cap \cD(V)$ contains more than one point: Pick such a $Y$ and let $Z_1$, $Z_2$ be two distinct points of $(\RR^{l+k} \times \{Y\}) \cap \cD(V)$. But now if we write $\sigma = |Z_1 - Z_2|$, we can violate estimate \ref{en:L2Z1} of Corollary \ref{cor:L2Z} when we apply it with $\eta_{0,16\sigma}\Gamma_{Z_1\; *}^{-1}V_{Z_1}$ in place of $V$ and $(16\sigma)^{-1}\Gamma_{Z_1}^{-1}(8(Z_2-Z_1))$ in place of $Z$. Thus $\cD(V) \cap B_{1/4}(0)$ is graphical over $A(\bC^{(0)})$: There exists a function $\tilde{\varphi} : A(\bC^{(0)}) \cap B_{3/16}(0) \to A(\bC^{(0)})^{\perp}$ for which $\cD(V) \cap B_{1/8}(0) \subset \gr \tilde{\varphi}$. In fact, if, for $Z \in \cD(V) \cap B_{1/4}(0)$, we write $S_Z = Z + \Gamma_Z(A(\bC^{(0)}))$, then using \ref{en:L2Z1} of Corollary \ref{cor:L2Z} in a similar way actually tells us that 
\beq \label{main10}
\cD(V) \cap B_{\rho}(Z) \subset (S_Z)_{c\rho^{1+\alpha}}
\eeq
for every $\rho \in (0,1/8)$, which, in light of \ref{en:main6} above implies that $\tilde{\varphi}$ is Lipschitz.

\wl 

\noindent Now, pick two points $X_1$, $X_2 \in \spt \|V\| \cap (B_{1/64}^n(0)\times \RR^k)$ with $\Theta_V(X_i) \geq 2$ for $i=1,2$ and write $\sigma :=|\pi X_1 - \pi X_2| > 0$. Using \ref{en:main8} above, we have that
\begin{align} 
&  c\sigma^{-n-2}\int_{B_{32\sigma}^n(0)\times \RR^k}  \dist^2(X, \spt \|\bC_{X_2}\|) d\|\Gamma_{X_2\; *}^{-1}V_{X_2}\|(X) \nonumber \\
& \leq  c\sigma^{2\alpha} \cQ^2_V(\bC^{(0)}). \label{main11}
\end{align}
Note that $\tilde{Z} := \Gamma_{X_2}^{-1}(8(X_1 - X_2))$ is a point of density at least two for the varifold $\Gamma_{X_2 *}^{-1}V_{X_2}$. Using the inclusion $B^n_{2\sigma}(\pi X_2) \supset B^n_{\sigma}(\pi X_1)$, followed by \ref{en:L2Z2} of Corollary \ref{cor:L2Z} with $\eta_{0,16\sigma *}\Gamma_{X_2\; *}^{-1}V_{X_2}$ and $\bC_{X_2}$ in place of $V$ and $\bC$ respectively and with $Z = (16\sigma)^{-1}\tilde{Z}$, we have that
\begin{align} 
& \int_{B_2^n(0)\times \RR^k}\dist^2(X,\spt\|\bC_{X_2}\|)d\|\eta_{\tilde{Z},8\sigma *}\Gamma_{X_2 *}^{-1}V_{X_2}\|(X) \nonumber \\
& = c\sigma^{-n-2}\int_{B_{16\sigma}^n(\pi \tilde{Z})\times \RR^k}\dist^2(X,\spt\|T_{\tilde{Z} *}\bC_{X_2}\|)d\|\Gamma_{X_2 *}^{-1}V_{X_2}\|(X) \nonumber \\
& \leq c\sigma^{-n-2}\int_{B_{32\sigma}^n(0)\times \RR^k}\dist^2(X,\spt\|T_{\tilde{Z} *}\bC_{X_2}\|)d\|\Gamma_{X_2 *}^{-1}V_{X_2}\|(X) \nonumber \\
& \leq c \sigma^{-n-2}\int_{B_{32\sigma}^n(0)\times \RR^k}\dist^2(X,\spt\|\bC_{X_2}\|)d\|\Gamma_{X_2 *}^{-1}V_{X_2}\|(X) \nonumber \\
& \leq c\sigma^{2\alpha} \cQ^2_V(\bC^{(0)})\ \ \text{(by \eqref{main11})}. \label{main12} 
\end{align}
This shows that we may apply Lemma \ref{lemm:exceimprlemm} with $V' := \eta_{\tilde{Z},8\sigma *}\Gamma_{X_2 *}^{-1}V_{X_2}$ in place of $V$ and with $\bC_{X_2}$ in place of $\bC$ and perform the same iteration argument that led to \ref{en:main6} - \ref{en:main9} above. The result is that we deduce the existence of some $\bC_{X_1}^{\prime} \in \cC$ and rotation $\Gamma_{X_1}^{\prime}$ for which 
\begin{enumerate}[label = \roman*)',ref= \roman*)']
\item \label{en:main13} $|\Gamma_{X_1}^{\prime} - \mathrm{id}_{\RR^{n+k}}| \leq c\sigma^{\alpha} \cQ_{V'}(\bC^{(0)})$
\item \label{en:main14} $\nu_{\bC_{X_2},\bC_{X_1}^{\prime}} \leq c\sigma^{\alpha} \cQ_{V'}(\bC^{(0)})$.
\item \label{en:main15} $\rho^{-n-2}\int_{B_{\rho}^n(0)\times \RR^k} \dist^2(X, \spt \|\Gamma_{X_1 *}^{\prime}\bC_{X_1}^{\prime}\|) d\|V'\|(X) \\
\leq c\rho^{2\alpha}\sigma^{2\alpha} \cQ^2_{V'}(\bC^{(0)})$ for all $\rho \in (0,\theta)$.
\end{enumerate}
But now \eqref{en:main15} together with \ref{en:main8} (used with $X_1$) implies that 
\beq
\eta_{X_2,1/8\; *}^{-1}\Gamma_{X_2 *}\eta_{Z,8\sigma\; *}^{-1}\Gamma_{X_1 *}^{\prime}\bC_{X_1}^{\prime} = T_{X_1}\Gamma_{X_1 *}\bC_{X_1}.
\eeq
Unravelling this and using \ref{en:main13} tells us that 
\beq
|(\Gamma_{X_2}^{-1} \circ \Gamma_{X_1}) - \mathrm{id}_{\RR^{n+k}}| \leq c\sigma^{\alpha}\cQ_V(\bC^{(0)}),
\eeq
whence (using the fact that $\tilde{\varphi}$ is Lipschitz)
\beq \label{main15}
|\Gamma_{X_1} - \Gamma_{X_2}| \leq c(n,k,L,\bC^{(0)}) |X_1^{\top_{{A(\bC^{(0)})}}} - X_2^{\top_{{A(\bC^{(0)})}}}|^{\alpha} \cQ_V(\bC^{(0)}).
\eeq
From here, one can invoke general Whitney-type extension theorems to deduce that there is a $C^{1,\alpha}$ function $\varphi : A(\bC^{(0)}) \cap B_{1/64}(0) \to A(\bC^{(0)})$, for which $\cD(V) \cap B_{1/128}(0) \subset \gr \varphi$ (the classical Whitney extension will only give $C^1$ regularity of $\varphi$, but \cite[Theorem 4; \textsection 2.3, Chapter VI]{steinsingular} suffices to deduce the existence of a $C^{1,\alpha}$ extension satisfying $\|\varphi\|_{C^{1,\alpha}(A(\bC^{(0)}) \cap B_{1/64}(0))} \leq c\cQ_V(\bC^{(0)})$).

Using 3) of Remark \ref{rema:dagger} and initial choice of $\epsilon$, we may assume that $\sing V \cap B_{1/4}(0) \subset (\cD(V))_{\theta/2}$. If there exists $X_0 \in (\sing V \setminus \cD(V)) \cap B_{1/4}(0)$, then let $Z \in \cD(V)$ be such that $r := |X_0 - Z| = \dist(X_0,\cD(V))$. But now, using \ref{en:main7}, \ref{en:main8} and \ref{en:main9}, we have that $\cQ_{(\eta_{0,4r})_*\Gamma_Z^{-1}V_Z \tiny{\res} B_1(0)}(\bC^{(0)}) \leq c\eps$ and yet the singular point $(4r)^{-1}\Gamma^{-1}_Z(8(X_0 - Z))$ is distance $1/4$ from the nearest point of $\cD((\eta_{0,4r})_*\Gamma_Z^{-1}V_Z)$. For correct choice of $\epsilon$, this would directly violate the observations of 3) of Remark \ref{rema:dagger}. Combining this with \eqref{main10}, we therefore deduce that 
\beq \label{main16}
\cD(V) \cap B_{1/128}(0) = \sing V \cap B_{1/128}(0) \subset \gr \varphi.
\eeq

\bigskip

\noindent \textbf{Specifics of the proof of Theorem 1.} So we have that $V \res  B_{1/128}(0)$ decomposes as two disjoint smooth graphs locally away from $\gr \varphi$. This means that we can write $V \res B_{1/130}(0) = (|\gr \bar{u}_1| + |\gr \bar{u}_2|) \res B_{1/130}(0)$, where for $i=1,2$, we have that $\bar{u}_i \in C^{0,1}(\bP_i^{(0)} \cap B_{1/130}(0),\bP_i^{(0)\perp})$ and $\bar{u}_i$ is smooth and solves the Minimal Surface System on $B_{1/130}(0)\setminus \bp_{\bP_i^{(0)}}(\gr \varphi)$. Now a removability result due to Harvey and Lawson (\cite[Theorem 1.2]{harveylawsonextending}) gives us that $\bar{u}_i$ extends over $\bp_{\bP_i^{(0)}}(\gr \varphi)$ as a weak solution to the minimal surface system, after which Allard Regularity implies that $\bar{u}_i$ is actually smooth in $B_{1/132}(0)$. With $M_i := \gr \bar{u}_i$, the conclusions of Theorem 2 now hold in ball $B_{132}(0)$, but it is clear that our arguments show that it can be made to hold with $B_{1/2}(0)$ in place of $B_{1/132}(0)$.

\bigskip

\noindent \textbf{Specifics of the Proof of Theorem 3.} Notice that \ref{en:main7} and Remark \ref{en:C0inP2} imply that $C_Z \in \cC_{n-1}\setminus \cP_{n-1}$ for all $Z \in \cD(V) \cap B_{3/4}(0)$. Then using \ref{en:main8} and the argument of \emph{1)} of Remark \ref{rema:dagger}, one can deduce that we have equality in \eqref{main16}, \emph{i.e.} $\sing V \cap B_{1/128}(0) = \gr \varphi \cap B_{1/128}(0)$. Now observe that $B^n_{1/128}(0)\setminus \pi (\gr \varphi)$ is the disjoint union of two simply connected components $U_a$ and $U_b$, say, and whence $V \res (U_a \times \RR^k)$ decomposes as $|\gr f_1| + |\gr f_2|$, where $f_i$ is smooth on $U_a$. Now, using a Campanato regularity lemma (\emph{e.g.} \cite[Theorem 4.4]{rafeiromorreycampanato}), we can separately prove $C^{1,\alpha}$ regularity of each $f_i$ up to its boundary $\pi \gr \varphi$. Therefore $V \res (U_a \times \RR^k)$ consists of two separate smooth minimal submanifolds and the same holds for $U_b$ and 1) of Theorem 3 follows directly.

\subsection{Proof of Theorem 4}

By looking at the cross-section of $\bC$, the problem is immediately reduced to that of showing that a two-dimensional Lipschitz minimal two-valued graphical cone $\bC_0$ with trivial spine must be a pair of planes meeting only at the origin. Suppose then that $\bC_0$ is the varifold associated to the graph of the two-valued function $f : \RR^2 \to \RR^k$. We will analyse the link $\Sigma := \spt \|\bC_0\| \cap S^{1+k}$, which defines a one-dimensional stationary varifold in the sphere $S^{1+k}$. If the link does not contain any singularities, then (by the Allard-Almgren classification of one-dimensional stationary varifolds in Riemannian manifolds - \cite{allardalmgrenstructure}) it is the disjoint union of two great circles, in which case $\bC_0$ is a pair of planes meeting only at the origin and we are done. Thus we may assume that $\Sigma$ has at least one singular point. In fact we show that this leads to a contradiction.

\bigskip

\noindent Let $S^1\times\{0\}$ denote the unit circle in the domain, \emph{i.e.} $S^1\times\{0\} := \{(x,0) \in \RR^2\times\RR^k : |x| = 1\}$. For each $(x,0) \in S^1\times\{0\}$, write 
\[
S^k_x := \{Z \in S^{1+k} : Z = (x,y)/|(x,y)|,\ \text{for some}\ y \in \RR^k\}.
\]
This set is an open $k$-dimensional hemisphere. The fact that $\Sigma$ is the link of a two-valued graph implies that for every $(x,0) \in S^1\times\{0\}^k$, we have that $S^k_x \cap \Sigma$ consists of two (possibly coinciding) points. For notational ease we define the following two-valued function: For $(x,0) \in S^1\times\{0\}$, let $\tilde{f}((x,0)) = S^k_x \cap \Sigma$. We also write $\tilde{\bp}$ for the `projection' which sends $S^k_x$ to $(x,0)$.

\bigskip

\noindent Note that every singular point of $\Sigma$ is a multiplicity two point of $V$. The work of \cite{allardalmgrenstructure} gives us a good description of the singularities: For each point $X \in \sing \Sigma$, there is a $\delta$ such that, writing $d_S$ for the distance on the sphere, we have that
\beq \label{tangconeclass1}
\Sigma \cap \{ d_S(\cdot,X) < \delta\} = \bigcup_{i=1}^4 \{\gamma^{X}_i(s) : s \in [0,t_{\delta})\}
\eeq 
where for $i=1,...,4,$ $\gamma^X_i : [0,1] \to S^{1+k}$ are geodesics in the sphere with $\gamma^X_i(0) = X$ and such that
\beq \label{tangconeclass2}
\sum_{i=1}^4 \dot{\gamma}^X_i(0) = 0.
\eeq
Note that here we can actually take $\delta$ be the distance to the nearest singular point, \emph{i.e.}
\beq \label{tangconeclass3}
\delta = \dist(X, \sing \Sigma \setminus \{X\}),
\eeq
where this distance is computed in the sphere metric.
\bigskip

\noindent Now fix a singular point $X_0 \in \sing \Sigma$. Let $X_1$ denote a singular point at distance $\delta$ from $X_0$  and write $X_1 = \gamma^{X_0}_1(t_{\delta})$ (where $\delta$ and $t_{\delta}$ are as in \eqref{tangconeclass1}). Write $x_i := \tilde{\bp} X_i$ for $i=0,1$ and write $S^1\times\{0\} = [-\pi,\pi)\times\{0\}$ in such a way that $x_0 = 0$ and $x_1 > 0$. Consider 
\beq
R:= \tilde{f}(\{ (x,0) : 0 < x < x_1 \} ) \setminus \gamma^{X_0}_1((0,t_{\delta})).
\eeq
Notice that for every  $x \in (0,x_1)$, $R \cap S^k_x$ is a single point. Thus in fact $R = \gamma^{X_0}_j( (0,t') )$ for some $j \in \{2,3,4\}$, where
\beq
t' : = \inf_{t \in [0,1]}\gamma^{X_0}_j(t) \in S^k_{x_1}
\eeq
Assume without loss of generality that $j=2$. Since $X_1$ is a singular point, it is a multiplicity two point. This means that $\Sigma \cap S^k_{x_1}$ is a single point and therefore that $\gamma^{X_0}_2(t') = X_1$. However, observe that the great circles of which $\gamma^{X_0}_i$ for $i=1,2$ are segments can only possibly meet at two antipodal points. Since they meet at $X_0$, we deduce that they do in fact meet at $-X_0$ and therefore that $X_1 = -X_0$. This means $\delta = \diam S^{1+k}$ which implies that $\Sigma$ is the union of four half-great-circles meeting only at the points $X_0$ and $-X_0$. We deduce that $\bC_0$ is four half-planes meeting along a line, which means that $\dim S(\bC) = n-1$. This contradiction shows that $\Sigma$ could not have had any singularities and this completes the proof.


\providecommand{\bysame}{\leavevmode\hbox to3em{\hrulefill}\thinspace}
\providecommand{\MR}{\relax\ifhmode\unskip\space\fi MR }
\providecommand{\MRhref}[2]{%
  \href{http://www.ams.org/mathscinet-getitem?mr=#1}{#2}
}
\providecommand{\href}[2]{#2}

\end{document}